\documentclass[11pt]{article} 
\usepackage{latexsym, amsfonts, amssymb, amsmath, amsthm, mathtools}
\usepackage{graphicx}
\usepackage{multirow}
\usepackage{multicol}
\graphicspath{ {images/} }
\usepackage[dvipsnames]{xcolor}
\usepackage[colorlinks=true,
linkcolor=blue,
filecolor=magenta,
urlcolor=blue,
citecolor=magenta,
pdftitle={Chemotaxis models with signal-dependent sensitivity and a logistic-type source, II: Persistence and stabilization},
pdfauthor={Le Chen, Ian Ruau, Wenxian Shen},
pdfkeywords={persistence, stabilization, Keller--Segel systems, parabolic--elliptic chemotaxis models, signal-dependent sensitivity, signal desensitization, logistic-type source},
pdfpagemode=FullScreen,
breaklinks=true,
]{hyperref}
\usepackage[left=0.9in, right=0.9in, top=0.9in, bottom=0.9in]{geometry}
\usepackage[dvipsnames,svgnames]{xcolor}
\usepackage{colortbl}
\usepackage[normalem]{ulem}
\usepackage{caption}

\usepackage{pgfplots}
\pgfplotsset{compat=1.18}
\usepackage{tikz}
\usetikzlibrary{shapes.geometric,positioning,arrows,intersections,patterns}

\newif\ifrefaudit
\refauditfalse
\ifrefaudit
\usepackage[nocites]{refcheck}
\fi

\definecolor{TableRowGray}{gray}{0.9}

\newcommand{\Norm}[1]{\left\| #1 \right\|}

\usepackage{amsrefs}

\renewcommand{\MR}[1]{~\href{https://mathscinet.ams.org/mathscinet-getitem?mr=MR#1}{MR#1}.}

\BibSpec{article}{%
+{}  {\PrintAuthors}                {author}
+{,} { \textit}                     {title}
+{.} { }                            {part}
+{:} { \textit}                     {subtitle}
+{,} { \PrintContributions}         {contribution}
+{.} { \PrintPartials}              {partial}
+{,} { }                            {journal}
+{}  { \textbf}                     {volume}
+{}  { \PrintDatePV}                {date}
+{,} { \issuetext}                  {number}
+{,} { \eprintpages}                {pages}
+{,} { }                            {status}
+{,} { \url}                        {url}
+{,} { \PrintDOI}                   {doi}
+{,} { available at \eprint}        {eprint}
+{}  { \parenthesize}               {language}
+{}  { \PrintTranslation}           {translation}
+{;} { \PrintReprint}               {reprint}
+{.} { }                            {note}
+{.} {}                             {transition}
+{}  {\SentenceSpace \PrintReviews} {review}
}

\newtheorem{theorem}{Theorem}[section] 

\newtheorem{definition}{Definition}[section]

\newtheorem{lemma}{Lemma}[section]

\theoremstyle{remark}
\newtheorem{remark}{Remark}[section]
\theoremstyle{remark}

\numberwithin{equation}{section}
\numberwithin{figure}{section}
\numberwithin{table}{section}

\def\p{\partial}

\def\R{\mathbb R}

\def\exp{\mathop{\hbox{exp}}}

\newcommand\abs[1]{\ensuremath{\lvert#1\rvert}}

\begin{document}

\title{Chemotaxis models with signal-dependent sensitivity and a logistic-type
source, II: Persistence and stabilization\thanks{This version focuses on the two canonical coefficient choices $\mu=\nu=1$ and either $a=b=1$ or $a=b=0$, while retaining general nonlinear exponents.
It also includes an improved pointwise gradient estimate and corresponding improvements to the global-stability criteria.
The general coefficient setting is treated in \href{https://arxiv.org/abs/2604.02599v1}{version~1 of this preprint (arXiv:2604.02599v1)}.}}
\author{
\begin{tabular}{ccc}
  Le Chen                  & Ian Ruau                 & Wenxian Shen             \\
  \url{lzc0090@auburn.edu} & \url{iir0001@auburn.edu} & \url{wenxish@auburn.edu}
\end{tabular} \bigskip \\
Department of Mathematics and Statistics\\
Auburn University, Auburn, AL 36849, USA}

\date{}
\maketitle

\begin{abstract}

  This paper is Part~II of a series of papers on boundedness, global existence,
  and asymptotic behavior of positive classical solutions to the following chemotaxis model~\eqref{E:main-abstract-eq}
  \begin{equation}\label{E:main-abstract-eq}\tag{CM}
    \begin{dcases}
      u_t=\Delta u- \chi_0\nabla \cdot\left(\frac{u^m}{(1+v)^\beta}\nabla v\right)+au-b u^{1+\alpha}, & x \in \Omega,         \\
      0=\Delta v-\mu v+\nu u^\gamma,                                                                  & x \in \Omega,         \\
      \frac{\partial u}{\partial n}=\frac{\partial v}{\partial n}=0,                                  & x \in \partial\Omega.
    \end{dcases}
  \end{equation}
  Here, $\Omega$ is a bounded smooth domain in $\mathbb{R}^N$. The parameters
  $\alpha,\gamma,m,\mu,\nu$ are positive, $\chi_0\in\mathbb{R}$, and $a,b,\beta$
  are nonnegative. In Part~I~\cite{chen.ruau.ea:26:chemotaxis*1}, we established
  several {biologically relevant} parameter regimes ensuring boundedness and global existence of
  positive classical solutions.

  In this part, we study persistence and
  stabilization of globally defined bounded positive solutions. In particular,
  we quantify how the signal-dependence exponent $\beta$ influences long-time dynamics. We restrict attention to the two canonical coefficient choices $\mu=\nu=1$ and either $a=b=1$ (a canonical
  model with a logistic source) or $a=b=0$ (a canonical minimal model).
  We show that, when $m\ge1$,
  every globally defined bounded positive solution stays uniformly away from
  zero in space.
  {We determine exact $\beta$-dependent critical sensitivity thresholds for local stability of constant solutions.
  We also establish $\beta$-dependent thresholds for global stability of constant solutions. The critical local stability thresholds and the effective global-stability thresholds go to $\infty$ as $\beta\to\infty$.}

  It is mathematically and biologically important to understand how
  the nonlinear sensitivity function
  $\chi(v):=\frac{\chi_0}{(1+v)^\beta}$, particularly the exponent
  $\beta$, influences the asymptotic dynamics of (CM).
  {The factor $(1+v)^{-\beta}$ describes signal-dependent desensitization in the sense that
  large $v$ reduces the drift toward the signal when $\chi_0>0$.
  One therefore expects large $\beta$ to promote stabilization. However,
  the nonlinear  signal-dependent sensitivity actually complicates the analysis of (CM) a lot.}
  The paper makes two main contributions. First,
  we obtain $\beta$-dependent thresholds for the local as well as global stability of constant solutions of (CM) with $\beta>0$ and general exponents $m,\alpha,\gamma$.
  It is  seen from these thresholds that  signal
  saturation (large $\beta$) or repulsion ($\chi_0<0$) can prevent aggregation
  and promote relaxation to spatially homogeneous states.
  Second,  the paper develops several new techniques/ideas to study the stability of constant solutions of (CM) with $\beta>0$ and the exponents $m,\alpha,\gamma$ not all $1$, including  the development of a Lyapunov function, pointwise estimates for $\nabla v$, and a nontrivial extension of the
  rectangle/ODE method from $\beta=0$ to $\beta>0$.
\end{abstract}


\medskip


\tableofcontents
\medskip

\section{Introduction}

This paper is the second of a series of papers on the boundedness, global
existence,  and asymptotic behavior of classical solutions to chemotaxis models of
the following form on a bounded, connected, and smooth domain $\Omega \subset
\mathbb{R}^N$:
\begin{align}\label{E:main-PE0}
  \begin{dcases}
    u_t = \Delta u - \nabla \cdot (u^m \chi(v) \nabla v) + au - b u^{1+\alpha}, & x \in \Omega, \\
    0= \Delta v - \mu v + \nu u^\gamma,                                         & x \in \Omega
  \end{dcases}
\end{align}
supplemented with the Neumann boundary conditions on $\partial\Omega$:
\begin{equation}
  \label{E:main-BC0}
  \frac{\partial u}{\partial n} = \frac{\partial v}{\partial n} = 0,
  \qquad x \in \partial\Omega.
\end{equation}
Here, the function $\chi(v)$, known in the literature as the \textit{chemotactic
sensitivity function}, is given by
\begin{equation}\label{E:chi-function-eq}
  \chi(v) = \frac{\chi_0}{(1+ v)^\beta}.
\end{equation}
In equations~\eqref{E:main-PE0} and~\eqref{E:chi-function-eq}, $m, \alpha, \mu,
\nu$, and $\gamma$ are positive constants, $\chi_0$ is a real number, and $a,b$,
and $\beta$ are nonnegative constants.
Unless otherwise stated, all PDEs in this paper are considered on $\Omega$
with homogeneous Neumann boundary conditions on $\partial\Omega$.
With this agreement, we may not write out the Neumann boundary conditions
explicitly in such equations, and when we say
system~\eqref{E:main-PE0}, it always
means~\eqref{E:main-PE0}+\eqref{E:main-BC0}.

The model ~\eqref{E:main-PE0} describes the evolution of a population that responds to a
chemical signal produced by the population itself.
Here $u(t,x)$ is the population density, and $v(t,x)$ is the signal density.
When $\beta = 0$, $\chi(v)$ reduces to the constant sensitivity
$\chi(v) \equiv \chi_0$.
When $\beta>0$, the sensitivity $\chi(v)$ accounts for the partial loss of the signal substance
through its binding (see~\cite{viglialoro.woolley:18:boundedness,
baghaei:24:global, black.lankeit.ea:20:stabilization, winkler:10:absence}).
The signal is \textit{chemoattractant} when $\chi_0>0$ and
\textit{chemorepellent} when $\chi_0<0$.
The reaction $au-bu^{1+\alpha}$ in
\eqref{E:main-PE0}  with $a,b>0$ is a \textit{logistic-type source}, whereas
\eqref{E:main-PE0} with $a=b=0$ has no reaction.
In the literature, a chemotaxis system is
also referred to as a \textit{Keller--Segel system}, named after the pioneering
works of E. F. Keller and L. A. Segel~\cite{keller.segel:70:initiation,
keller.segel:71:model, keller.segel:71:traveling}.

Basic questions for~\eqref{E:main-PE0} concern global existence and the
asymptotic behavior of globally defined bounded solutions of~\eqref{E:main-PE0}.
Numerous studies have addressed these issues for various special cases of
\eqref{E:main-PE0} and their parabolic-parabolic counterparts. For example, the
works~\cite{herrero.medina.ea:97:finite-time, herrero.velazquez:96:singularity,
nagai:01:blowup, nagai.senba:98:global, tello.winkler:07:chemotaxis} studied
finite-time blow-up and global existence for constant sensitivity when
$m=\alpha=\gamma=1$, in both parabolic--elliptic and parabolic--parabolic
formulations and with or without logistic reaction. Constant-sensitivity
parabolic--elliptic systems with general positive exponents $m$, $\alpha$, and
$\gamma$ were considered in~\cite{galakhov.salieva.ea:16:on,
hong.tian.ea:20:attraction-repulsion, hu.tao:17:boundedness, zheng:24:on,
xiang:19:dynamics}. The
works~\cite{biler:99:global, black:20:global, ding.wang.ea:19:global,
fujie:15:boundedness, fujie.senba:16:global, fujie.winkler.ea:14:blow-up,
kurt.shen:23:chemotaxis, kurt.shen.ea:24:stability,
nagai.senba:98:global, winkler:11:global, zhao.zheng:18:global} investigated
finite-time blow-up and global existence for the singular sensitivity
$\chi_0/v$ when $m=\alpha=\gamma=1$, again in both temporal formulations.
For the bounded signal-dependent sensitivity $\chi_0(1+v)^{-\beta}$ with
$m=\alpha=\gamma=1$, global existence was studied in
\cite{black.lankeit.ea:20:stabilization,
mizukami.yokota:17:unified, viglialoro.woolley:18:boundedness,
winkler:10:absence}. In particular,
\cite{black.lankeit.ea:20:stabilization, mizukami.yokota:17:unified,
winkler:10:absence} prove global existence in the minimal
parabolic--parabolic case when $\beta>1$.
Part~I~\cite{chen.ruau.ea:26:chemotaxis*1} gives a more detailed review of these
constant, singular, and bounded signal-dependent sensitivity models.

There is also a large literature on the asymptotic behavior of globally defined
solutions to chemotaxis systems. For constant sensitivity, stability of the
positive logistic equilibrium was studied in the parabolic--elliptic and
parabolic--parabolic settings in~\cite{tello.winkler:07:chemotaxis,
winkler:14:global}, while nonlinear-mobility and nonlinear-production variants
were treated in~\cite{galakhov.salieva.ea:16:on, xiang:19:dynamics}. The works
\cite{ahn.kang.ea:19:eventual,  kurt:26:large-time, le.kurt:26:persistence, kurt.shen.ea:24:stability, li.li:21:large,
winkler.yokota:18:stabilization, zheng.mu.ea:18:global} studied the stability of
positive constant solutions for singular sensitivity. The work
\cite{black.lankeit.ea:20:stabilization} considered stabilization in a minimal
parabolic--parabolic model with bounded signal-dependent sensitivity.


Fewer studies have treated bounded signal-dependent sensitivity in the
parabolic--elliptic setting.
Part~I established global existence and boundedness of solutions of \eqref{E:main-PE0}.
Here we study the asymptotic behavior of \eqref{E:main-PE0} and
quantify the effects of $\beta$ and $\chi_0$.
We restrict attention to the two canonical coefficient choices $\mu=\nu=1$ and $a=b=1$ or $a=b=0$, that is, the following two canonical parabolic--elliptic models:
\begin{subequations}\label{E:main-PE}
  \begin{equation}\label{E:main-PE-L}
    \begin{dcases}
      u_t=\Delta u-\chi_0\nabla\cdot\left(\frac{u^m}{(1+v)^\beta}\nabla v\right)
      +u-u^{1+\alpha}, & x\in\Omega, \\
      0=\Delta v-v+u^\gamma, & x\in\Omega,
    \end{dcases}
  \end{equation}
  \begin{equation}\label{E:main-PE-M}
    \begin{dcases}
      u_t=\Delta u-\chi_0\nabla\cdot\left(\frac{u^m}{(1+v)^\beta}\nabla v\right),
      & x\in\Omega, \\
      0=\Delta v-v+u^\gamma, & x\in\Omega.
    \end{dcases}
  \end{equation}
\end{subequations}

We point out that the factor $(1+v)^{-\beta}$ describes signal-dependent desensitization:
large $v$ reduces the drift toward the signal when $\chi_0>0$.
One therefore expects large $\beta$ to promote stabilization, but few studies have examined how $\beta$ influences the stabilization of solutions.
The nonlinear  signal-dependent sensitivity function  $\chi(v)=\frac{\chi_0}{(1+v)^\beta}$ ($\beta>0$) actually complicates the analysis of \eqref{E:main-PE} even
though it is bounded above by a constant. More specifically, put  $f_\beta(s)\coloneqq(1+s)^{-\beta}$. One has
\begin{equation}\label{E:expanded-flux-intro}
  \begin{aligned}
    -\chi_0\nabla\cdot\left(u^m f_\beta(v)\nabla v\right)
    ={}&-\chi_0m u^{m-1}f_\beta(v)\nabla u\cdot\nabla v
    -\chi_0u^m f_\beta(v)\Delta v\\
    &+\chi_0\beta u^m(1+v)^{-\beta-1}\left|\nabla v\right|^2.
  \end{aligned}
\end{equation}
Here $\Delta v=v-u^\gamma$ in both canonical models. The last term vanishes when $\beta=0$ but must be controlled in the nonlinear
dynamics when $\beta>0$, which makes the analysis more complicated.

Our contributions include the development of $\beta$-dependent stability
criteria for constant solutions of the two canonical models with
signal-dependent sensitivity, nonlinear mobility, and production,
together with new techniques for establishing these criteria.
We will state the main results proved in this paper and the techniques/methods developed in this paper in Section 2.
In the rest of the introduction, we give a brief review of results in Part I.

In Part~I~\cite{chen.ruau.ea:26:chemotaxis*1} of this series, we studied the
boundedness and global existence of positive classical solutions  of ~\eqref{E:main-PE0}, including the canonical  models  \eqref{E:main-PE}, emphasizing the roles of the
nonlinear cross-diffusion rate and the logistic-type source.

We call a function $(u(t,x),v(t,x))$ a \textit{classical solution} of the
parabolic-elliptic system~\eqref{E:main-PE0} on the time interval $(0,T)$ if
\begin{equation*}
  u(\cdot,\cdot) \in C^{1,2}\left( (0,T) \times \Omega\right)\cap C^{0,1}\left(
  (0,T) \times \overline{\Omega}\right),\quad
  v(\cdot,\cdot) \in C^{0,2}\left( (0,T) \times \Omega\right)\cap C^{0,1}\left(
  (0,T) \times \overline{\Omega}\right),
\end{equation*}
and $(u(t,x),v(t,x))$ satisfies~\eqref{E:main-PE0} for all $t\in (0,T)$. A given
function $(u(t,x), v(t,x))$ is said to be a {\it global classical solution}
of~\eqref{E:main-PE0} if it is a classical solution of~\eqref{E:main-PE0} on $(0,
\infty)$. A global classical solution is said to be {\it positive} (resp. {\it
bounded}) if $\inf_{x\in\Omega}u(t,x) > 0$ for any $t \in (0, \infty)$
(resp.\ $\limsup_{t\to\infty} \Norm{u(t,\cdot)}_\infty < \infty$).

In the following, we recall some results of~\cite{chen.ruau.ea:26:chemotaxis*1}
for the two canonical models, which will be utilized in this paper.
Assume throughout that
\begin{equation}\label{E:initial-cond-PE}
  u_0\in C(\overline{\Omega}),\qquad \inf_{x\in\Omega}u_0(x)>0.
\end{equation}
Part~I~\cite[Proposition~1.1]{chen.ruau.ea:26:chemotaxis*1}, specialized to
either canonical model, provides a maximal time $T_{\max}(u_0)\in(0,\infty]$
and a unique positive classical solution $(u(t,x;u_0),v(t,x;u_0))$  with $u(0,x;u_0)=u_0(x)$  on
$(0,T_{\max}(u_0))$ such that
\[
  \lim_{t\to0_+}\Norm{u(t,\cdot;u_0)-u_0}_{\infty}=0,
\]
with $u\in C^{1,2}((0,T_{\max})\times\overline\Omega)$ and
$v\in C^{0,2}((0,T_{\max})\times\overline\Omega)$.
We use this notation for the maximal solution throughout.

We retain only the following two constants from Part~I:
\begin{equation}\label{E:Theta-chi-beta}
  \Theta_\eta \coloneqq \begin{cases} \eta^\eta(1+\eta)^{-(1+\eta)}\quad & {\rm for}\,\, \eta>0\cr
    1\quad & {\rm for}\,\, \eta=0,
  \end{cases} \qquad{\rm and}\quad
  \chi_\beta   \coloneqq \frac{2(2\beta-1)}{\max\{2,\gamma N\}}
  \quad  {\rm for}\,\,\, \beta\ge1.
\end{equation}
Part~I~\cite[Theorem~1.2(2)]{chen.ruau.ea:26:chemotaxis*1} implies that in
either canonical model, if $m=1$, $\beta\ge1$, and $\chi_0<\chi_\beta$, then
every datum satisfying~\eqref{E:initial-cond-PE} generates a global bounded
classical solution. For the canonical non-minimal model with $m\ge1$,
Part~I~\cite[Theorem~1.3(1)(i) and (2)]{chen.ruau.ea:26:chemotaxis*1} gives
the same conclusion whenever $\beta\ge0$ and $\alpha>m+\gamma-1$, without an
upper restriction on $\chi_0$.

We point out that, by the arguments of
\cite[Theorem~1.2(2)]{chen.ruau.ea:26:chemotaxis*1}, in the canonical minimal
model, when $m=1$, $\beta\ge 1$, and $0<\,\chi_0 <
\min\big\{\frac{\chi_\beta}{2},\chi_\beta^{1/2}\big\}$,
there exist constants $C_N>0$ and
$0<q'_N\le 1\le q''_N$, depending only on $(\Omega,N,\gamma)$, such that for every
$u^*>0$ and any globally defined classical solution $(u(t,x),v(t,x))$ of \eqref{E:main-PE} with
$u(0,\cdot)$ satisfying~\eqref{E:initial-cond-PE} and $\int_\Omega
u(0,x)\,dx=u^*|\Omega|$,
\begin{equation}
  \label{E:CN-eq}
  \sup_{x\in\Omega} u(t,x)\le \overline{u}_0(u^*)
  \coloneqq C_N\big((u^*)^{q'_N}+(u^*)^{q''_N}\big)\quad \forall\, t\gg 1.
\end{equation}
We will use this bound to study the asymptotic behavior of the minimal model.
Since it is not stated explicitly in~\cite{chen.ruau.ea:26:chemotaxis*1},
we give a proof in Lemma~\ref{L:minimal-eventual-bounds}.

\smallskip

\noindent {\bf Organization of the paper:} In Section~\ref{S:main-results},
we state the main results and discuss the thresholds and proof methods.
Section~\ref{S:Persistence} contains the proof of Theorem~\ref{T:persistence}.
In Section~\ref{S:linear-stability}, we prove Theorem~\ref{T:linear-stability-thm}.
In Section~\ref{S:Stability}, we prove Theorem~\ref{T:global-stable}
for the non-minimal model.
Section~\ref{S:stability-minimal-model} contains the proof of
Theorem~\ref{T:stability-minimal-model} for the minimal model.
In Appendix~\ref{S:Prelim}, we collect the regularity, compactness, and
eventual bounds for classical solutions of~\eqref{E:main-PE} used in the
main proofs.

\section{Statements of the main results and remarks}\label{S:main-results}

In this section, we state the main results on persistence and stability of
positive constant solutions.
We first introduce the definitions and notation,  next state the main results, and
then   discuss the
thresholds and proof methods.

\subsection{Definitions and notation}\label{SS:Definition}

Throughout, $\Omega\subset\mathbb R^N$ is a bounded, connected, smooth domain, and $C$ denotes a generic solution-independent constant that may depend on fixed
parameters and change from line to line. We write $\Norm{\cdot}_p$ and
$\Norm{\cdot}_\infty$ for the standard $L^p(\Omega)$ and uniform norms.
For $u_0$ satisfying~\eqref{E:initial-cond-PE}, we write
$(u(t,x;u_0),v(t,x;u_0))$ and $T_{\max}(u_0)$ for the corresponding maximal
solution with initial condition $u(0,x;u_0)=u_0(x)$  and suppress arguments when no confusion can arise.

In the canonical non-minimal model, \eqref{E:main-PE} has the unique positive
constant solution
\begin{equation}\label{E:equilibrium}
  (u^*,v^*)=(1,1).
\end{equation}
In the canonical minimal model, it has the one-parameter family
$\{(u^*,(u^*)^\gamma)\}_{u^*>0}$. For a solution in this model, $u^*$ is fixed
by the conserved mean of $u$.

Let $0=\lambda_0<\lambda_1\le\lambda_2\le\cdots$ denote the eigenvalues of
the negative Neumann Laplacian, repeated according to multiplicity.

\begin{definition}\label{D:stability}
  An equilibrium { of  \eqref{E:main-PE}}  is \emph{linearly stable} if the spectrum of its linearized
  evolution operator lies strictly in the open left half-plane, and is
  \emph{linearly unstable} if that operator has an eigenvalue with positive
  real part. For the minimal model, these notions are understood on the invariant mass-preserving, or mean-zero, perturbation space.

  The equilibrium $(1,1)$ of the canonical non-minimal model is
  \emph{globally asymptotically stable} if every globally defined, bounded,
  positive classical solution satisfies
  $\Norm{u(t,\cdot)-1}_\infty\to0$ as $t\to\infty$.
  An equilibrium $(u^*,(u^*)^\gamma)$ of the canonical minimal model is
  globally asymptotically stable if the analogous convergence
  $\Norm{u(t,\cdot)-u^*}_\infty\to0$ holds for every such solution with
  $\int_\Omega u(0,x)\,dx=|\Omega|u^*$.
\end{definition}

This definition applies to solutions already known to exist globally and
remain bounded.
It does not assert global existence or boundedness for every initial datum.

\subsection{Statements of the main results}\label{SS:Main}

Our first result concerns persistence of globally defined bounded positive
solutions of system~\eqref{E:main-PE}.

\begin{theorem}[Persistence uniformly in space]\label{T:persistence}
  \begin{itemize}
    \item[(1)]
    Assume $m\ge 1$ and consider either the canonical non-minimal
    model or the canonical minimal model. Then persistence uniformly in space
    occurs in the following sense: for
    any globally defined, bounded, and positive solution $(u(t,x),v(t,x))$
    of~\eqref{E:main-PE}, we have
    \begin{align}\label{E:persistence}
      \liminf_{t \to \infty} \inf_{x \in \Omega} u(t,x) > 0
      \quad \text{and} \quad
      \liminf_{t\to\infty} \inf_{x\in\Omega} v(t,x)
      \;\ge\;
      \Big(\liminf_{t\to\infty}\inf_{x\in\Omega} u(t,x)\Big)^{\!\gamma}
      \,>\, 0.
    \end{align}

    \item[(2)] In the canonical non-minimal model, assume that
    $m\ge 1$,  $\beta\ge 1$, and $ 0<\chi_0<1/\Theta_{\beta-1}$. Then every globally
    defined, bounded, and positive solution $(u(t,x),v(t,x))$
    of~\eqref{E:main-PE} satisfies
    \begin{equation}\label{E:new-eventual-lower-bound-for-u}
      \liminf_{t\to\infty}\inf_{x\in\Omega} u(t,x)
      \ge \left(1-\chi_0\Theta_{\beta-1}\right)^{\frac{1}{\alpha}}
    \end{equation}
    and
    \begin{equation}\label{E:new-eventual-lower-bound-for-v}
      \liminf_{t\to\infty}\inf_{x\in\Omega} v(t,x)
      \ge \left(1-\chi_0\Theta_{\beta-1}\right)^{\frac{\gamma}{\alpha}}.
    \end{equation}

    \item[(3)] In the canonical minimal model, assume that $m=1$, $\beta\ge 1$,
    and ${0<\chi_0 <
    \min\big\{\frac{\chi_\beta}{2},\chi_\beta^{1/2}\big\}}$.   For any $u^*>0$, every globally defined, bounded, and positive solution $(u(t,x),v(t,x))$ of~\eqref{E:main-PE} with $\int_\Omega u(0,x)\,dx=u^*|\Omega|$ satisfies
    \begin{equation}\label{E:new-lower-bound-for-v}
      \liminf_{t\to\infty}\inf_{x\in\Omega} v(t,x)
      \ge \underline v_0(u^*)
      \coloneqq C_\Omega
      \begin{cases}
        u^*\left(\overline{u}_0(u^*)\right)^{\gamma-1}, & 0<\gamma\le 1, \\[4pt]
        (u^*)^\gamma,                                   & \gamma>1,
      \end{cases}
    \end{equation}
    where $C_\Omega>0$ depends only on $(\Omega,N)$ and
    $\overline{u}_0(u^*)$ is defined
    in~\eqref{E:CN-eq}.
  \end{itemize}
\end{theorem}

Theorem~\ref{T:persistence} is proved in Section~\ref{S:Persistence}.

Our second result concerns linear stability and instability of positive
constant solutions of the two canonical models. Define
\begin{subequations}\label{E:chi-star}
  \begin{equation}\label{E:chi-star-L}
    \chi_{\mathrm L,\beta}^*
    \coloneqq \frac{2^\beta}{\gamma}
    \min_{n\ge1}\frac{(\lambda_n+\alpha)(1+\lambda_n)}{\lambda_n},
  \end{equation}
  and, for $u^*>0$,
  \begin{equation}\label{E:chi-star-M}
    \chi_{\mathrm M,\beta}^*(u^*)
    \coloneqq
    \frac{(1+(u^*)^\gamma)^\beta}{\gamma(u^*)^{m+\gamma-1}}
    (1+\lambda_1).
  \end{equation}
\end{subequations}
Here $\mathrm L$ denotes the canonical logistic, or non-minimal, model and
$\mathrm M$ denotes the canonical minimal model. The minimum
in~\eqref{E:chi-star-L} is attained because its modal terms tend to infinity as
$n\to\infty$. Moreover,
\begin{equation}\label{E:chi-star-lower}
  \chi_{\mathrm L,\beta}^*
  \ge \frac{2^\beta}{\gamma}
  \inf_{\lambda>0}\frac{(\lambda+\alpha)(1+\lambda)}{\lambda}
  =\frac{2^\beta}{\gamma}(1+\sqrt\alpha)^2.
\end{equation}

\begin{theorem}[Linear stability and instability]\label{T:linear-stability-thm}
  \begin{itemize}
    \item[(1)] In the canonical non-minimal model, if
    $\chi_0<\chi_{\mathrm L,\beta}^*$, then $(1,1)$ is linearly stable, whereas
    if $\chi_0>\chi_{\mathrm L,\beta}^*$, then it is linearly unstable.
    Moreover,  if $\chi_0<\chi_{\mathrm L,\beta}^*$, there are $C, \delta,
    \lambda > 0$ such that for any $u_0$ satisfying~\eqref{E:initial-cond-PE}
    and $\Norm{u_0-1}_\infty < \delta$, we have that $T_{\max}(u_0) =
    \infty$ and
    \begin{equation}\label{E:exponential-decay-eq1}
      \Norm{u(t,\cdot;u_0)-u^*}_{C^1(\overline\Omega)} + \Norm{v(t,\cdot;u_0)-v^*}_{C^1(\overline\Omega)}
      \le C e^{-\lambda t}\quad \forall\, t \ge 1.
    \end{equation}
    Here $(u^*,v^*)=(1,1)$.

    \item[(2)] In the canonical minimal model, fix $u^*>0$ and set
    $v^*=(u^*)^\gamma$. If $\chi_0<\chi_{\mathrm M,\beta}^*(u^*)$, then
    $(u^*,v^*)$ is linearly stable relative to mass-preserving perturbations,
    whereas if $\chi_0>\chi_{\mathrm M,\beta}^*(u^*)$, then it is linearly
    unstable relative to mass-preserving perturbations.
    Moreover,  if $\chi_0<\chi_{\mathrm M,\beta}^*(u^*)$, there are $C, \delta,
    \lambda > 0$ such that for any $u_0$ satisfying~\eqref{E:initial-cond-PE}
    with $\Norm{u_0-u^*}_\infty < \delta$ and $\int_\Omega u_0(x)\,dx =
    |\Omega| u^*$, we have that $T_{\max}(u_0) = \infty$
    and~\eqref{E:exponential-decay-eq1} holds.

  \end{itemize}
\end{theorem}

Theorem~\ref{T:linear-stability-thm} is proved in Section~\ref{S:linear-stability}.
We refer to
$\chi_{\mathrm L,\beta}^*$ and $\chi_{\mathrm M,\beta}^*(u^*)$ as the
\textit{critical linear sensitivity thresholds}.

Our third result concerns global stability of the unique positive
constant solution when the chemotaxis sensitivity  is relatively weak.
To state this theorem,
we first introduce some notation.

In the canonical non-minimal model, the equilibrium in~\eqref{E:equilibrium}
is $(1,1)$.
We use the notation $\xi\wedge\eta\coloneqq\min\{\xi,\eta\}$,
$\xi\vee\eta\coloneqq\max\{\xi,\eta\}$, and
$[z]_+\coloneqq\max\{z,0\}$.
For $m\ge 1$, let
\begin{align}\label{E:b-star}
  \chi_{\mathrm L,\beta}^{**,1}
  \coloneqq
  4\sqrt{\frac{1+\tilde{\beta}}{(2m-1)C_{\alpha,\gamma}}}
  \qquad \text{with} \quad
  \tilde{\beta} \coloneqq \left[1 \wedge (2\beta - 1)\right]_+,
\end{align}
where $C_{\alpha,\gamma}$ is defined by
\begin{equation}\label{E:power-diff-constant}
  C_{\alpha,\gamma} \coloneqq
  \begin{cases}
    \dfrac{(\alpha+1)^2}{4\alpha}, & 0<\alpha<1,                             \\[6pt]
    1,                             & \alpha\ge 1\ \text{and}\ 0<\gamma\le 1, \\[6pt]
    \dfrac{\gamma^2}{2\gamma-1},   & \alpha\ge 1\ \text{and}\ \gamma>1.
  \end{cases}
\end{equation}

For $m\ge 1$, define
\begin{equation}\label{E:bar-chi-bar-v-eq}
  \bar\chi_{\mathrm L,\beta} \coloneqq
  \dfrac{1}{2\Theta_{\beta-1}} \quad\text{(for $\beta\ge 1$)}, \quad
  \underline v_{\mathrm L} \coloneqq
  2^{-\frac{\gamma}{\alpha}},
\end{equation}
and
\begin{equation}\label{E:chi-triple-star}
  \chi_{\mathrm L,\beta}^{**,2}
  \coloneqq
  \begin{cases}
    \min\Big\{
    \bar\chi_{\mathrm L,\beta},\,
    \frac{4\left(1+\underline v_{\mathrm L}\right)^\beta}
    {\sqrt{(2m-1)C_{\alpha,\gamma}}}
    \Big\},\quad & \beta\ge 1\cr
    \chi_{\mathrm L,\beta}^{**,1},\quad  & 0\le \beta<1.
  \end{cases}
\end{equation}

Let $C_G(\Omega)>0$ be the constant, depending only on $\Omega$, in the
gradient estimate for the Neumann problem proved in
Lemma~\ref{L:gradient-glaeser}.
For $m\ge 1$, $\gamma\ge 1$, and $\beta\ge 0$, let
\begin{equation}\label{E:chi-3-star}
  \chi_{\mathrm L,\beta}^{**,3}
  \coloneqq \frac{1}{2+\beta C_G(\Omega)},
\end{equation}
and
\begin{equation}\label{E:chi-4-star-1}
  \chi_{\mathrm L,\beta}^{**,4}
  \coloneqq
  \begin{cases}\min\left\{\bar\chi_{\mathrm L,\beta},
    \left(1+\underline v_{\mathrm L}\right)^\beta
    \chi_{\mathrm L,\beta}^{**,3}\right\},\quad &\beta\ge 1\cr
    \chi_{\mathrm L,\beta}^{**,3},\quad &0\le \beta<1.
  \end{cases}
\end{equation}

\begin{theorem}[Global stability of the unique constant solution in the non-minimal
  model]
  \label{T:global-stable}

  Consider the canonical non-minimal model with  $m\ge 1$,  $\beta\ge 0$ and
  $\alpha,\gamma>0$.
  Assume that one of the following sets of conditions holds:

  \begin{itemize}

    \item[(i)] $\chi_0\le 0$.

    \item[(ii)]  $\alpha+1\ge 2\gamma $,  and  $0 < \chi_0 <\max\big\{    \chi_{\mathrm L,\beta}^{**,1},\,
    \chi_{\mathrm L,\beta}^{**,2}\big\}$.

    \item[(iii)] $\gamma\ge 1$, $\alpha+1\ge m+\gamma$, and
    $0<\chi_0<\max\big\{ \chi_{\mathrm L,\beta}^{**,3},  \, \chi_{\mathrm L,\beta}^{**,4}\big\}$.
  \end{itemize}
  Then $(1,1)$ is a globally asymptotically stable equilibrium
  of~\eqref{E:main-PE}. Moreover, there is a constant $\lambda>0$ such that,
  for every globally defined bounded positive classical solution
  $(u(t,x),v(t,x))$ of~\eqref{E:main-PE}, there is a constant $C>0$ such
  that~\eqref{E:exponential-decay-eq1} holds with $(u^*,v^*)=(1,1)$.
\end{theorem}

Theorem~\ref{T:global-stable} is proved in Section~\ref{S:Stability}.

Our fourth theorem concerns stability of positive constant solutions in the
canonical minimal model. Every such equilibrium is given by
\begin{equation}\label{E:equilibrium-minimal}
  (u^*,v^*)=\big(u^*,(u^*)^\gamma\big),\qquad u^*>0.
\end{equation}
This family of equilibria plays the same role as~\eqref{E:equilibrium} in the
non-minimal model.

Recall that $\overline{u}_0(u^*)$ is defined
in~\eqref{E:CN-eq} and $\underline{v}_0(u^*)$ is defined
in~\eqref{E:new-lower-bound-for-v}. Define
\begin{equation}\label{E:Gamma-gamma-minimal}
  \Gamma_\gamma(u^*)\coloneqq
  \begin{cases}
    (u^*)^{\gamma-1}\overline{u}_0(u^*),            & 0<\gamma\le 1, \\[6pt]
    \gamma \left(\overline{u}_0(u^*)\right)^\gamma, & \gamma>1,
  \end{cases}
\end{equation}
and
\begin{equation}\label{E:chi0-star-minimal}
  \chi_{\mathrm M,\beta}^{**}(u^*) \coloneqq \min\Big\{
  \frac{\chi_\beta}{2},\,
  \chi_\beta^{1/2},\,
  \frac{2\sqrt{\lambda_1}\left(1+\underline{v}_0(u^*)\right)^{\beta}}
  {\Gamma_\gamma(u^*)}
  \Big\},
\end{equation}
where $\lambda_1$ is the first nonzero Neumann eigenvalue of $-\Delta$ on
$\Omega$.

\begin{theorem}[Global stability of constant solutions in the minimal
  model]\label{T:stability-minimal-model}
  Consider the canonical minimal model, and let $(u^*,v^*)$ be a positive constant solution
  of~\eqref{E:main-PE}.   Assume that one of the following conditions holds.

  \begin{itemize}

    \item[(i)]  $m\ge 1$ and $\chi_0\le 0$.

    \item[(ii)]  $m = 1$, $\beta \ge 1$, and
    {$0<\chi_0<\chi_{\mathrm M,\beta}^{**}(u^*)$.}
  \end{itemize}
  Then $(u^*,v^*)$ is globally asymptotically stable in the sense of
  Definition~\ref{D:stability}. Moreover, there exists $\lambda > 0$ such
  that for any globally defined bounded positive classical solution
  $(u(t,x),v(t,x))$ of~\eqref{E:main-PE} with $\int_\Omega
  u(0,x)dx=|\Omega|u^*$, there exists a constant $C > 0$ such that \eqref{E:exponential-decay-eq1} holds. Under condition (ii),  every $u_0$ satisfying~\eqref{E:initial-cond-PE} and
  $\int_\Omega u_0(x)\,dx=|\Omega|u^*$ generates a global bounded positive
  solution.
\end{theorem}

Theorem~\ref{T:stability-minimal-model} is proved in
Section~\ref{S:stability-minimal-model}.
Under condition (i) no global existence assertion is made: the statement
concerns the globally defined bounded positive solutions themselves.
The asserted global existence and boundedness under condition (ii) follow from Part~I~\cite[Theorem~1.2(2)]{chen.ruau.ea:26:chemotaxis*1}, because $\chi_{\mathrm M,\beta}^{**}(u^*)\le\chi_\beta/2<\chi_\beta$ (see item (2) below).

\subsection{Remarks on the main results, biological implications,  and techniques}\label{SS:remarks}

In this subsection, we present some remarks on the  comparison of  the thresholds, the  implications of the main results, and    the  techniques developed in
the proofs.

\begin{remark}[Threshold hierarchy]
  \label{R:sensitivity-thresholds}
  \begin{itemize}
    \item[(1)]
    The quantities $\chi_{\mathrm L,\beta}^*$ and
    $\chi_{\mathrm M,\beta}^*(u^*)$ are the exact critical thresholds
    for the linearized operators.
    The quantities $\chi_{\mathrm L,\beta}^{**,i}$ ($i=1,2,3,4$) and
    $\chi_{\mathrm M,\beta}^{**}(u^*)$ are sufficient nonlinear thresholds under
    the respective hypotheses of Theorems~\ref{T:global-stable}
    and~\ref{T:stability-minimal-model}.
    We do not claim that these nonlinear thresholds are optimal.

    \item[(2)]  Under the corresponding theorem hypotheses, none of the nonlinear thresholds exceeds the relevant critical linear threshold, that is,
    \begin{equation*}
      \chi_{\mathrm L,\beta}^{**,i}\le \chi_{\mathrm L,\beta}^*\quad {\rm for}\quad i=1,2\quad {\rm provided \,\,\, that}\,\, m\ge 1,\,\, \alpha+1\ge 2\gamma,
    \end{equation*}
    \begin{equation*}
      \chi_{\mathrm L,\beta}^{**,i}\le \chi_{\mathrm L,\beta}^*\quad {\rm for}\quad i=3,4  \quad {\rm provided \,\,\, that}\,\, m\ge 1,\,\,\gamma\ge 1,\,\,  \alpha+1\ge m+\gamma,
    \end{equation*}
    and
    \begin{equation*}
      \chi_{\mathrm M,\beta}^{**}(u^*)\le \chi_{\mathrm M,\beta}^*(u^*) \quad {\rm provided\,\, that}\,\, m=1,\,\, \beta\ge 1.
    \end{equation*}
    These inequalities follow by direct computation from the definitions of the thresholds, \eqref{E:chi-star-lower} and \eqref{E:chi-star-M}.

    \item[(3)]    $\chi_{\mathrm L,\beta}^{**,i}$ ($i=1,2$) and $\chi_{\mathrm L,\beta}^{**,j}$ ($j=3,4$)
    are obtained via two different approaches. The regions $R_1:=\{\chi_0<\max\{\chi_{\mathrm L,\beta}^{**,1},\chi_{\mathrm L,\beta}^{**,2}\}\}$ and $R_2:=\{\chi_0<\max\{\chi_{\mathrm L,\beta}^{**,3},\chi_{\mathrm L,\beta}^{**,4}\}\}$ overlap. Which one is larger depends on the parameters $\beta,m,\alpha,\gamma$.
    The two conditions are  complementary in how the exponent $m$ of the
    mobility enters.   In {Theorem  \ref{T:global-stable} (ii)}  the exponent condition $\alpha+1\ge2\gamma$ does not
    involve $m$, but the upper bound
    $\max\{\chi_{\mathrm L,\beta}^{**,1},\chi_{\mathrm L,\beta}^{**,2}\}$ decreases
    at least like $(2m-1)^{-1/2}$. In {Theorem  \ref{T:global-stable} (iii)} the exponent condition $\alpha+1\ge m+\gamma$
    involves $m$, but the upper bound
    $\max\{\chi_{\mathrm L,\beta}^{**,3},\chi_{\mathrm L,\beta}^{**,4}\}$ does not.   When $\beta\to\infty$, both effective thresholds
    $\max\{\chi_{\mathrm L,\beta}^{**,1},\chi_{\mathrm L,\beta}^{**,2}\}$ and
    $\max\{\chi_{\mathrm L,\beta}^{**,3},\chi_{\mathrm L,\beta}^{**,4}\}$ tend to infinity, so
    both approaches eventually cover any fixed $\chi_0$.
  \end{itemize}
\end{remark}

\begin{remark}[Implications of the results and the impacts of $\beta$]
  \begin{itemize}
    \item[(1)] ($\beta\to\infty$).
    Observe that
    $$
    \lim_{\beta\to \infty} \chi_{\mathrm L,\beta}^*=\lim_{\beta\to\infty}\chi_{\mathrm M,\beta}^*=\infty,
    $$
    and the effective nonlinear thresholds diverge as well:
    $$
    \lim_{\beta\to \infty}\max\big\{\chi_{\mathrm L,\beta}^{**,1},\chi_{\mathrm L,\beta}^{**,2}\big\}
    =\lim_{\beta\to \infty}\max\big\{\chi_{\mathrm L,\beta}^{**,3},\chi_{\mathrm L,\beta}^{**,4}\big\}
    =\infty,
    $$
    while, at fixed mass (fixed $u^*$),
    $\lim_{\beta\to\infty}\chi_{\mathrm M,\beta}^{**}(u^*)=\infty$.
    Biologically, this indicates that signal
    saturation (large $\beta$) or repulsion ($\chi_0<0$) can prevent aggregation
    and promote relaxation to spatially homogeneous states.

    \item[(2)]  $(\beta=0$).
    At $\beta=0$, when $m\ge1$, $\gamma\ge1$, and
    $\alpha+1\ge m+\gamma$, the value
    $\chi_{\mathrm L,0}^{**,3}=1/2$ recovers the canonical condition in
    \cite[Theorem~1.2]{galakhov.salieva.ea:16:on}.
    When $m=\alpha=\gamma=1$,
    $\chi_{\mathrm L,0}^{**,1}=4$ is the entropy threshold already obtained by
    He and Zheng~\cite[Theorems~1--2]{he.zheng:16:convergence} and
    Xiang~\cite[Theorem~5.1]{xiang:19:dynamics}.
    Thus, our results recover the existing constant-sensitivity thresholds in the literature.
  \end{itemize}
\end{remark}

\begin{remark}[Methods and nonlinear mechanisms]
  \label{R:positive-sensitivity-strong-logistic}
  \begin{itemize}

    \item[(1)]  (A new Lyapunov function)
    As mentioned in the introduction,
    the factor $(1+v)^{-\beta}$ describes signal-dependent desensitization:
    large $v$ reduces the drift toward the signal when $\chi_0>0$.
    One therefore expects large $\beta$ to promote stabilization. However, the nonlinear sensitivity factor $(1+v)^{-\beta}$ complicates the analysis. In this paper,
    we introduce for $m\ge 1$ the following Lyapunov function
    \begin{equation}\label{E:F-def}
      F(u(t,\cdot)) \coloneqq \int_\Omega h_m\left(u(t,x)\right)\,dx,
      \quad {\rm where}\quad
      h_m(s)\coloneqq \int_1^s\left(1-\tau^{-(2m-1)}\right)\,d\tau
    \end{equation}
    in the study of global stability of constant solutions (see   Theorem \ref{T:global-stable}(ii)).
    Using this Lyapunov function, we
    prove
    $$
    \lim_{t\to\infty}\int_\Omega (u(t,x)-1)(u^\alpha(t,x)-1)=0,
    \quad {\rm
    and \,\,\, then}\,\,\,
    \lim_{t\to\infty}\|u(t,\cdot)-1\|_\infty=0.
    $$
    Note that, when $m=1$,
    $
    F(u(t,\cdot))=\int_\Omega \big(u(t,x)-1-\ln u(t,x)\big)dx,
    $
    which has been used in the study
    of positive stationary solutions in other chemotaxis models (see, for example,
    \cite{cao.wang.ea:16:asymptotic, he.zheng:16:convergence,
    kurt.shen:24:stabilization, wang.zhu.ea:23:long}).

    \item[(2)]  (Updated rectangle/ODE method) We extend the  rectangle/ODE  method in
    \cite[Section~3]{galakhov.salieva.ea:16:on} to study the global stability of constant solutions (see   Theorem \ref{T:global-stable}(iii)).  The difficulty is to handle the gradient-square term  $\chi_0\beta u^m(1+v)^{-\beta-1}\left|\nabla v\right|^2$ in \eqref{E:expanded-flux-intro}.
    Our idea is to establish the pointwise estimate $|\nabla v(x)|^2\le C_G(\Omega) v(x) \big((\sup_\Omega u)^\gamma-(\inf_\Omega u)^\gamma\big)$ for $\nabla v$, with a constant $C_G(\Omega)$ independent of the solution (see Lemma~\ref{L:gradient-glaeser}). To the best of our knowledge, this estimate has not been used in the chemotaxis literature.  We then find the following
    auxiliary ODE system
    \begin{equation*}
      \begin{dcases}
        \overline{u}_t=\big(1+\beta C_G \big)\widehat\chi_0 \overline{u}^m(\overline{u}^\gamma-\underline{u}^\gamma)
        +\overline{u}(1-\overline{u}^\alpha),
        & \overline{u}(0)=1\vee \Norm{u_0}_\infty,           \\
        \underline{u}_t=\widehat\chi_0 \underline{u}^m(\underline{u}^\gamma-\overline{u}^\gamma)+\underline{u}(1-\underline{u}^\alpha),
        & \underline{u}(0)=1\wedge \inf_{x\in\Omega}u_0(x),
      \end{dcases}
    \end{equation*}
    where  $$
    \widehat\chi_0\coloneqq \begin{cases}
      \chi_0(1+ (\underline v_{\mathrm L}-\epsilon))^{-\beta}\quad & {\rm if}\quad \chi_{\mathrm L,\beta}^{**,4}> \chi_{\mathrm L,\beta}^{**,3},\cr
      \chi_0\quad &{\rm if}\quad \chi_{\mathrm L,\beta}^{**,4}\le \chi_{\mathrm L,\beta}^{**,3},
    \end{cases}\quad 0<\epsilon\ll 1,
    $$
    to bound the solutions of \eqref{E:main-PE},  i.e.,
    $$
    \underline u(t)\le u(t,x;u_0)\le \overline{u}(t),
    $$
    and prove $\lim_{t\to\infty}\|u(t,\cdot;u_0)-1\|_\infty=0$. When $\beta=0$, the above ODE system is used in  \cite[Section~3]{galakhov.salieva.ea:16:on}.

    \item[(3)]  (Negative sensitivity and maximum principle) When $\chi_0<0$, the aforementioned Lyapunov function method and rectangle/ODE method are difficult to apply.
    We prove the stabilization of globally defined bounded solutions directly through a nontrivial application of the maximum principle and Hopf's lemma for parabolic equations.
  \end{itemize}
\end{remark}

\section{Persistence and proof of Theorem~\ref{T:persistence}}\label{S:Persistence}

In this section, we study persistence and prove Theorem \ref{T:persistence}.

\subsection{Proof of Theorem \ref{T:persistence}(1)}
\label{SS:Persistence-1}

\begin{proof}[Proof of Theorem~\ref{T:persistence}(1)]
  Let $(u,v)$ be a globally defined bounded positive solution of~\eqref{E:main-PE}.
  Fix $T>0$ and define, for each $t>0$,
  \begin{equation}\label{E:super-inf-u}
    \underline{u}(t)\coloneqq\min_{x\in\overline\Omega} u(t,x) \quad \text{and}\quad
    \overline{u}(t)\coloneqq\max_{x\in\overline\Omega}  u(t,x).
  \end{equation}
  It then suffices to prove that
  \begin{equation}\label{E:positive-lower-bound-eq}
    \liminf_{t\to\infty} \underline{u}(t)>0.
  \end{equation}
  We divide the proof of \eqref{E:positive-lower-bound-eq} into two cases.
  \smallskip

  \noindent\textbf{Case 1 (the canonical minimal model).} In this case, we have
  $\int_{\Omega}u(t,x;u_0)=\int_{\Omega}u_0(x)$ for all $t>0$. This implies that
  \begin{equation}\label{E:positive-lower-bound-eq1}
    \overline{u}(t)\ge \frac{1}{|\Omega|}\int_\Omega u_0(x)\,dx\quad \forall\, t>0.
  \end{equation}
  Assume that~\eqref{E:positive-lower-bound-eq} does not hold. Then there are
  sequences $\{t_n\}\subset(0,\infty)$ and $\{x_n\}\subset\Omega$ such that
  \[
    t_n\to \infty
    \qquad \text{and} \qquad
    u(t_n,x_n;u_0)\to 0.
  \]
  By Lemma~\ref{L:Holder}(4), without loss of generality, we may assume that
  $x_n\to x^*$,
  \[
    u(t+t_n,x;u_0)\to u^\infty(t,x)
    \quad \text{and} \quad
    v(t+t_n,x;u_0)\to v^\infty(t,x),
  \]
  and that $(u^\infty(t,x),v^\infty(t,x))$ is a nonnegative classical solution
  of the canonical minimal model defined for all $t\in\mathbb{R}$. Then
  $u^\infty(0,x^*)=0$.
  By~\eqref{E:positive-lower-bound-eq1}, for each $t\in\mathbb{R}$ we have
  $u^\infty(t,\cdot)\not\equiv 0$. Expanding the flux as
  in~\eqref{E:expanded-flux-intro}, $u^\infty\ge 0$ solves a linear parabolic
  equation
  \[
    u^\infty_t-\Delta u^\infty+B\cdot\nabla u^\infty-H\,u^\infty=0,
  \]
  whose coefficients $B$ and $H$ are bounded, because $m\ge 1$ and
  $(u^\infty,v^\infty)$ is bounded in $C^{1,2}\times C^{0,2}$.
  By the strong comparison principle for parabolic equations,
  $u^\infty(0,x)>0$, for all
  $x\in\overline\Omega$. This contradicts $u^\infty(0,x^*)=0$. Therefore,
  \eqref{E:positive-lower-bound-eq} holds in the canonical minimal model.
  \smallskip

  \noindent\textbf{Case 2 (the canonical non-minimal model).} To
  prove~\eqref{E:positive-lower-bound-eq}, we show that
  $\underline{u}$ cannot remain arbitrarily small on any open time interval
  after time $T$.  The proof proceeds by three claims. \smallskip

  \noindent\textbf{Claim~1.} For every $\varepsilon\in(0,1)$ there exists
  $\delta_\varepsilon \in (0, \varepsilon]$ such that for all $t\ge T$ either
  $\underline{u}(t) \ge \delta_\varepsilon$ or $\overline{u}(t) -
  \underline{u}(t) \le \varepsilon$. \medskip

  Suppose, by contradiction, that there exist $\varepsilon_0 > 0$ and a
  sequence $t_n \ge T$ with $t_n \to \infty$ and
  \[
    \underline{u}(t_n) < \delta_n\quad\text{and}\quad
    \overline{u}(t_n)-\underline{u}(t_n) > \varepsilon_0,\qquad \delta_n \coloneqq \frac1n.
  \]
  By Lemma~\ref{L:Holder}(4), after time translation we can extract a
  subsequence (not relabeled here) and a global classical limit $(u^\infty,
  v^\infty)$ on $\R\times\Omega$ such that
  \[
    u(t+t_n,\cdot)\to u^\infty(t,\cdot),\qquad v(t+t_n,\cdot)\to v^\infty(t,\cdot)
    \quad\text{locally uniformly on }\R\times\overline\Omega.
  \]
  In particular,
  \[
    \inf_{x\in\Omega} u^\infty(0,x)=\lim_{n\to\infty}\underline{u}(t_n)=0 \quad \text{and} \quad
    \sup_{x\in\Omega} u^\infty(0,x)\ge\varepsilon_0.
  \]
  The limiting pair solves the canonical non-minimal model and inherits its Neumann boundary conditions; expanding the divergence term in the $u$-equation gives
  \[
    u^\infty_t-\Delta u^\infty+B(t,x)\cdot\nabla u^\infty-H(t,x)u^\infty=0
    \quad\text{on }(-1,1)\times\Omega,\quad
    \partial_n u^\infty=0\ \text{on }(-1,1)\times\partial\Omega,
  \]
  where
  \begin{align*}
    B(t,x) & = m\chi_0\,\frac{(u^\infty)^{m-1}}{(1+v^\infty)^\beta}\,\nabla v^\infty \quad \text{and}                                                                                                 \\
    H(t,x) & = 1-(u^\infty)^\alpha
    -\chi_0\,\frac{(u^\infty)^{m-1}}{(1+v^\infty)^\beta}\,\Delta v^\infty
    +\chi_0\beta\,\frac{(u^\infty)^{m-1}}{(1+v^\infty)^{\beta+1}}\,|\nabla v^\infty|^2.
  \end{align*}
  Since $u^\infty, v^\infty, \nabla v^\infty,\Delta v^\infty$ are bounded, the
  coefficients $B$ and $H$ are bounded. Because $u^\infty\ge0$ on
  $(-1,1)\times\Omega$ and attains its minimum value~$0$ at the interior time
  $t=0$, by the strong comparison principle for parabolic equations, $u^\infty\equiv0$ on $(-1,0]\times\Omega$. This contradicts
  $\sup_xu^\infty(0,x)\ge\varepsilon_0$. Hence, Claim~1 holds. \smallskip

  \noindent\textbf{Claim~2.} There exists
  $\varepsilon\in(0,\min\{1/2,\underline u(T)\})$ such that, for the
  corresponding $\delta_\varepsilon$ from Claim~1, there is no interval
  $(t_1,t_2)$ with $T<t_1<t_2\le\infty$ for which
  $\underline u(t_1)=\delta_\varepsilon$ and
  $\underline u(t)<\delta_\varepsilon$ for all $t\in(t_1,t_2)$. \medskip

  Apply the same elliptic estimates as in the proof of
  Lemma~\ref{L:Holder}(2), together with the boundedness of $u$, to
  $-\Delta v+v=u^\gamma$. There exists $K>0$ such that, for all $t\ge T$,
  \begin{equation}\label{E:estimate-for-v}
    \Norm{v(t,\cdot)}_{C(\overline{\Omega})}
    + \beta\, \Norm{\nabla v(t,\cdot)}^2_{C(\overline{\Omega})}
    + \Norm{\Delta v(t,\cdot)}_{C(\overline{\Omega})}
    \le K\,\overline{u}^\gamma(t).
  \end{equation}
  Set $\eta_\varepsilon\coloneqq2^\gamma|\chi_0|K\varepsilon^\gamma$ and
  choose $\varepsilon\in(0,\min\{1/2,\underline u(T)\})$ so small that
  $\eta_\varepsilon<1$ and
  $\varepsilon<(1-\eta_\varepsilon)^{1/\alpha}$.
  Assume, to the contrary, that such an interval exists. Then Claim~1 gives
  \[
    \overline{u}(t)
    \le \varepsilon+\underline{u}(t)
    \le \varepsilon+\delta_\varepsilon\le 2\varepsilon
    <1, \quad \text{for all $t\in(t_1,t_2)$.}
  \]
  In particular, $\overline{u}^\gamma(t)\le(2\varepsilon)^\gamma$ on
  $(t_1,t_2)$.
  Using~\eqref{E:estimate-for-v} in the $u$-equation and $m\ge1$ yields, on
  $(t_1, t_2) \times \Omega$,
  \[
    u_t \ge \Delta u - \frac{\chi_0 m u^{m-1}}{(1+v)^\beta}\,\nabla u\cdot\nabla v
    +(1-\eta_\varepsilon)u-u^{1+\alpha}.
  \]
  Let $\tilde{u}$ solve the spatially homogeneous ODE
  $\tilde{u}_t=(1-\eta_\varepsilon)\tilde{u}-\tilde{u}^{1+\alpha}$ with
  $\tilde{u}(t_1)=\delta_\varepsilon$.
  The parabolic comparison principle implies $u(t,x)\ge\tilde{u}(t)$ for every $t\in[t_1,t_2)$ and
  $x\in\overline\Omega$.
  Since $\delta_\varepsilon <
  (1-\eta_\varepsilon)^{1/\alpha}$, the ODE solution
  satisfies $\tilde{u}(t)>\delta_\varepsilon$ for all $t>t_1$, hence
  $u(t,x)>\delta_\varepsilon$ on $(t_1,t_2)\times\Omega$, contradicting the
  assumption. Thus, Claim~2 follows. \bigskip

  \noindent\textbf{Claim~3.} $\displaystyle \liminf_{t\to\infty}\underline{u}(t)
  > 0$, i.e., \eqref{E:positive-lower-bound-eq} holds. \medskip

  Let $\varepsilon$ and $\delta_\varepsilon$ be as in Claim~2. The set
  \[
    \mathcal{O}_\varepsilon
    \coloneqq\left\{t>T:\underline u(t)<\delta_\varepsilon\right\}
  \]
  is open in $(T,\infty)$. If it were nonempty, any connected component
  $(t_1,t_2)$ would have $t_1>T$: indeed,
  $\underline u(T)>\varepsilon\ge\delta_\varepsilon$, so continuity gives a
  right neighborhood of $T$ disjoint from $\mathcal{O}_\varepsilon$.
  Continuity at the left endpoint would then give
  $\underline u(t_1)=\delta_\varepsilon$, while the definition of the component
  gives $\underline u(t)<\delta_\varepsilon$ for $t\in(t_1,t_2)$, contrary to
  Claim~2. Thus, $\mathcal{O}_\varepsilon$ is empty. Therefore,
  $\underline{u}(t) \ge \delta_\varepsilon$ for all $t \ge T$, and so
  $\liminf_{t\to\infty}\underline{u}(t) \ge \delta_\varepsilon > 0$.
  This proves Claim~3 and the $u$-part of~\eqref{E:persistence}.
  For the $v$-part, fix $t>0$.
  By the elliptic comparison principle, we have
  \begin{equation}\label{E:v-lower}
    \inf_{x\in\Omega} v(t,x)
    \;\ge\;\Big(\inf_{x\in\Omega}u(t,x)\Big)^{\!\gamma}
    \quad \text{for all $t>0$,}
  \end{equation}
  which implies the $v$-part of~\eqref{E:persistence}.  This completes
  the proof of Theorem~\ref{T:persistence}(1).
\end{proof}

\subsection{Proof of Theorem \ref{T:persistence}(2)}
\label{SS:Persistence-2}

\begin{proof}[Proof of Theorem \ref{T:persistence}(2)]
  Consider the canonical non-minimal model and assume that $\chi_0>0$,  $m\ge 1$,
  $\beta\ge1$, and $\chi_0<1/\Theta_{\beta-1}$.
  We first note that \eqref{E:new-eventual-lower-bound-for-v} follows from
  \eqref{E:new-eventual-lower-bound-for-u} via the comparison principle for elliptic equations. Hence it suffices  to  prove \eqref{E:new-eventual-lower-bound-for-u}.
  Put $c\coloneqq\chi_0\Theta_{\beta-1}\in(0,1)$.
  Since $\frac{v}{(1+v)^\beta}\le \Theta_{\beta-1}$ for all $v\ge 0$, we have
  \begin{align*}
    u_t
    & \ge
    \Delta u
    - \chi_0\,\frac{ m u^{m-1}}{(1+v)^\beta}\nabla u\cdot\nabla v
    - \chi_0 u^m \frac{1}{(1+v)^\beta}(v-u^\gamma)+u-u^{1+\alpha}\\
    &\ge
    \Delta u
    - \chi_0\,\frac{ m u^{m-1}}{(1+v)^\beta}\nabla u\cdot\nabla v
    + u-c u^m-u^{1+\alpha}.
  \end{align*}
  Let $w$ be the solution of
  \begin{equation}\label{E:new-ode-eq2}
    w'(t)=w(t)-w(t)^{1+\alpha}-c w(t)^m,
    \qquad
    w(0)=\inf_{x\in\Omega}u(0,x).
  \end{equation}
  The parabolic comparison principle implies $w(t)\le u(t,x)$ for all $t\ge 0$ and $x\in\Omega$.

  For $s>0$, define
  $G(s)\coloneqq1-s^\alpha-cs^{m-1}$. Since $G$ is continuous and strictly decreasing on
  $(0,\infty)$ and satisfies
  \[
    \lim_{s\searrow 0}G(s)=
    \begin{cases}
      1-c>0, & m=1, \\
      1>0,   & m>1,
    \end{cases}
    \quad \text{and} \quad
    \lim_{s\to\infty}G(s)=-\infty,
  \]
  there exists a unique $s_\infty>0$ such that $G(s_\infty)=0$, that is,
  $s_\infty^\alpha+c s_\infty^{m-1}=1$.  Note that
  $
  G(1)<0$.
  Hence $0<s_\infty<1$. Note also that $G(s)\ge 1-s^\alpha- c$ for $0<s<1$.
  This implies that
  $s_\infty\ge (1-c)^{\frac{1}{\alpha}}$.
  Therefore,
  \begin{equation}\label{E:new-eventual-lower-bound-for-w}
    \lim_{t\to\infty} w(t)
    =s_\infty
    \ge (1-c)^{\frac{1}{\alpha}}
    =\left(1-\chi_0\Theta_{\beta-1}\right)^{\frac{1}{\alpha}}.
  \end{equation}
  This proves \eqref{E:new-eventual-lower-bound-for-u}.
\end{proof}

\subsection{Proof of Theorem \ref{T:persistence}(3)}
\label{SS:Persistence-4}

\begin{proof}[Proof of Theorem \ref{T:persistence}(3)]
  Consider the canonical minimal model.
  Its conserved mass is $u^*|\Omega|$.
  If $\gamma\ge1$, then by Jensen's inequality we have
  \begin{equation}\label{E:new-persist-eq1}
    \int_{\Omega} u^\gamma(t,x)\,dx
    \ge |\Omega|\Big(\frac{1}{|\Omega|}\int_\Omega u(t,x)\,dx\Big)^{\!\gamma}
    = |\Omega|\,(u^*)^\gamma
    \qquad \text{for all $t\ge 0$.}
  \end{equation}
  If $0<\gamma<1$, then by Lemma~\ref{L:minimal-eventual-bounds}, there exists
  $T_1>0$ such that for all $t\ge T_1$, $\Norm{u(t,\cdot)}_\infty \le
  \overline{u}_0(u^*)$. Since $\gamma-1<0$, we have, for all $t\ge T_1$,
  $u^{\gamma-1}(t,\cdot)\ge \left(\overline{u}_0(u^*)\right)^{\gamma-1}$, and
  hence
  \begin{align}\label{E:new-persist-eq2}
    \int_\Omega
    u^\gamma(t,x)\,dx = \int_\Omega u(t,x)\,u^{\gamma-1}(t,x)\,dx
    \ge \frac{1}{\left(\overline{u}_0(u^*)\right)^{1-\gamma}} \int_\Omega u(t,x)\,dx
    = \frac{|\Omega|u^*}{\left(\overline{u}_0(u^*)\right)^{1-\gamma}}.
  \end{align}

  We next derive a pointwise lower bound for $v$ from the elliptic equation
  $-\Delta v+v=u^\gamma$ in $\Omega$ with $\partial_n v=0$ on $\partial\Omega$.
  Let $K(\tau,x,y)$ be the Neumann heat kernel on $\Omega$.
  Using the standard resolvent identity (cf.
  \cite[Lemma~2.1]{fujie.winkler.ea:15:boundedness}), we may write, for each
  fixed $t\ge 0$,
  \begin{equation*}
    v(t,x)
    =\int_0^\infty e^{-\tau}\left(e^{\tau\Delta}u^\gamma(t,\cdot)\right)(x)\,d\tau
    =\int_0^\infty e^{-\tau}\int_\Omega K(\tau,x,y)u^\gamma(t,y)\,dy\,d\tau.
  \end{equation*}
  By the Gaussian lower bound for the Neumann heat kernel on a smooth domain
  (see~\cite[Theorem~4]{choulli.ouhabaz.ea:06:stable}), there exist positive
  constants $c_0$, $C_0$, and $\omega_0$ depending only on $(\Omega,N)$ such that
  \begin{equation*}
    K(\tau,x,y)
    \ge c_0\frac{1}{(4\pi\tau)^{\frac{N}{2}}}\exp\!\Big(-\omega_0\tau-\frac{C_0}{\tau}\Big)
    \qquad \text{for all $\tau>0$ and all $x,y\in\Omega$.}
  \end{equation*}
  Consequently, for all $t\ge 0$ and $x\in\Omega$, $v(t,x) \ge \delta_0
  \int_\Omega u^\gamma(t,y)\,dy$, where we may take
  \begin{equation*}
    \delta_0
    \coloneqq
    c_0 \int_0^\infty \frac{1}{(4\pi\tau)^{\frac{N}{2}}}
    \exp\!\Big(-(1+\omega_0)\tau-\frac{C_0}{\tau}\Big)\,d\tau
    \in (0,\infty).
  \end{equation*}
  This, together with \eqref{E:new-persist-eq1} and
  \eqref{E:new-persist-eq2},
  implies that \eqref{E:new-lower-bound-for-v} holds with
  $C_\Omega\coloneqq \delta_0|\Omega|$.
\end{proof}

\section{Linear stability and instability and proof of
Theorem~\ref{T:linear-stability-thm}}\label{S:linear-stability}

In this section, we study linear stability and instability of constant solutions and prove
Theorem~\ref{T:linear-stability-thm}.

We first present a lemma.
For $1<p<\infty$, let $A_p=-\Delta+I$ on $L^p(\Omega)$ with
\[
  \mathcal D(A_p)
  =\left\{w\in W^{2,p}(\Omega)\,\middle|\,\partial_nw=0
  \text{ on }\partial\Omega\right\},
\]
and write $X_p^\sigma\coloneqq\mathcal D(A_p^\sigma)$ for its fractional
power space; see~\cite[Definitions~1.3.1 and~1.4.1]{henry:81:geometric}.

\begin{lemma}
  \label{L:linear-stability-lm}
  Fix $1/2<\sigma<1$ and let $u^*>0$ be the $u$-component of an equilibrium of
  either canonical model.
  For every $\epsilon>0$, there are $T_0\in(0,1]$ and $\delta>0$ such that
  $\Norm{u_0-u^*}_\infty<\delta$ implies $T_{\max}(u_0)>T_0$ and
  \[
    \Norm{u(T_0,\cdot;u_0)-u^*}_{X_p^\sigma}<\epsilon.
  \]
\end{lemma}

\begin{proof}
  By the fixed-point proof of Part~I
  \cite[Proposition~1.1]{chen.ruau.ea:26:chemotaxis*1}, there are $T_0$,
  $0<\delta_0<u^*/2$, $C>0$, and  $0<\kappa<1$ such that
  \begin{equation}
    \label{E:linear-stability-eq1}
    \frac{u^*}{2}\le u(t,x;u_0)\le2u^*
    \quad\text{for $0\le t\le T_0$ and $x\in\Omega$}
  \end{equation}
  whenever $\Norm{u_0-u^*}_\infty\le\delta_0$, and for any $w_1,w_2$  with
  $\|w_i-u^*\|_\infty<\delta_0$,
  \[
    \sup_{0\le t\le T_0}
    \Norm{u(t,\cdot;w_1)-u(t,\cdot;w_2)}_\infty
    \le\frac{C}{1-\kappa}\Norm{w_1-w_2}_\infty.
  \]

  We claim that for any $\epsilon>0$, there is $0<\delta<\delta_0$ such that
  \[
    \Norm{u(T_0,\cdot;u_0)-u^*}_{X_p^\sigma}<\epsilon
  \]
  for any $u_0$ with $\|u_0-u^*\|_\infty<\delta$.
  Suppose the claim does not hold. Then there are
  $u_n\to u^*$ uniformly and some $\epsilon_0>0$ such that
  \[
    \Norm{u(T_0,\cdot;u_n)-u^*}_{X_p^\sigma}\ge\epsilon_0.
  \]
  By Lemma~\ref{L:Holder} and \eqref{E:linear-stability-eq1},  after passing to a subsequence,
  there is $u^\infty$ such that
  \[
    u(\cdot,\cdot;u_n)\longrightarrow u^\infty  \quad\text{in }C^{1,2}([T_0/2,T_0]\times\overline\Omega),
  \]
  and
  $\sup_{0\le t\le T_0}\Norm{u(t,\cdot;u_n)-u(t,\cdot;u^*)}_\infty
  \le\frac{C}{1-\kappa}\Norm{u_n-u^*}_\infty\longrightarrow 0$. So,
  $u^\infty=u(\cdot,\cdot;u^*)\equiv u^*$.
  Since $C^2$ convergence under the Neumann boundary condition implies convergence in $X_p^\sigma$ for $\sigma<1$,
  $\Norm{u(T_0,\cdot;u_n)-u^*}_{X_p^\sigma}\longrightarrow 0$, a
  contradiction. The claim then holds and the lemma is proved.
\end{proof}

\subsection{Proof of  Theorem \ref{T:linear-stability-thm}(1)}
\label{SS:linear-stability-1}

\begin{proof}[Proof of Theorem \ref{T:linear-stability-thm}(1)]
  We divide the proof into two steps. \smallskip

  \noindent\textbf{Step 1.} We prove the linear alternatives. Consider the eigenvalue problem of the
  linearization at $(1,1)$:
  \begin{equation}\label{E:main-linear-PE-L}
    \begin{dcases}
      \Delta U-\dfrac{\chi_0}{2^\beta}\Delta V-\alpha U=\sigma U, & x\in\Omega, \cr
      0=\Delta V-V+\gamma U,                                  & x\in\Omega.
    \end{dcases}
  \end{equation}
  Let $0=\lambda_0<\lambda_1\le\lambda_2\le\cdots$ denote the eigenvalues of
  the negative Neumann Laplacian, repeated according to multiplicity, and let  $\{\phi_n\}_{n\ge0}$ be corresponding Neumann eigenfunctions.
  Let
  \begin{equation}\label{E:sigma-L}
    \sigma_{\mathrm L,n}(\chi_0)
    \coloneqq-\lambda_n-\alpha
    +\frac{\chi_0\gamma}{2^\beta}\frac{\lambda_n}{1+\lambda_n}.
  \end{equation}
  It can be verified directly that   for any $n\ge 0$, $\sigma=\sigma_{\mathrm L,n}(\chi_0)$ is an eigenvalue of \eqref{E:main-linear-PE-L} with $\phi_n$ being an eigenfunction.
  Using $V=\gamma RU$, where $R=(-\Delta+I)^{-1}$ with the Neumann boundary condition, we find that the operator in~\eqref{E:main-linear-PE-L} is diagonal in the complete orthogonal basis $\{\phi_n\}_{n\ge0}$ of $L^2(\Omega)$; hence $\{\sigma_{\mathrm L,n}\}_{n\ge0}$ are all of its eigenvalues.

  For $n\ge1$, solving
  $\sigma_{\mathrm L,n}=0$ for $\chi_0$ gives
  \[
    \chi_0=\frac{2^\beta}{\gamma}
    \frac{(\lambda_n+\alpha)(1+\lambda_n)}{\lambda_n}.
  \]
  Since $\sigma_{\mathrm L,0}=-\alpha$, definition~\eqref{E:chi-star-L}
  now gives strict negativity of each $\sigma_{\mathrm L,n}$ when $\chi_0<\chi_{\mathrm L,\beta}^*$, and a positive $\sigma_{\mathrm L,n}$ for some $n\ge 1$ when $\chi_0>\chi_{\mathrm L,\beta}^*$. Hence, if $\chi_0<\chi_{\mathrm L,\beta}^*$, then $(1,1)$ is linearly stable, whereas if $\chi_0>\chi_{\mathrm L,\beta}^*$, then it is linearly unstable.

  \smallskip

  \noindent\textbf{Step 2.} In this step, we
  prove~\eqref{E:exponential-decay-eq1}. \smallskip

  Put $\widetilde u=u-1$, $\widetilde v=v-1$, and
  $R\coloneqq(-\Delta+I)^{-1}$ with the Neumann boundary condition. Then
  \begin{subequations}\label{E:main-perturbation-eq1}
    \begin{equation}\label{E:main-perturbation-eq1-L}
      \widetilde u_t
      =\Delta\widetilde u-\frac{\chi_0\gamma}{2^\beta}(R-I)\widetilde u
      -\alpha\widetilde u+F_{\mathrm L}(\widetilde u),
    \end{equation}
    with the Neumann boundary condition, where
    \begin{align}\label{E:main-perturbation-eq1-F}
      F_{\mathrm L}(\widetilde u)
      \coloneqq-\chi_0\nabla\cdot\Big(
      \frac{(1+\widetilde u)^m}{(2+\widetilde v)^\beta}
      \nabla\widetilde v\Big)
      +(1+\widetilde u)-(1+\widetilde u)^{1+\alpha}\nonumber+\frac{\chi_0\gamma}{2^\beta}(R-I)\widetilde u
      +\alpha\widetilde u,
    \end{align}
    and \begin{equation}\label{E:main-perturbation-eq1-v}
      \widetilde v=R\left((1+\widetilde u)^\gamma-1\right).
    \end{equation}
  \end{subequations}
  Choose $1/2<\vartheta<1$ and $p>1$ so that
  $2\vartheta-N/p>1$, and fix
  $0<\eta<2\vartheta-N/p-1$.
  By
  \cite[Theorem~1.6.1]{henry:81:geometric}, we have
  \[
    X_p^\vartheta\hookrightarrow C^{1,\eta}(\overline\Omega)
    \hookrightarrow C^1(\overline\Omega).
  \]
  Let $C_{\rm emb}$ be the resulting embedding constant and choose
  $r_{\mathrm L}>0$ so that
  $C_{\rm emb}r_{\mathrm L}\le1/2$ and $r_{\mathrm L}\le1/2$. Thus, for
  \[
    B_{\mathrm L}(r_{\mathrm L})
    \coloneqq\left\{w\in X_p^\vartheta\,\middle|\,
    \Norm{w}_{X_p^\vartheta}<r_{\mathrm L}\right\},
  \]
  we have
  \[
    1+w(x)\ge\frac12\ge r_{\mathrm L}
    \quad\text{for every $x\in\overline\Omega$ and
    $w\in B_{\mathrm L}(r_{\mathrm L})$.}
  \]
  It is then not difficult to see that
  $F_{\mathrm L}:B_{\mathrm L}(r_{\mathrm L})\to L^p(\Omega)$ is locally
  Lipschitz, and
  $F_{\mathrm L}(0)=0$ and $DF_{\mathrm L}(0)=0$. Consequently,
  \begin{equation}\label{E:F-Lip}
    \Norm{F_{\mathrm L}(w)}_{L^p}
    =o\left(\Norm{w}_{X_p^\vartheta}\right)
    \quad\text{as $\Norm{w}_{X_p^\vartheta}\to0$.}
  \end{equation}

  Define the linear operator
  \begin{equation}\label{E:A-star}
    \mathcal L_{\mathrm L}w
    \coloneqq\Delta w-\frac{\chi_0\gamma}{2^\beta}(R-I)w-\alpha w
  \end{equation}
  with domain
  \[
    \mathcal D(\mathcal L_{\mathrm L})
    =\big\{w\in W^{2,p}(\Omega)\,\big|\,
    \frac{\partial w}{\partial n}=0\text{ on }\partial\Omega\big\}.
  \]
  The operator $-\mathcal L_{\mathrm L}$ is sectorial on $L^p(\Omega)$ by the
  bounded-perturbation theorem
  \cite[Theorem~1.3.2]{henry:81:geometric}. The spectrum of
  $\mathcal L_{\mathrm L}$ is
  $\{\sigma_{\mathrm L,n}\}_{n\ge0}$. If
  $\chi_0<\chi_{\mathrm L,\beta}^*$, then
  \[
    \omega_{\mathrm L}
    \coloneqq\min\big\{\alpha,\min_{n\ge1}(-\sigma_{\mathrm L,n})\big\}>0.
  \]
  By the principle of linearized stability \cite[Theorem~5.1.1]{henry:81:geometric}, for any $0<\lambda<\omega_{\mathrm L}$ there are $0<\epsilon<r_{\mathrm L}$ and $C>0$ such that every $\widetilde u_0\in X_p^\vartheta$ with $\Norm{\widetilde u_0}_{X_p^\vartheta}\le\epsilon$ generates a global solution satisfying
  \begin{equation}\label{E:stability-eq4}
    \Norm{\widetilde u(t,\cdot;\widetilde u_0)}_{X_p^\vartheta}
    \le Ce^{-\lambda t}\Norm{\widetilde u_0}_{X_p^\vartheta}
    \quad\forall\,t\ge0.
  \end{equation}
  By Lemma~\ref{L:linear-stability-lm}, there is
  $0<\delta<r_{\mathrm L}$ such that $\Norm{u_0-1}_\infty<\delta$ implies
  \[
    \Norm{u(1,\cdot;u_0)-1}_{X_p^\vartheta}\le\epsilon.
  \]
  By ~\eqref{E:stability-eq4}
  and  the embedding
  $X_p^\vartheta\hookrightarrow C^1(\overline\Omega)$,  we have
  \begin{equation*}
    \Norm{u(t+1,\cdot;u_0)-1}_{C^1(\overline\Omega)}= \Norm{u(t,\cdot;u(1,\cdot;u_0))-1}_{C^1(\overline\Omega)}
    \le C\epsilon e^{-\lambda t}\quad\forall\,t\ge0.
  \end{equation*}
  Absorbing the fixed factors $\epsilon$ and $e^{\lambda }$ into $C$ gives
  \begin{equation}\label{E:stability-eq5}
    \Norm{u(t,\cdot;u_0)-1}_{C^1(\overline\Omega)}
    \le Ce^{-\lambda t}\quad\forall\,t\ge1.
  \end{equation}
  The resolvent identity
  $v-1=R(u^\gamma-1)$, elliptic regularity, and the local Lipschitz
  continuity of $z\mapsto z^\gamma$ on the positive range of the solution
  give
  \begin{equation}\label{E:stability-eq6}
    \Norm{v(t,\cdot;u_0)-1}_{C^1(\overline\Omega)}
    \le C\Norm{u^\gamma(t,\cdot;u_0)-1}_\infty
    \le Ce^{-\lambda t}\quad\forall\,t\ge1.
  \end{equation}
  This, together with~\eqref{E:stability-eq5},
  proves~\eqref{E:exponential-decay-eq1}. This completes the proof of
  Theorem~\ref{T:linear-stability-thm}(1).
\end{proof}

\subsection{Proof of Theorem \ref{T:linear-stability-thm}(2)}
\label{SS:linear-stability-2}

In this subsection, we prove Theorem \ref{T:linear-stability-thm}(2). Throughout
this subsection, we fix $p>1$ and decompose
\begin{subequations}\label{E:Lp-decomposition}
  \begin{equation}
    L^p(\Omega)=X_1\oplus X_2,
  \end{equation}
  where
  \begin{equation}
    X_1 \coloneqq \big\{u\in L^p(\Omega) \big | u \text{ is a constant function} \big\}\quad \text{and} \quad
    X_2 \coloneqq \big\{u\in L^p(\Omega)\,\big|\,\int_\Omega u(x)\,dx=0\big\}.
  \end{equation}
\end{subequations}
Thus, every $u\in L^p(\Omega)$ can be written uniquely as
$u=\bar{u}+(u-\bar{u})$ with $\bar{u}\coloneqq \tfrac1{|\Omega|}\int_\Omega
u(x)\,dx\in X_1$ and $u-\bar u\in X_2$.

\begin{proof}[Proof of Theorem \ref{T:linear-stability-thm}(2)]
  First, consider the eigenvalue problem of the  linearization at $(u^*,(u^*)^\gamma)$:
  \begin{equation}\label{E:main-linear-PE-M}
    \begin{dcases}
      \Delta U-\chi_0\dfrac{(u^*)^m}{(1+(u^*)^\gamma)^\beta}\Delta V=\sigma U, & x\in\Omega, \cr
      0=\Delta V-V+\gamma(u^*)^{\gamma-1}U,                              & x\in\Omega.
    \end{dcases}
  \end{equation}
  Let
  \begin{equation}\label{E:sigma-M}
    \sigma_{\mathrm M,n}(\chi_0;u^*)
    \coloneqq-\lambda_n
    +\frac{\chi_0\gamma(u^*)^{m+\gamma-1}}{(1+(u^*)^\gamma)^\beta}
    \frac{\lambda_n}{1+\lambda_n},\quad n\ge 0.
  \end{equation}
  Similarly, $\{\sigma_{\mathrm M,n}\}_{n\ge0}$ are all the eigenvalues of \eqref{E:main-linear-PE-M}. Note that $\sigma_{\mathrm M,0}=0$;  the constant mode $n=0$ is excluded here, because in the minimal model linear stability is understood on the mean-zero space $X_2$ of mass-preserving perturbations (Definition~\ref{D:stability}), which is spanned by $\{\phi_n\}_{n\ge1}$. For $n\ge1$, solving
  $\sigma_{\mathrm M,n}=0$ for $\chi_0$ gives
  \[
    \chi_0=\frac{(1+(u^*)^\gamma)^\beta}
    {\gamma(u^*)^{m+\gamma-1}}(1+\lambda_n).
  \]
  These modal thresholds are minimized at $n=1$, which proves both strict
  linear alternatives in part~{\rm(2)}.

  Next, we prove the local exponential assertion on the invariant fixed-mass space.
  Put $\widetilde u=u-u^*$, $\widetilde v=v-(u^*)^\gamma$,
  $R=(-\Delta+I)^{-1}$, and
  $ c_*\coloneqq
  \frac{\chi_0\gamma(u^*)^{m+\gamma-1}}
  {(1+(u^*)^\gamma)^\beta}.$
  Then
  \begin{subequations}\label{E:main-perturbation-eq1-min}
    \begin{equation}\label{E:main-perturbation-eq1-M}
      \widetilde u_t
      =\mathcal L_{\mathrm M}\widetilde u+G_{\mathrm M}(\widetilde u),
    \end{equation}
    where
    \begin{equation}\label{E:L-M}
      \mathcal L_{\mathrm M}w\coloneqq\Delta w-c_*(R-I)w,
    \end{equation}
    \begin{align}\label{E:G-M}
      G_{\mathrm M}(\widetilde u)
      &\coloneqq-\chi_0\nabla\cdot\Big(
      \frac{(u^*+\widetilde u)^m}
      {(1+(u^*)^\gamma+\widetilde v)^\beta}\nabla\widetilde v\Big)
      +c_*(R-I)\widetilde u,
    \end{align}
    and
    \begin{equation}\label{E:main-perturbation-eq1-v-M}
      \widetilde v
      =R\left((u^*+\widetilde u)^\gamma-(u^*)^\gamma\right).
    \end{equation}
  \end{subequations}
  We equip $\mathcal L_{\mathrm M}$ with the Neumann domain
  $\mathcal D(\mathcal L_{\mathrm M})
  =\Big\{w\in W^{2,p}(\Omega)\,\Big|\,
  \frac{\partial w}{\partial n}=0\text{ on }\partial\Omega\Big\}$.

  As in the proof of Theorem \ref{T:linear-stability-thm}(1),
  choose $1/2<\vartheta<1$ and $p>1$ with $2\vartheta-N/p>1$. Let
  \[
    B_{\mathrm M}(r_{\mathrm M})
    \coloneqq\big\{z\in X_p^\vartheta\,\big|\,
    \Norm{z}_{X_p^\vartheta}<r_{\mathrm M}\big\}.
  \]
  Then
  $G_{\mathrm M}:B_{\mathrm M}(r_{\mathrm M})\to L^p(\Omega)$ is locally
  Lipschitz for $0<r_M\ll 1$, and
  \[
    G_{\mathrm M}(0)=0,
    \qquad
    DG_{\mathrm M}(0)=0,
    \qquad
    \Norm{G_{\mathrm M}(w)}_p
    =o\big(\Norm{w}_{X_p^\vartheta}\big).
  \]

  It can be verified directly that
  \[
    \mathcal L_{\mathrm M}\big(\mathcal D(\mathcal L_{\mathrm M})\cap X_2\big)
    \subset X_2,
    \qquad
    G_{\mathrm M}\big(B_{\mathrm M}(r_{\mathrm M})\cap X_2\big)\subset X_2.
  \]
  The mean projection commutes with $\Delta$ and $R$.
  Hence, by the bounded-perturbation theorem
  \cite[Theorem~1.3.2]{henry:81:geometric},
  the part $\mathcal L_{\mathrm M,2}$ of $\mathcal L_{\mathrm M}$
  in $X_2$ has spectrum $\{\sigma_{\mathrm M,n}:n\ge1\}$,
  and $-\mathcal L_{\mathrm M,2}$ is sectorial.
  If
  $\chi_0<\chi_{\mathrm M,\beta}^*(u^*)$, then
  \[
    \omega_{\mathrm M}
    \coloneqq\min_{n\ge1}(-\sigma_{\mathrm M,n})
    =\lambda_1\Big(1-\frac{c_*}{1+\lambda_1}\Big)>0.
  \]
  The commuting mean projection makes the resolvent and the fractional-power
  formula restrict to $X_2$; thus
  \cite[Theorem~1.4.8]{henry:81:geometric} identifies the fractional domain
  with $X_p^\vartheta\cap X_2$, up to an equivalent norm.
  By restricting \eqref{E:main-perturbation-eq1-M} to $X_2$ and
  applying the  arguments in Theorem \ref{T:linear-stability-thm}(1),
  there is  $\delta>0$ such that  for any $u_0$ with $\Norm{u_0-u^*}_\infty<\delta$
  and $\int_\Omega u_0(x)dx=|\Omega| u^*$,
  ~\eqref{E:exponential-decay-eq1} holds.   Theorem \ref{T:linear-stability-thm}(2) is thus proved.
\end{proof}

\section{Global stability in the non-minimal model and proof of Theorem \ref{T:global-stable}}\label{S:Stability}

In this section, we analyze the stability of the unique constant equilibrium
$(1,1)$, defined in~\eqref{E:equilibrium}, of the canonical non-minimal model
and prove Theorem~\ref{T:global-stable}.

We first present some lemmas. The first lemma concerns the stabilization of globally bounded solutions.

\begin{lemma}
  \label{L:nonlinear-stability-lm}
  Consider either canonical model and assume  $m\ge1$.
  Let $(u^*,v^*)$ be the relevant positive constant equilibrium and let
  $(u,v)$ be a globally defined bounded positive classical solution.

  \begin{itemize}
    \item[(1)]
    For any $\theta>0$,
    $\int_\Omega (u(t,x)-u^*)\big(u(t,x)^\theta-(u^*)^\theta\big)\,dx
    \longrightarrow0$
    implies $\Norm{u(t,\cdot)-u^*}_\infty\to0$.

    \item[(2)] Assume that $\chi_0<\chi_{\mathrm L,\beta}^*$ in the non-minimal model, and
    $\chi_0<\chi_{\mathrm M,\beta}^*(u^*)$ and $\int_\Omega u(t,x)dx=|\Omega| u^*$ in the minimal model.
    Then there is $\lambda>0$, independent of the solution, such that
    $\Norm{u(t,\cdot)-u^*}_\infty\to0$ implies \eqref{E:exponential-decay-eq1}
    for some $C>0$ depending on the solution.
  \end{itemize}
\end{lemma}

\begin{proof}
  (1) Assume the statement is not true. Then
  there are $\epsilon_0>0$ and $t_n\to\infty$ such that
  $\Norm{u(t_n,\cdot)-u^*}_\infty\ge\epsilon_0$.
  Lemma~\ref{L:Holder}(1) and the Arzel\`a--Ascoli theorem give, after passing
  to a subsequence, $u(t_n,\cdot)\to\widetilde u$ uniformly, with
  $\Norm{\widetilde u-u^*}_\infty\ge\epsilon_0$.
  Passing to the limit in the displayed integral yields a continuous
  nonnegative integrand with zero integral.
  The strict monotonicity of $s\mapsto s^\theta$ on $[0,\infty)$ then forces
  the limit to equal $u^*$, a contradiction. Therefore (1) holds.

  (2)  Since $\Norm{u(t,\cdot)-u^*}_\infty\to0$, there is a time $T>0$,
  depending on the solution, after which $(u,v)$ stays inside the local
  stability neighborhood of Theorem~\ref{T:linear-stability-thm}. Restarting
  the solution at that time and applying Theorem~\ref{T:linear-stability-thm}
  gives exponential convergence with a uniform rate $\lambda>0$; the shift
  $T$ only changes the prefactor, which therefore may depend on the solution,
  and~\eqref{E:exponential-decay-eq1} follows.
\end{proof}

The second lemma concerns a pointwise gradient estimate in terms of the oscillation of the source.

\begin{lemma}[Pointwise gradient estimate]\label{L:gradient-glaeser}
  There exists a constant $C_G(\Omega)>0$, depending only on $\Omega$, such
  that the following holds: for any $f\in C(\overline{\Omega})$ with
  $f\ge0$, if $v\in C^2(\overline\Omega)$ solves
  $$
  \begin{cases}
    \Delta v-v+f=0\quad &x\in \Omega\cr
    \frac{\partial v}{\partial n}=0\quad &x\in\partial \Omega,
  \end{cases}
  $$
  then
  \begin{equation}\label{E:glaeser-bound}
    |\nabla v(x)|^2
    \le C_G(\Omega)\, v(x)
    \Big(\sup_{\Omega} f-\inf_{\Omega} f\Big)
    \qquad\text{for all $x\in\overline\Omega$.}
  \end{equation}
  When $N=1$, \eqref{E:glaeser-bound} holds with $C_G=2$.
\end{lemma}

\begin{proof}
  {  Put ${\rm osc(f)}\coloneqq\sup_\Omega f-\inf_\Omega f$ and $w\coloneqq v-\inf_\Omega f$.
  By the maximum principle for elliptic equations, $\inf_\Omega f\le v\le\sup_\Omega f$.
  Hence, since $f\ge0$,
  \begin{equation}\label{E:gradient-estimate-eq0}
    0\le w\le\min\{{\rm osc}(f),v\},\qquad
    \Delta w=v-f\eqqcolon g,\quad |g|\le {\rm osc}(f)\ \text{in }\Omega,\qquad
    \frac{\p w}{\p n}=0\ \text{on }\p\Omega.
  \end{equation}
  If ${\rm osc}(f)=0$, then $v$ is constant and there is nothing to prove, so let
  ${\rm osc}(f)>0$. Since $\nabla w=\nabla v$ and $w\le v$, it suffices to prove
  \begin{equation}\label{E:gradient-estimate-eqw}
    |\nabla w(x)|^2\le C_G(\Omega)\,w(x)\, {\rm osc}(f)\qquad\text{for all }x\in\overline\Omega.
  \end{equation}

  \smallskip
  \noindent\emph{The case $N=1$.} Without loss of generality, let
  $\Omega=(0,L)$. Define $\tilde w(x)\coloneqq w(|x|)$ for $|x|\le L$ and
  extend $\tilde w$ to $\mathbb R$ by $2L$-periodicity.
  The Neumann conditions ensure that the first and second derivatives match at every reflection point. Thus $\tilde{w}\in C^2(\mathbb{R})$, $\tilde{w}\ge0$, and $|\tilde{w}''|\le\text{osc}(f)$.
  Taylor's formula gives, for
  $x\in[0,L]$ and every $t\in\mathbb R$,
  $0\le\tilde w(x+t)\le\tilde w(x)+t\,\tilde w'(x)+\tfrac{{\rm osc}(f)}{2}t^2$.
  The right-hand side is a nonnegative quadratic polynomial in $t$, so its
  discriminant is nonpositive, that is, $|w'(x)|^2\le2\, {\rm osc}(f) w(x)$, which is
  \eqref{E:gradient-estimate-eqw} with $C_G=2$.

  \smallskip
  \noindent\emph{The case $N\ge2$.} We divide the proof into three steps.

  \smallskip

  \noindent{\bf Step 1.} \emph{Even reflection across $\p\Omega$.}
  In this step, we define the even reflection $\bar w$ of $w$ across $\p\Omega$ and show that $\bar w$ is a weak solution of some elliptic equation.

  \smallskip

  First, let
  $\tilde d$ be the signed distance to $\p\Omega$, positive in $\Omega$, and
  $n$ the inward unit normal.
  For a given $\rho>0$, let
  $$
  \mathcal N_\rho\coloneqq\{x\in\mathbb R^N:|\tilde d(x)|<\rho\}.
  $$
  Since $\p\Omega$ is smooth,  by the arguments of  \cite[Lemma~14.16]{gilbarg.trudinger:01:elliptic},
  there is
  $\rho=\rho(\Omega)>0$ such that
  the function
  $\tilde d$ is smooth  on $\mathcal N_\rho$. Moreover,  every $x\in\mathcal N_\rho$
  has a unique nearest point $y(x)\in\p\Omega$ such that
  $$
  |\tilde d(x)|=|x-y(x)|,  \quad x=y(x)+\tilde d(x) n(y(x)),\quad {\rm and}\quad { \nabla \tilde d(y(x))}=n(y(x)).
  $$
  Define the normal reflection $R(x)$ by
  $$R(x)\coloneqq x-2\tilde d(x)\nabla\tilde d(x)=2y(x)-x.$$
  $R(\cdot)$ is then a smooth
  involution of $\mathcal N_\rho$ which exchanges
  $\mathcal N_\rho^+\coloneqq\mathcal N_\rho\cap\Omega$ and
  $\mathcal N_\rho^-\coloneqq\mathcal N_\rho\setminus\overline\Omega$. Note that
  $$R=\mathrm{id}\quad {\rm on} \quad \p\Omega\quad {\rm and}\quad
  D R=I-2\nabla\tilde d\otimes\nabla\tilde d-2\tilde d\,D^2\tilde d.
  $$
  It is clear that
  $DR|_{\p\Omega}$ equals $I-2n\otimes n$,  an orthogonal symmetric matrix with
  determinant $-1$, and shrinking $\rho$ if necessary, the
  matrices $DR$, ${ (DR)^{-1}}$, $D^2R$ and $|\det DR|^{-1}$ are bounded on
  $\mathcal N_\rho$.

  \smallskip

  Next, define the even reflection $\bar w$ of $w$ by
  $$
  \bar w\coloneqq w\ \text{on }\overline\Omega,\qquad
  \bar w\coloneqq w\circ R\ \text{on }\mathcal N_\rho^-.
  $$
  We claim that  $\bar w\in C^2(\Omega\cup\mathcal N_\rho)$,
  $\nabla\bar w=\nabla w$ on $\overline\Omega$, and $0\le\bar w\le {\rm osc}(f)$.  In fact, fix $y_0\in \p\Omega$ and let  $y'\mapsto\sigma(y')$ be local
  coordinates of $\p\Omega$ near $y_0$.
  Then, in the coordinates $(y',t)\mapsto\sigma(y')+t\,n(y')$ the reflection $R$ is
  $(y',t)\mapsto(y',-t)$, the Neumann condition reads $\p_t\hat w(y',0)=0$
  for $\hat w\coloneqq w(\sigma(y')+t\,n(y'))$, and $\bar w$ corresponds to
  $\hat w(y',|t|)$. Moreover,  it is not difficult to verify that the one-sided derivatives of $\hat w(y',|t|)$ of order at
  most two agree on $\{t=0\}$.  The claim then follows.

  \smallskip

  Now,  the change of variables
  $x=R(\xi)$ in the weak form of $\Delta w=g$ shows that
  $\bar w$ is a weak solution in $\Omega\cup\mathcal N_\rho$ of
  \begin{equation}\label{E:gradient-estimate-reflected}
    \p_i\big(a^{ij}\p_j\bar w\big)=\bar g,
  \end{equation}
  where { $a^{ij}\coloneqq\delta^{ij}$  and $\bar g\coloneqq g$ on $\overline\Omega$,}
  and $a^{ij}\coloneqq|\det DR|\,{\Big((DR)^{-1}\big((DR)^{-1}\big)^{T}\Big)^{ij}}$,
  $\bar g\coloneqq|\det DR|\,g\circ R$ on $\mathcal N_\rho^-$. Here $(a^{ij})$
  is symmetric, uniformly elliptic and bounded with constants depending only
  on $\Omega$, and continuous across $\p\Omega$, where $|\det DR|=1$ and
  {$(DR)^{-1}\big((DR)^{-1}\big)^{T}=I$.}  $(a^{ij})$ is also Lipschitz on every ball
  $B\subset\Omega\cup\mathcal N_\rho$, with constant depending only on
  $\Omega$;  and $|\bar g|\le C(\Omega){\rm osc}(f)$.

  \smallskip
  \noindent{\bf Step 2.} \emph{A two-scale gradient bound.} In this step, we show that  there are $r_\Omega>0$ and $C'=C'(\Omega)>0$ such that
  \begin{equation}\label{E:glaeser-two-scale}
    |\nabla w(x)|\le C' \Big(\frac{w(x)}{r}+r\times{\rm osc}(f)\Big)
    \qquad\text{for all }x\in\overline\Omega\text{ and }0<r\le r_\Omega.
  \end{equation}

  First, let $\rho>0$ be as in {\bf Step 1.}
  Put
  $r_\Omega\coloneqq\min\{1,\frac{\rho}{8}\}$. Since $\tilde d$ is { Lipschitz with Lipschitz constant $1$} and
  $\tilde d\ge0$ on $\overline\Omega$, we have
  $B(x,5r)\subset\Omega\cup\mathcal N_\rho$ for every $x\in\overline\Omega$
  and $0<r\le r_\Omega$. Fix such $x$ and $r$, and a $p$ with
  $1<p<\frac{N}{N-2}$ { ($\frac{N}{N-2}=\infty$ when $N=2$).}  Since $\bar w\ge0$ solves
  \eqref{E:gradient-estimate-reflected} with equality, it is both a
  subsolution and a supersolution, and the local maximum principle and the
  weak Harnack inequality
  \cite[Theorems~8.17 and~8.18]{gilbarg.trudinger:01:elliptic}, applied on
  the domain $B(x,5r)$ with the balls $B(x,2r)$, $B(x,4r)$ of those theorems,
  $q=2N$, and chained through $\Norm{\bar w}_{L^p(B(x,2r))}$, give
  \begin{equation}\label{E:gradient-estimate-eq1}
    \sup_{B(x,r)}\bar w\le C_1\Big(\inf_{B(x,r)}\bar w+r^2{\rm osc}(f)\Big)
    \le C_1\big(\bar w(x)+r^2{\rm osc}(f)\big)
  \end{equation}
  for some constant $C_1=C_1(\Omega)$ independent of $x$ and $r$.

  Next, put $u(\xi)\coloneqq\bar w(x+r\xi)$,
  $\hat a^{ij}(\xi)\coloneqq a^{ij}(x+r\xi)$ and
  $\hat g(\xi)\coloneqq r^2\bar g(x+r\xi)$ for $\xi\in B(0,1)$. Then
  $\p_i(\hat a^{ij}\p_ju)=\hat g$ weakly in $B(0,1)$, the ellipticity
  constants are unchanged, the H\"older norms of $\hat a^{ij}$ do not exceed
  those of $a^{ij}$ because $r\le1$, $|\hat g|\le C(\Omega)r^2{\rm osc}(f)$, and
  $u\in C^2\subset C^{1,\theta}$. The interior $C^{1,\theta}$ estimate
  \cite[Theorem~8.32]{gilbarg.trudinger:01:elliptic} on $B(0,1/2)\subset\subset B(0,1)$
  gives
  \begin{equation}\label{E:gradient-estimate-eq2}
    |\nabla\bar w(x)|\le C_2\Big(\frac1r\sup_{B(x,r)}\bar w+{ r \times {\rm osc}(f)}\Big),
  \end{equation}
  with $C_2$ independent of $x$ and $r$. \eqref{E:glaeser-two-scale} then follows from
  \eqref{E:gradient-estimate-eq1}, \eqref{E:gradient-estimate-eq2} and
  $\nabla\bar w=\nabla w$ on $\overline\Omega$, with $C'\coloneqq C_2(C_1+1)$.

  \smallskip
  \noindent{\bf Step 3.} \emph{Choice of $r$.}
  In this step, we show that \eqref{E:gradient-estimate-eqw} holds by properly choosing $r$ in
  \eqref{E:glaeser-two-scale}.

  \smallskip

  Fix $x\in\overline\Omega$. If
  $w(x)=0$, then $x$ is a minimum point of $\bar w\ge0$ in the open set
  $\Omega\cup\mathcal N_\rho$, so $\nabla w(x)=\nabla\bar w(x)=0$ and
  \eqref{E:gradient-estimate-eqw} holds trivially. Assume
  $w(x)>0$. Since $C'$ depends neither on $x$ nor on $r$, we may choose $r$
  depending on $x$. If $\sqrt{\frac{w(x)}{{\rm osc}(f)}}\le r_\Omega$, take $r=\sqrt{\frac{w(x)}{{\rm osc}(f)}}$
  in \eqref{E:glaeser-two-scale} to obtain $|\nabla w(x)|^2\le4(C')^2\,w(x){\rm osc}(f)$.
  If $\sqrt{\frac{w(x)}{{\rm osc}(f)}}>r_\Omega$, take $r=r_\Omega$; since then
  ${\rm osc}(f)<\frac{w(x)}{r_\Omega^2}$ and~${w(x)\le {\rm osc}(f)}$,
  $$
  |\nabla w(x)|^2\le2(C')^2\Big(\frac{w(x)^2}{r_\Omega^2}+r_\Omega^2({\rm osc}(f))^2\Big)
  \le2(C')^2\big(r_\Omega^{-2}+1\big)w(x){\rm osc}(f).
  $$
  As $r_\Omega\le1$, \eqref{E:gradient-estimate-eqw} holds in both cases
  with $C_G(\Omega)\coloneqq2(C')^2(r_\Omega^{-2}+1)$. This proves
  Lemma~\ref{L:gradient-glaeser}.}
\end{proof}

For ease of reference, we also recall the following lemma.

\begin{lemma}\label{L:nonincreasing}
  Assume $\chi_0\le 0$.
  For any $u_0$ satisfying~\eqref{E:initial-cond-PE}, the following hold.
  \begin{itemize}
    \item[(1)] Consider the non-minimal model. If $\bar{u}(t_0)> u^* =1$ for some $t_0 \in { (0, T_{\max}(u_0))}$,
    then $\bar{u}(t)$ is nonincreasing on $(0,t_0]$.

    \item[(2)]  Consider the minimal model.  Then $\bar{u}(t)$ is
    nonincreasing on $(0,T_{\max}(u_0))$.
  \end{itemize}
\end{lemma}

\begin{proof}
  It follows from Lemma~3.1 of~\cite{chen.ruau.ea:26:chemotaxis*1}.
\end{proof}

\subsection{Proof of Theorem~\ref{T:global-stable} under condition (i)}\label{SS:global-stable}

\begin{proof}[Proof of Theorem~\ref{T:global-stable} under condition (i)]
  Assume that $m\ge 1$ and $\chi_0 \le 0$. Let $(u(t,x), v(t,x))$ be a globally
  defined positive bounded classical solution of~\eqref{E:main-PE}. For any $t >
  0$, let $\underline{u}(t)$ and $\bar{u}(t)$ be as defined
  in~\eqref{E:super-inf-u}, and let $\bar{v}(t)$ be as defined by
  \begin{equation}\label{E:super-v}
    \overline{v}(t)\coloneqq\max_{x\in\overline\Omega} v(t,x).
  \end{equation}
  By the second equation in~\eqref{E:main-PE} and the comparison principle for
  elliptic equations, we have $\bar{v}(t) \le \bar{u}^\gamma
  (t)$. We prove the theorem by proving three claims.

  \smallskip

  \medskip\noindent\textbf{Claim~1.} $\displaystyle
  \bar{u}_\infty^*\coloneqq\limsup_{t\to\infty} \bar{u}(t)\le u^* =
  1$. \medskip

  In fact, if $\bar{u}(t) \le u^*$ for all $t > 0$,
  then Claim~1 holds automatically. If there is
  $t_0 > 0$ such that $\bar{u}(t_0) > u^*$, then, by
  Lemma~\ref{L:nonincreasing}, there is $t_\infty \in (t_0, \infty]$ such that
  $\bar{u}(t) > u^*$ for all $t \in (0, t_\infty)$ and $\bar{u}(t)$ is
  nonincreasing on $[0,t_\infty)$. Moreover, if $t_\infty < \infty$, then
  $\bar{u}(t_\infty) = u^*$.

  In the case that $t_\infty=\infty$, $\bar{u}_\infty^* = \lim_{t\to\infty} \bar
  u(t)$. Take any sequence $t_n\to\infty$. By Lemma~\ref{L:Holder},
  without loss of generality, we may assume that $u_\infty(t,x) =
  \lim_{n\to\infty} u(t_n+t,x)$ and $v_\infty(t,x) = \lim_{n\to\infty}
  v(t+t_n,x)$ exist, and $(u_\infty(t,x), v_\infty(t,x))$ is a classical
  solution of~\eqref{E:main-PE} defined for all $t \in \mathbb{R}$. Notice that
  \begin{align*}
    \bar{u}_\infty(t) = \max_{x\in\overline{\Omega}}u_\infty(t,x)
    = \bar{u}_\infty^*\ge u^*,\quad \text{for all $t\in\mathbb{R}$.}
  \end{align*}
  Assume that $\bar{u}^*_\infty>u^*$. For a given $t\in\mathbb{R}$, let
  $x_\infty(t)\in\overline\Omega$ be such that $u_\infty(t,x_\infty(t)) =
  \bar{u}_\infty(t)$.  Let $B(y, \epsilon) \coloneqq \left\{ x \in \Omega \,:\,
  |x-y| < \epsilon \right\}$.   It is then not difficult to see that there is
  $\epsilon>0$ such that
  \begin{equation}\label{E:global-stability-proof-PE3}
    \p_t u_\infty < \Delta u_\infty-\chi_0 m\frac{u_\infty^{m-1}}{(1+v_\infty)^\beta}\nabla u_\infty\cdot\nabla v_\infty,\quad \text{for all $(t,x) \in [t-\epsilon, t+\epsilon] \times B(x_\infty(t),\epsilon)$},
  \end{equation}
  where the time neighborhood is taken on both sides of the chosen time, so
  that $(t,x_\infty(t))$ is an ordinary lateral point of the cylinder rather
  than a terminal corner.
  Applying the maximum principle for parabolic equations
  to~\eqref{E:global-stability-proof-PE3} (precisely, applying [Chapter 2,
  Theorem 1] in~\cite{friedman:64:partial}
  to~\eqref{E:global-stability-proof-PE3} in the case that $x_\infty(t) \in
  \Omega$ and applying Theorem 2 in~\cite{friedman:58:remarks}
  to~\eqref{E:global-stability-proof-PE3} in the case that $x_\infty(t) \in
  \partial \Omega$), we must have
  \begin{equation}\label{E:proof-global-stable-eq3}
    u_\infty(\tau, x) = \bar{u}_\infty^*,\quad
    \text{for all $\tau\in [t-\epsilon, t+\epsilon]$ and $x \in B(x_\infty(t),\epsilon)\cap \Omega$.}
  \end{equation}
  Repeating the arguments of~\eqref{E:proof-global-stable-eq3}, we have
  \[
    u_\infty(\tau,x) = \bar{u}_\infty^*, \quad
    \text{for all $\tau\in [t-\epsilon, t+\epsilon]$, $x \in B(y,\epsilon)\cap \Omega$, and $y \in B(x_\infty(t),\epsilon)$.}
  \]
  Repeating this process, we have $u_\infty(\tau,x) = \bar{u}_\infty^*$ for all
  $\tau \in [t-\epsilon, t+\epsilon]$ and $x \in \Omega$. Note that the above propagation
  to the whole domain follows from a covering argument using the connectedness
  of $\Omega$ and the time-continuity of $u_\infty$ on $[t-\epsilon, t+\epsilon]$.
  Therefore, $u_\infty(t,x) = \bar{u}_\infty^*$ for all $t\in\mathbb{R}$ and
  $x\in\Omega$. We then must have $\bar{u}_\infty^* = u^*$ (since $(u^*, v^*)$
  is the only positive constant solution of~\eqref{E:main-PE}), which
  contradicts the assumption that $\bar{u}^*_\infty>u^*$. Hence,
  Claim~1 holds.

  In the case that $t_\infty<\infty$, $\bar{u}(t_\infty)=u^*$. If there is
  $t'_0 > t_\infty$ such that $\bar{u}(t'_0) > u^*$, then by Lemma
  \ref{L:nonincreasing}, $\bar{u}(t) > u^*$ for $t \in (0,t'_0]$, in
  particular, $\bar{u}(t_\infty) > u^*$, which is a contradiction. Hence,
  $\bar{u}(t) \le u^*$ for all $t \ge t_\infty$,
  and Claim~1 holds. Claim~1 is thus proved.

  \smallskip

  \medskip\noindent\textbf{Claim~2.} $\displaystyle \lim_{t\to\infty}\int_\Omega
  (u(t,x)-u^*) = 0$.

  \smallskip

  In fact, without loss of generality, we assume that $u(t,x) \not\equiv u^*$.
  Then by the arguments of Claim~1, there is $t^* \in (0, \infty]$ such that $0
  < \limsup_{t \to t^*} \bar{u}(t) \le u^*$. Moreover, if $t^* < \infty$,
  $\bar{u}(t) < u^*$ for $t > t^*$, and in the case $t^* = \infty$, $\limsup_{t
  \to \infty} \bar{u}(t) = u^*$.

  In the case that $t^* < \infty$, we have $u(t,x) < u^*$ for $t > t^*$. Note
  that
  \begin{align*}
    \frac{d}{dt}\int_\Omega( u(t,x)-u^*)
    =  \int_\Omega u(t,x)((u^*)^\alpha-u^\alpha(t,x))
    \quad \text{for all $t > t^*$.}
  \end{align*}
  By Theorem~\ref{T:persistence}, $\liminf_{t\to\infty} \inf_{x\in\Omega}
  u(t,x)>0$. Then there are~${0 < m_0 < M_1 < \infty}$ such that the following
  inequalities hold for all $t \ge t^*$:
  \begin{align*}
    & \frac{d}{dt}\int_\Omega (u-u^*) = \int_\Omega u ((u^*)^\alpha-u^\alpha) \le M_1\int_\Omega (u^*-u) = -M_1\int_\Omega (u-u^*) \quad \text{and} \\
    & \frac{d}{dt}\int_\Omega (u-u^*) = \int_\Omega u ((u^*)^\alpha-u^\alpha) \ge m_0\int_\Omega (u^*-u) = -m_0\int_\Omega (u-u^*).
  \end{align*}
  It then follows from the comparison principle for scalar ODEs that
  $\lim_{t\to\infty} \int_\Omega (u(t,x)-u^*)=0$.

  In the case where $t^* = \infty$, consider any sequence $t_n \to \infty$.
  Without loss of generality, we can assume that both limits below exist:
  \begin{align*}
    u_\infty(t,x) = \lim_{n \to \infty} u(t + t_n, x) \quad \text{and} \quad
    v_\infty(t,x) = \lim_{n \to \infty} v(t + t_n, x)
  \end{align*}
  and $(u_\infty,v_\infty)$ is a classical solution of~\eqref{E:main-PE} defined
  for all $t\in\mathbb{R}$. By Claim~1, $\bar{u}_\infty(t)= u^*$ for all $t \in
  \mathbb{R}$, so $u_\infty\le u^*$ everywhere. By Theorem~\ref{T:persistence} again, $\displaystyle \inf_{t \in
  \mathbb{R}, x \in \Omega} u_\infty(t,x) > 0$. Consider the mass deficit
  \[
    q(t)\coloneqq\int_\Omega \left(u^*-u_\infty(t,x)\right)\,dx\ \ge\ 0,
  \]
  which is bounded on the whole time axis, and write the divided difference
  \[
    D(t,x)\coloneqq\frac{(u^*)^\alpha-u_\infty^\alpha(t,x)}{u^*-u_\infty(t,x)},
  \]
  which extends continuously to $u_\infty=u^*$ with the value $D=\alpha$ at
  $u_\infty=u^*=1$. Since $u_\infty$ is uniformly positive, there is $c_0>0$
  such that $u_\infty D\ge c_0$ on $\mathbb{R}\times\Omega$, and
  \begin{align*}
    q'(t)
    = -\int_\Omega u_\infty \left( (u^*)^\alpha- u_\infty^\alpha \right )
    = -\int_\Omega u_\infty D\,(u^* - u_\infty)
    \le -c_0\, q(t).
  \end{align*}
  Hence, for every $s<t$,
  \[
    0\ \le\ q(t)\ \le\ q(s)\,e^{-c_0(t-s)}\ \longrightarrow\ 0
    \quad\text{as $s\to-\infty$,}
  \]
  since $q(s)$ is bounded. Therefore $q\equiv 0$, so $u_\infty\equiv u^*$
  and $\displaystyle \lim_{n\to\infty}\int_\Omega(u(t_n,x)-u^*) = 0$. Claim~2
  then follows.

  \smallskip

  \medskip\noindent\textbf{Claim~3.} $\lim_{t\to\infty} \Norm{u(t,\cdot) -
  u^*}_\infty = 0$.

  \smallskip

  Assume that Claim~3 does not hold. By Claim~1, there are
  $\epsilon_0 > 0$, $x_n \in \overline{\Omega}$, and $t_n \to \infty$ such that
  $u(t_n, x_n) - u^* \le - \epsilon_0$. Then there is a constant $\delta_0 > 0$
  such that $u(t_n,x)-u^*<-\frac{\epsilon_0}{2}$ for all $x \in
  B(x_n,\delta_0)\cap \Omega$. This implies that $\lim_{n\to\infty} \int_\Omega
  (u(t_n,x)-u^*) < 0$, which contradicts Claim~2. Hence, Claim~3
  holds.

  Theorem~\ref{T:global-stable} under condition (i) then follows from Claim 3 and Lemma \ref{L:nonlinear-stability-lm}.
\end{proof}

\subsection{Proof of Theorem~\ref{T:global-stable} under condition (ii)}\label{SS:global-stable-ii}

\begin{proof}[Proof of Theorem~\ref{T:global-stable} under condition (ii)]
  Let $(u(t,x),v(t,x))$ be a globally defined positive bounded solution
  of~\eqref{E:main-PE} in the canonical non-minimal model.
  {Let $F(\cdot)$ be defined as in \eqref{E:F-def}, that is,
  $$
  F(t) \coloneqq \int_\Omega h_m\left(u(t,x)\right)\,dx,
  \quad {\rm where}\quad
  h_m(s)\coloneqq \int_1^s\big(1-\tau^{-(2m-1)}\big)\,d\tau.
  $$}
  We will prove the following estimates:
  \begin{equation}\label{E:new-global-stab-eq0}
    F(t)\ge 0\quad \text{for all $t\ge 0$,}
  \end{equation}
  \begin{align}\label{E:new-global-stab-eq1}
    F'(t)
    & \le -\Big(1-\frac{\chi_0^2}{\big(\chi_{\mathrm L,\beta}^{**,i}\big)^2-\epsilon}\Big)
    \int_\Omega (u-1)(u^\alpha-1),\quad i=1,2,\quad t\gg 1,
  \end{align}
  { for any $0<\epsilon<\min\{(\chi_{\mathrm L,\beta}^{**,1})^2, (\chi_{\mathrm L,\beta}^{**,2})^2\}$,}  and
  \begin{equation}\label{E:new-global-stab-eq2}
    D(t)\coloneqq \int_\Omega (u(t,x)-1)\big(u(t,x)^\alpha-1\big)
    \to 0\quad \text{as $t\to\infty$.}
  \end{equation}

  The key point is \eqref{E:new-global-stab-eq2}.
  We start with \eqref{E:new-global-stab-eq0}. Indeed, $h_m(1)=0$ and
  \[
    h_m'(s)=1-s^{-(2m-1)}
    \quad
    \begin{cases}
      <0, & 0<s<1, \\
      >0, & s>1,
    \end{cases}
  \]
  hence $h_m$ attains its minimum at $1$ and $h_m\ge 0$ on $(0,\infty)$.

  \smallskip

  Next, we prove \eqref{E:new-global-stab-eq1} for the case $i=1$.
  Since $h_m''(s)=(2m-1)s^{-2m}$, we have
  \begin{align}\label{E:global-stab-1}
    F'(t)
    & =\int_\Omega h_m'(u)\,u_t
    =-(2m-1)\int_\Omega \frac{|\nabla u|^2}{u^{2m}}
    +\chi_0(2m-1)\int_\Omega \frac{\nabla u\cdot\nabla v}{u^{m}(1+v)^\beta} \nonumber \\
    & \quad -\int_\Omega \left(u-u^{-(2m-2)}\right)(u^\alpha-1)\nonumber \\
    & \le \frac{\chi_0^2(2m-1)}{4}\int_\Omega \frac{|\nabla v|^2}{(1+v)^{2\beta}}
    -\int_\Omega (u-1)(u^\alpha-1).
  \end{align}

  Here we have applied Young's inequality in the form $xy\le x^2+
  \tfrac14y^2$, and used that
  \[
    \left(u-u^{-(2m-2)}-(u-1)\right)(u^\alpha-1)
    =\left(1-u^{-(2m-2)}\right)(u^\alpha-1)\ge 0.
  \]
  Recall that $\tilde{\beta}$ is defined in~\eqref{E:b-star}. By the elliptic
  equation common to~\eqref{E:main-PE-L} and~\eqref{E:main-PE-M}, we have
  \begin{align*}
    0
    & = \int_\Omega \frac{v-1}{(1+v)^{\tilde{\beta}}}
    \left(\Delta v-v+u^\gamma\right) \\
    & = -\int_\Omega \frac{1+v-\tilde{\beta}(v-1)}{(1+v)^{\tilde{\beta}+1}}|\nabla v|^2
    -\int_\Omega\frac{(v-1)^2}{(1+v)^{\tilde{\beta}}}
    +\int_\Omega\frac{(v-1)(u^\gamma-1)}{(1+v)^{\tilde{\beta}}} \\
    & \le -\int_\Omega \frac{1+v-\tilde{\beta}(v-1)}{(1+v)^{\tilde{\beta}+1}}|\nabla v|^2
    +\frac14\int_\Omega\frac{(u^\gamma-1)^2}{(1+v)^{\tilde{\beta}}}.
  \end{align*}
  It then follows that
  \begin{align}\label{E:global-stab-m4}
    \int_\Omega\frac{|\nabla v|^2}{(1+v)^{2\beta}}
    & \le \frac{1}{4(1+\tilde{\beta})}\int_\Omega
    \frac{(u^\gamma-1)^2}{(1+v)^{\tilde{\beta}}}
    \le \frac{1}{4(1+\tilde{\beta})}\int_\Omega (u^\gamma-1)^2.
  \end{align}
  Here we have used the choice $\tilde{\beta} = \min\{1,\,2\beta-1\}_+$ and the
  pointwise inequality
  \[
    (1+\tilde\beta)(1+v)^{-2\beta}
    \le \left(1+\tilde\beta+(1-\tilde\beta)v\right)
    (1+v)^{-(\tilde\beta+1)}.
  \]
  This covers the three
  regimes $0 \le \beta < 1/2$ (where $\tilde{\beta} = 0$), $1/2 \le \beta \le
  1$ (interpolating between $\tilde{\beta} = 0$ and $1$), and $\beta \ge 1$
  (where $\tilde{\beta} = 1$). This, together with~\eqref{E:global-stab-1}
  and~\eqref{E:global-stab-m4}, implies that
  \begin{align}\label{E:global-stab-3}
    F'(t)
    & \le \frac{\chi_0^2(2m-1)}{16(1+\tilde{\beta})}
    \int_\Omega (u^\gamma-1)^2
    -\int_\Omega (u-1)(u^\alpha-1).
  \end{align}
  We verify the pointwise estimate needed here. Fix $y>0$, set
  $a=\min\{1,y\}$ and $b=\max\{1,y\}$, and note that
  $r\mapsto\abs{y^r-1}$ is nondecreasing for $r>0$ and
  $(y-1)(y^\alpha-1)=\abs{y-1}\abs{y^\alpha-1}$.
  If $0<\alpha<1$, then $\gamma\le(\alpha+1)/2$, and the Cauchy--Schwarz
  inequality gives
  \[
    \abs{y^\alpha-1}\abs{y-1}
    =\alpha\Big(\int_a^b s^{\alpha-1}\,ds\Big)(b-a)
    \ge\frac{4\alpha}{(\alpha+1)^2}
    \abs{y^{(\alpha+1)/2}-1}^2
    \ge\frac{4\alpha}{(\alpha+1)^2}\abs{y^\gamma-1}^2.
  \]
  If $\alpha\ge1$ and $0<\gamma\le1$, then
  $\abs{y^\gamma-1}\le\abs{y-1}\le\abs{y^\alpha-1}$.
  Finally, if $\alpha\ge1$ and $\gamma>1$, then $\alpha\ge2\gamma-1$, and
  another application of the Cauchy--Schwarz inequality gives
  \[
    \abs{y^\gamma-1}^2
    \le\gamma^2(b-a)\int_a^b s^{2\gamma-2}\,ds
    =\frac{\gamma^2}{2\gamma-1}\abs{y-1}\abs{y^{2\gamma-1}-1}
    \le\frac{\gamma^2}{2\gamma-1}\abs{y-1}\abs{y^\alpha-1}.
  \]
  These three cases and~\eqref{E:power-diff-constant} show that
  \[
    (y^\gamma-1)^2
    \le C_{\alpha,\gamma}(y-1)(y^\alpha-1)
    \qquad\text{for every $y>0$.}
  \]
  Applying this estimate with $y=u(t,x)$ in~\eqref{E:global-stab-3} yields
  \begin{align}\label{E:global-stab-4}
    F'(t)
    & \le -\Big(1-\frac{\chi_0^2(2m-1)C_{\alpha,\gamma}}
    {16(1+\tilde{\beta})}\Big)\int_\Omega (u-1)(u^\alpha-1).
  \end{align}
  \eqref{E:new-global-stab-eq1} {with $i=1$}  then follows.

  \smallskip

  Now, we prove
  \eqref{E:new-global-stab-eq1} {with $i=2$}. If
  $0<\chi_0<\chi_{\mathrm L,\beta}^{**,1}$, the basic estimate just proved
  applies. Otherwise, the hypothesis
  $0<\chi_0<\max\{\chi_{\mathrm L,\beta}^{**,1},\chi_{\mathrm L,\beta}^{**,2}\}$
  gives $\chi_0<\chi_{\mathrm L,\beta}^{**,2}$, and by the definition of
  $\chi_{\mathrm L,\beta}^{**,2}$, we only need to prove that
  \eqref{E:new-global-stab-eq1}  with $i=2$ and $\beta\ge 1$ holds.
  Since $\chi_{\mathrm L,\beta}^{**,2}\le\bar\chi_{\mathrm L,\beta}$,
  Theorem~\ref{T:persistence}{\rm(2)} gives
  $\liminf_{t\to\infty}\inf_{x\in\Omega} v(t,x)\ge \underline v_{\mathrm L}$.
  After a time-shift, we may assume that
  \begin{equation}\label{E:new-global-stab-eq3}
    \frac{\chi_0^2(2m-1)}
    {4(1+v(t,x))^{2\beta}}   = \frac{ 4\chi_0^2}
    {\frac{16(1+v(t,x))^{2\beta}}{2m-1}}
    \le  \frac{4\chi_0^2}
    {C_{\alpha,\gamma}\big((\chi_{\mathrm L,\beta}^{**,2})^2-\epsilon\big)}
    \quad \forall\, t\ge 0,\ x\in\Omega,
  \end{equation}
  where $0<\epsilon<(\chi_{\mathrm L,\beta}^{**,2})^2-\chi_0^2$.
  By~\eqref{E:global-stab-1}
  and~\eqref{E:new-global-stab-eq3},
  \begin{align}\label{E:new-global-stab-1}
    F'(t)
    & \le\frac{4\chi_0^2}
    {C_{\alpha,\gamma}\big((\chi_{\mathrm L,\beta}^{**,2})^2-\epsilon\big)} \int_\Omega |\nabla v|^2
    -\int_\Omega (u-1)(u^\alpha-1).
  \end{align}
  By~\eqref{E:global-stab-m4} with $\tilde\beta=0$,
  \begin{equation}\label{E:new-global-stab-m4}
    \int_\Omega  |\nabla v|^2
    \le \frac14\int_\Omega (u^\gamma-1)^2.
  \end{equation}
  By \eqref{E:new-global-stab-eq3} and \eqref{E:new-global-stab-1} and the arguments of
  \eqref{E:global-stab-4},
  \eqref{E:new-global-stab-eq1}  with $i=2$ and $\beta\ge 1$ holds.

  \smallskip

  Finally, we prove \eqref{E:new-global-stab-eq2}.
  By~\eqref{E:new-global-stab-eq0} and \eqref{E:new-global-stab-eq1}, we have
  \begin{align}\label{E:global-stab-5}
    \Big(1-\frac{\chi_0^2}{\big(\chi_{\mathrm L,\beta}^{**,i}\big)^2-\epsilon}\Big)
    \int_{t_0}^t \int_\Omega (u-1)(u^\alpha-1)
    \le F(t_0) \quad \text{for all $t \ge t_0\gg 1$.}
  \end{align}
  By the assumption $0<\chi_0<\max\{\chi_{\mathrm L,\beta}^{**,1},\chi_{\mathrm L,\beta}^{**,2}\}$ and
  \eqref{E:global-stab-5}, we have $\int_{t_0}^\infty D(t)\,dt<\infty$ for $t_0\gg 1$. Moreover,
  since $(u,v)$ is bounded, Lemma~\ref{L:Holder}(1) implies that $u\in
  C^{\theta/2,\theta}([1,\infty)\times\overline\Omega)$ for some
  $\theta\in(0,1)$, hence {$D(t)$ is uniformly
  continuous on $[t_0,\infty)$ for any $t_0\ge 1$.}  Therefore, the uniform continuity of $D$ on
  $[t_0,\infty)$, together with $\int_{t_0}^\infty D(t)\,dt<\infty$ for $t_0\gg 1$, yields
  $D(t)\to 0$ as $t\to\infty$, which is \eqref{E:new-global-stab-eq2}.
  This, together with Lemma
  \ref{L:nonlinear-stability-lm}, implies Theorem~\ref{T:global-stable} under condition (ii).
\end{proof}

\subsection{Proof of Theorem \ref{T:global-stable} under condition (iii)}

In this subsection, we prove Theorem~\ref{T:global-stable} under condition
(iii). The proof follows the rectangle method of
Galakhov--Salieva--Tello~\cite[Section~3]{galakhov.salieva.ea:16:on}, adapted to
our normalization.
For $\beta > 0$, we must also control the term involving $\abs{\nabla V}^2$
(where $V$ is defined below) produced by differentiating $(1+V)^{-\beta}$.

\begin{proof}[Proof of Theorem~\ref{T:global-stable} under condition (iii)]
  Assume that $m\ge 1$, $\gamma\ge 1$, $\alpha+1\ge m+\gamma$, and
  $0<\chi_0 { <\max\{\chi_{\mathrm L,\beta}^{**,3}, \chi_{\mathrm L,\beta}^{**,4}\}}$.

  \smallskip

  First, let $(u(t,x),v(t,x))$ be a globally defined positive bounded classical
  solution of~\eqref{E:main-PE} in the canonical non-minimal model.
  Note that, if $\chi_{\mathrm L,\beta}^{**,3}<\chi_{\mathrm L,\beta}^{**,4}$ (the enhanced
  branch), then we must have $\beta\ge 1$,
  $\chi_0<\bar\chi_{\mathrm L,\beta}$, and
  $\chi_0<(1+\underline v_{\mathrm L})^\beta
  \chi_{\mathrm L,\beta}^{**,3}$. In this case, for  $\epsilon\in(0,\underline v_{\mathrm L})$ sufficiently small, it holds that
  $\chi_0<(1+\underline v_{\mathrm L}-\epsilon)^\beta
  \chi_{\mathrm L,\beta}^{**,3}$. Moreover, since
  $\chi_0<\bar\chi_{\mathrm L,\beta}=1/(2\Theta_{\beta-1})<1/\Theta_{\beta-1}$
  and $m,\beta\ge 1$, Theorem~\ref{T:persistence}(2) applies and gives
  $\liminf_{t\to\infty}\inf_{x\in\Omega} v(t,x)
  \ge(1-\chi_0\Theta_{\beta-1})^{\gamma/\alpha}>2^{-\gamma/\alpha}
  =\underline v_{\mathrm L}$; hence there exists $t_\epsilon>0$ such that
  \[
    v(t,x)\ge v_{\mathrm{lb}}\coloneqq\underline v_{\mathrm L}-\epsilon>0
    \qquad \text{for all $t\ge t_\epsilon$ and all $x\in\Omega$,}
  \]
  and $\widehat\chi_0=\chi_0(1+v_{\mathrm{lb}})^{-\beta}$. Set
  $U(t,x)\coloneqq u(t_\epsilon+t,x)$ and $V(t,x)\coloneqq v(t_\epsilon+t,x)$,
  and relabel $U(0,\cdot)=u(t_\epsilon,\cdot)$ as $u_0$. In the basic branch $\chi_{\mathrm L,\beta}^{**,4}\le\chi_{\mathrm L,\beta}^{**,3}$, no such lower bound on $v$ is available; in that case set $t_\epsilon=0$, $U\coloneqq u$, $V\coloneqq v$, and
  use $\widehat\chi_0=\chi_0$ instead.
  The pair $(U,V)$ solves
  \begin{equation}\label{E:rectangle-system}
    \begin{dcases}
      U_t = \Delta U-\chi_0\nabla\cdot\!\Big(\frac{U^m}{(1+V)^\beta}\nabla V\Big)+U(1-U^\alpha), & x\in\Omega, \\
      0=\Delta V-V+U^\gamma,                                                                                                    & x\in\Omega, \\
      U(0,\cdot)=u_0,                                                                                                           & x\in\Omega.
    \end{dcases}
  \end{equation}
  Set $\overline{U}(t)\coloneqq \sup_{x\in\Omega} U(t,x)$ and
  $\underline{U}(t)\coloneqq \inf_{x\in\Omega} U(t,x)$.
  By the
  maximum principle for elliptic equations,
  \begin{equation}\label{E:rectangle-V-bounds}
    \underline{U}(t)^\gamma \le V(t,x)\le \overline{U}(t)^\gamma
    \qquad \text{for all $t>0$ and $x\in\Omega$.}
  \end{equation}
  By Lemma~\ref{L:gradient-glaeser} with $f=U^\gamma$ and $V\le1+V$,
  \begin{equation}\label{E:rectangle-gradV-glaeser}
    |\nabla V(t,x)|^2
    \le C_G(\Omega)\,\big(1+V(t,x)\big)
    \left(\overline{U}(t)^\gamma-\underline{U}(t)^\gamma\right)
    \qquad\text{for all $t>0$ and $x\in\overline\Omega$.}
  \end{equation}

  Next,
  we claim that for a.e.\
  $t\in(0,\infty)$,
  \begin{equation}\label{E:rectangle-ode-ineq}
    \begin{cases}
      \overline{U}'(t)
      \le \big(1+\beta C_G(\Omega)\big)\widehat\chi_0\, \overline{U}(t)^m\left(\overline{U}(t)^\gamma-\underline{U}(t)^\gamma\right)
      +\overline{U}(t)\left(1-\overline{U}(t)^\alpha\right),\\[4pt]
      \underline{U}'(t)
      \ge\widehat\chi_0 \underline{U}(t)^m\left(\underline{U}(t)^\gamma-\overline{U}(t)^\gamma\right)
      +\underline{U}(t)\left(1-\underline{U}(t)^\alpha\right).
    \end{cases}
  \end{equation}
  In fact, it is easy to see that $\overline{U}$ and $\underline{U}$ are
  locally Lipschitz on $(0,\infty)$; hence they are differentiable at a.e.
  $t\in(0,\infty)$. Moreover,
  for $t>0$, define the maximizing and minimizing sets
  \[
    A^+(t)\coloneqq \left\{x\in\overline{\Omega}: U(t,x)=\overline{U}(t)\right\},
    \qquad
    A^-(t)\coloneqq \left\{x\in\overline{\Omega}: U(t,x)=\underline{U}(t)\right\}.
  \]
  Since $U_t$ is continuous on $(0,\infty)\times\overline{\Omega}$, a
  standard envelope argument yields
  \begin{equation}\label{E:envelope-min}
    \limsup_{h\downarrow 0}\frac{\overline{U}(t+h)-\overline{U}(t)}{h}
    \ge  \max_{x\in A^+(t)} U_t(t,x)
    \quad \text{and} \quad
    \liminf_{h\downarrow 0}\frac{\overline{U}(t)-\overline{U}(t-h)}{h}
    \le \min_{x\in A^+(t)} U_t(t,x).
  \end{equation}
  Rewrite the first equation
  in~\eqref{E:rectangle-system} in the form
  \begin{align*}
    U_t
    &= \Delta U
    -m\chi_0\frac{U^{m-1}}{(1+V)^\beta}\,\nabla U\cdot\nabla V
    +\beta\chi_0\,\frac{U^m}{(1+V)^{\beta+1}}\,\abs{\nabla V}^2  +\chi_0 \frac{U^m}{(1+V)^\beta}\,(U^\gamma-V)+U(1-U^\alpha).
  \end{align*}
  Fix $t$ at which
  $\overline{U}$ is differentiable and choose $x_t\in A^+(t)$. If
  $x_t\in\Omega$, then
  \begin{equation}
    \label{E:envelope-min-1}
    \overline U(t)=U(t,x_t),\quad \nabla U(t,x_t)=0,\quad {\rm and}\quad \Delta
    U(t,x_t)\le 0.
  \end{equation}
  {If $x_t\in\partial\Omega$, choose $\psi\in C^2(\overline\Omega)$
  with $\partial_n\psi=1$ on $\partial\Omega$ and set
  $U_\epsilon\coloneqq U-\epsilon\psi$.
  Since $\partial_nU_\epsilon=-\epsilon<0$, its spatial maximum
  occurs at an interior point $x_{t,\epsilon}$.
  By compactness and the uniform convergence $U_\epsilon\to U$,
  a subsequence satisfies $x_{t,\epsilon}\to\bar x_t\in A^+(t)$.
  At these interior maxima,
  $\nabla U(t,x_{t,\epsilon})=\epsilon\nabla\psi(x_{t,\epsilon})$
  and $\Delta U(t,x_{t,\epsilon})\le\epsilon\Delta\psi(x_{t,\epsilon})$.
  Passing to the limit gives} 
  \begin{equation}
    \label{E:envelope-min-2}
    \overline{U}(t)=U(t,\bar x_t),\quad  \nabla U(t,\bar x_t)=0,\quad \Delta
    U(t,\bar x_t)\le 0.
  \end{equation}
  By ~\eqref{E:rectangle-gradV-glaeser}:
  \begin{equation}
    \label{E:envelope-min-3}
    \chi_0\beta U^m(1+V)^{-\beta-1}|\nabla V|^2
    \le \chi_0\beta U^m(1+V)^{-\beta}C_G(\Omega)
    \left(\overline U^\gamma-\underline U^\gamma\right)
    \le \beta C_G(\Omega)\,\widehat\chi_0\,\overline U^m
    \left(\overline U^\gamma-\underline U^\gamma\right).
  \end{equation}
  Moreover, at any point where $U=\overline U(t)$, \eqref{E:rectangle-V-bounds} gives
  \begin{equation}\label{E:envelope-min-4}
    \chi_0\frac{U^m}{(1+V)^\beta}\,(U^\gamma-V)
    \le \chi_0(1+V)^{-\beta}\,\overline U^m\big(\overline U^\gamma-\underline U^\gamma\big)
    \le \widehat\chi_0\,\overline U^m\big(\overline U^\gamma-\underline U^\gamma\big).
  \end{equation}
  If $\overline U$ is differentiable at $t$, both one-sided limits
  in~\eqref{E:envelope-min} equal $\overline U'(t)$, so
  $\overline U'(t)\ge\max_{A^+(t)}U_t\ge\min_{A^+(t)}U_t\ge\overline U'(t)$, whence
  \[
    \max_{x\in A^+(t)}U_t(t,x)=\min_{x\in A^+(t)}U_t(t,x)=\overline{U}'(t).
  \]
  {Evaluating the equation for $U$ at the maximizing point constructed
  above and using~\eqref{E:envelope-min-1} or~\eqref{E:envelope-min-2},
  together with~\eqref{E:envelope-min-3} and~\eqref{E:envelope-min-4},
  gives the first inequality in~\eqref{E:rectangle-ode-ineq}.} 
  The second inequality in
  \eqref{E:rectangle-ode-ineq} is proved in the same way, with
  $U_\epsilon=U+\epsilon\psi$ and a minimum point $\underline x_t$, where
  $\nabla U(t,\underline x_t)=0$ and $\Delta U(t,\underline x_t)\ge0$. There
  the Laplacian term and the gradient-square term
  $\beta\chi_0U^m(1+V)^{-\beta-1}|\nabla V|^2$ are both nonnegative and are
  discarded; this is why no factor $1+\beta C_G(\Omega)$ appears in the second
  line of~\eqref{E:rectangle-ode-ineq}. The chemotaxis term reverses direction:
  at $\underline x_t$, \eqref{E:rectangle-V-bounds} gives
  $U^\gamma-V=\underline U^\gamma-V\ge\underline U^\gamma-\overline U^\gamma$
  with $\underline U^\gamma-\overline U^\gamma\le0$, so
  $\chi_0U^m(1+V)^{-\beta}(U^\gamma-V)\ge\widehat\chi_0\underline U^m
  (\underline U^\gamma-\overline U^\gamma)$, because $(1+V)^{-\beta}\le\widehat\chi_0/\chi_0$.

  Now following~\cite[Section~3]{galakhov.salieva.ea:16:on}, we compare $\overline{U}$ and $\underline{U}$ with an
  auxiliary ODE system.
  Let $(\overline{u}(t),\underline{u}(t))$ be the unique maximal solution of
  \begin{equation}\label{E:rectangle-ode}
    \begin{dcases}
      \overline{u}_t=\big(1+\beta C_G(\Omega)\big)\widehat\chi_0\, \overline{u}^m(\overline{u}^\gamma-\underline{u}^\gamma)
      +\overline{u}(1-\overline{u}^\alpha),
      & \overline{u}(0)=1\vee \Norm{u_0}_\infty,           \\
      \underline{u}_t=\widehat\chi_0 \underline{u}^m(\underline{u}^\gamma-\overline{u}^\gamma)+\underline{u}(1-\underline{u}^\alpha),
      & \underline{u}(0)=1\wedge \inf_{x\in\Omega}u_0(x),
    \end{dcases}
  \end{equation}
  and denote its maximal existence interval by $[0,t_{\max})$.
  We claim that
  \begin{equation}\label{E:new-rectangle-1}
    (\overline{u}(t),\underline{u}(t))\ \text{exists for all $t>0$},
    \qquad
    \underline{u}(t)\le 1\le \overline{u}(t)\ \text{for all $t\ge 0$},
  \end{equation}
  \begin{equation}\label{E:new-rectangle-2}
    \underline{u}(t)\le \underline{U}(t)\le \overline{U}(t)\le \overline{u}(t)
    \qquad \text{for all $t\in[0,\infty)$,}
  \end{equation}
  and
  \begin{equation}
    \label{E:new-rectangle-3}
    \overline{u}(t)\to 1
    \quad\text{and}\quad
    \underline{u}(t)\to 1
    \qquad \text{as $t\to\infty$.}
  \end{equation}

  {The comparison~\eqref{E:new-rectangle-2} and
  convergence~\eqref{E:new-rectangle-3} imply
  $\Norm{U(t,\cdot)-1}_\infty\to0$ as $t\to\infty$.} 
  The comparison in
  Remark~\ref{R:sensitivity-thresholds}(2) gives
  $\chi_0<\max\{\chi_{\mathrm L,\beta}^{**,3},\chi_{\mathrm L,\beta}^{**,4}\}
  \le\chi_{\mathrm L,\beta}^*$, so Lemma~\ref{L:nonlinear-stability-lm}(2)
  gives~\eqref{E:exponential-decay-eq1}.
  In the following, we prove \eqref{E:new-rectangle-1}, \eqref{E:new-rectangle-2}, and \eqref{E:new-rectangle-3}.

  \smallskip

  First, we prove \eqref{E:new-rectangle-1}.
  Note that $0<\underline{u}(0)\le 1\le \overline{u}(0)$.
  Since $m\ge1$ and $\alpha>0$, each right-hand side in~\eqref{E:rectangle-ode} has the corresponding component as a factor; hence uniqueness implies that both components remain positive on $[0,t_{\max})$.
  Put
  \[
    \tau_*\coloneqq\sup\left\{\tau\in[0,t_{\max}) :
    \underline{u}(s)\le 1\le\overline{u}(s)
    \text{ for every }s\in[0,\tau]\right\}.
  \]
  Note that
  $\underline{u}_t=\widehat\chi_0(1-\overline{u}^\gamma)<0$ if
  $\underline{u}=1<\overline{u}$, while
  $\overline{u}_t=\big(1+\beta C_G(\Omega)\big)\widehat\chi_0(1-\underline{u}^\gamma)>0$ if
  $\underline{u}<1=\overline{u}$; if $\underline u(0)=\overline u(0)=1$, then
  $(\overline u,\underline u)\equiv(1,1)$ by uniqueness. Without loss of generality, we assume that $\underline{u}(0)\not =\overline{u}(0)$. Then $\tau_*>0$.
  Assume that $\tau_*<t_{\max}$. By continuity and the definition of $\tau_*$, $\underline u(\tau_*)\le1\le\overline u(\tau_*)$ with at least one equality, and not both, since $(\overline u,\underline u)$ cannot pass through the equilibrium $(1,1)$ by uniqueness. In the case $\underline u(\tau_*)=1<\overline u(\tau_*)$ we have $\underline u_t(\tau_*)<0$, and in the case $\underline u(\tau_*)<1=\overline u(\tau_*)$ we have $\overline u_t(\tau_*)>0$; in both cases $\underline u\le1\le\overline u$ persists beyond $\tau_*$, contradicting the definition of $\tau_*$. Therefore $\tau_*=t_{\max}$, and \eqref{E:new-rectangle-1} holds for all $t\in [0,t_{\max})$.

  Since $\underline{u}(t)\le 1\le \overline{u}(t)$,
  $\overline{u}^\gamma-\underline{u}^\gamma\le \overline{u}^\gamma$, and since
  $\alpha\ge m+\gamma-1$, for $\overline{u}\ge 1$ we have
  $\overline{u}^{m+\gamma}\le\overline{u}^{1+\alpha}$. Therefore,
  \eqref{E:rectangle-ode} implies
  \[
    \overline{u}_t
    \le \overline{u}-\Big(1-\big(1+\beta C_G(\Omega)\big)\widehat\chi_0\Big)\overline{u}^{1+\alpha}
    \qquad \text{for all $t\in[0,t_{\max})$,}
  \]
  with $1-(1+\beta C_G(\Omega))\widehat\chi_0>0$ because
  $\widehat\chi_0<\chi_{\mathrm L,\beta}^{**,3}$. Hence, $\overline{u}$ stays
  bounded on $[0,t_{\max})$, while $\underline u(t)\in(0,1]$ for $t\in[0,t_{\max})$. Since $m,\gamma\ge1$ and $\alpha>0$, both right-hand sides of~\eqref{E:rectangle-ode} are locally Lipschitz on $[0,\infty)^2$, so the trajectory remains in a compact subset of $[0,\infty)^2$ and the continuation theorem gives $t_{\max}=\infty$.
  \eqref{E:new-rectangle-1} is thus proved.

  Next, we prove \eqref{E:new-rectangle-2}.
  Compare the pair $(\overline{U},\underline{U})$ with
  $(\overline{u},\underline{u})$. Since $\overline{U}$ and $\underline{U}$ are
  locally Lipschitz, they are absolutely continuous and
  \eqref{E:rectangle-ode-ineq} holds in the integrated sense. Fix
  $T>0$.  Set
  \[
    w_1(t)\coloneqq \left(\overline{U}(t)-\overline{u}(t)\right)_+,
    \qquad
    w_2(t)\coloneqq \left(\underline{u}(t)-\underline{U}(t)\right)_+.
  \]
  Then $w_1$ and $w_2$ are absolutely continuous on $[\delta,T]$ for every
  $\delta\in(0,T)$, and $w_1(\delta),w_2(\delta)\to 0$ as $\delta\downarrow 0$
  by the continuity of the solution at $t=0$ and $w_1(0)=w_2(0)=0$.

  Write the
  right-hand sides of~\eqref{E:rectangle-ode} as $F(\overline u,\underline u)$
  and $G(\overline u,\underline u)$, and let $L_T$ be a Lipschitz constant of $F$ and $G$ on the range of the four functions $\overline U,\underline U,\overline u,\underline u$ over $[0,T]$ (finite, since all four are bounded there). Then
  $$\partial_{\overline u}G=-\widehat\chi_0\gamma\,\underline u^{m}
  \overline u^{\gamma-1}\le 0\quad {\rm and}\quad
  \partial_{\underline u}F=-\big(1+\beta C_G(\Omega)\big)\widehat\chi_0\gamma\,
  \overline u^{m}\underline u^{\gamma-1}\le 0.
  $$
  Recall that $(f)_+'=f'\,\mathbf 1_{\{f>0\}}$ a.e.\ for absolutely continuous
  $f$. On $\{w_1>0\}$ we have, by~\eqref{E:rectangle-ode-ineq},
  $$w_1'\le F(\overline U,\underline U)-F(\overline u,\underline u)
  =[F(\overline U,\underline U)-F(\overline u,\underline U)]
  +[F(\overline u,\underline U)-F(\overline u,\underline u)]\quad a.e.\,\, t\in [\delta,T].
  $$
  The first bracket is at most $L_Tw_1$, and the second is nonpositive if $\underline U\ge\underline u$ and at most $L_Tw_2$ otherwise.
  On $\{w_2>0\}$, similarly,
  $$w_2'\le G(\overline u,\underline u)-G(\overline U,\underline U)
  =[G(\overline u,\underline u)-G(\overline u,\underline U)]
  +[G(\overline u,\underline U)-G(\overline U,\underline U)]\quad a.e.\,\, t\in [\delta,T].
  $$
  The first bracket is at most $L_Tw_2$, and the second is nonpositive if $\overline U\le\overline u$ and at most $L_Tw_1$ otherwise.
  Hence, for a.e.\
  $t\in(\delta,T)$,
  \[
    w_1'(t)+w_2'(t)\le 2L_T\left(w_1(t)+w_2(t)\right).
  \]
  By Gr\"onwall's inequality,
  $w_1(t)+w_2(t)\le\big(w_1(\delta)+w_2(\delta)\big)e^{2L_T(T-\delta)}$ on
  $[\delta,T]$; letting $\delta\downarrow 0$ gives $w_1\equiv w_2\equiv 0$ on
  $[0,T]$. Since $T\in(0,\infty)$ is arbitrary, this
  yields~\eqref{E:new-rectangle-2}.

  \smallskip
  Finally, we prove \eqref{E:new-rectangle-3}.
  As in~\cite[Section~3]{galakhov.salieva.ea:16:on}, the identities
  \begin{align*}
    \frac{\overline{u}_t}{\overline{u}}
    & =\big(1+\beta C_G(\Omega)\big)\widehat\chi_0\, \overline{u}^{m-1}(\overline{u}^\gamma-\underline{u}^\gamma)
    +(1-\overline{u}^\alpha),\qquad
    \frac{\underline{u}_t}{\underline{u}}
    =\widehat\chi_0 \underline{u}^{m-1}(\underline{u}^\gamma-\overline{u}^\gamma)+(1-\underline{u}^\alpha),
  \end{align*}
  yield
  \begin{align}\label{E:rectangle-logdiff}
    \frac{d}{dt}\left(\ln\overline{u}-\ln\underline{u}\right)
    & \le \Big(\big(2+\beta C_G(\Omega)\big)\widehat\chi_0-1\Big)\left(\overline{u}^\alpha-\underline{u}^\alpha\right),
  \end{align}
  where we used that $\underline{u}\le 1\le \overline{u}$ on $[0,\infty)$,
  $\underline{u}^{m-1}\le \overline{u}^{m-1}$, and
  \[
    \overline{u}^{m-1}\left(\overline{u}^\gamma-\underline{u}^\gamma\right)
    \le \overline{u}^{m+\gamma-1}-\underline{u}^{m+\gamma-1}
    \le \overline{u}^\alpha-\underline{u}^\alpha,
  \]
  (the last inequality because $\alpha\ge m+\gamma-1$ and
  $\underline u\le1\le\overline u$), so that
  \[
    \Big(\big(1+\beta C_G(\Omega)\big)\overline{u}^{m-1}+\underline{u}^{m-1}\Big)\left(\overline{u}^\gamma-\underline{u}^\gamma\right)
    \le \big(2+\beta C_G(\Omega)\big)\left(\overline{u}^\alpha-\underline{u}^\alpha\right).
  \]
  Put
  $   c_0\coloneqq 1-\big(2+\beta C_G(\Omega)\big)\widehat\chi_0>0.$
  By~\eqref{E:rectangle-logdiff},
  \[
    \frac{d}{dt}\left(\ln\overline{u}-\ln\underline{u}\right)
    \le-c_0\left(\overline{u}^\alpha-\underline{u}^\alpha\right).
  \]
  Hence, $\ln\overline{u}-\ln\underline{u}$ is nonincreasing.
  In particular, $\underline{u}$ stays uniformly positive and $\overline{u}$
  is uniformly bounded on $[0,\infty)$.
  Integrating the last inequality shows that
  $\overline{u}^\alpha-\underline{u}^\alpha\in L^1(0,\infty)$.
  The bounded ODE right-hand sides make this power difference uniformly
  continuous, so it converges to zero.
  Since $\underline{u}\le 1\le\overline{u}$, it follows that
  $\underline{u}(t)\to 1$ and $\overline{u}(t)\to 1$.
  Thus, \eqref{E:new-rectangle-3} holds.
\end{proof}

\section{Global stability in the minimal model
and proof of Theorem~\ref{T:stability-minimal-model}}
\label{S:stability-minimal-model}

In this section, we study the stability of positive constant solutions
of~\eqref{E:main-PE} in the canonical minimal model and prove
Theorem~\ref{T:stability-minimal-model}.

\subsection{Proof of Theorem~\ref{T:stability-minimal-model} under condition (i)}

\begin{proof}[Proof of Theorem~\ref{T:stability-minimal-model} under condition (i)]
  Assume that $m\ge 1$ and $\chi_0 \le 0$. First, we prove that
  any bounded positive solution converges to $(u^*,v^*)$. \smallskip

  Let $(u(t,x),v(t,x))$ be a globally defined bounded positive solution with
  $\int_\Omega u(0,x)dx=|\Omega| u^*$. Let $\bar u(t)=\max_{x\in\bar\Omega}u(t,x).$
  Then $\int_\Omega u(t,x)dx=u^*|\Omega|$ for all $t\ge 0$. Hence $\bar u(t)\ge
  u^*$ for all $t \ge 0$. By Lemma \ref{L:nonincreasing}, there holds $\bar
  u(t_1)\ge \bar u(t_2)$ for all $0\le t_1\le t_2$. We claim that
  \begin{equation}\label{E:new-claim-eq1}
    \lim_{t\to\infty} \bar u(t)=u^*.
  \end{equation}

  Assume by contradiction that \eqref{E:new-claim-eq1} does not hold. Then
  there is $\delta^*>0$ such that $\lim_{t\to\infty} \bar u(t)=u^*+\delta^*$.
  Since $m\ge 1$, Lemma~\ref{L:Holder}(4) applies: there are $t_n\to\infty$
  and a bounded nonnegative pair $(u^\infty,v^\infty)$ such that
  $$u(t+t_n,x)\to u^\infty(t,x)\quad {\rm in}\,\,\, C^{1,2}_{\mathrm{loc}}(\mathbb{R}\times\bar\Omega)\quad {\rm and}\,\,\,
  v(t+t_n,x)\to v^\infty(t,x)\quad {\rm in}\,\,\, C^{0,2}_{\mathrm{loc}}(\mathbb{R}\times\bar\Omega)
  $$ as $n\to\infty$,
  $(u^\infty,v^\infty)$ is a classical solution of~\eqref{E:main-PE}
  defined for all $t\in\mathbb{R}$, and
  $\partial_n u^\infty=\partial_n v^\infty=0$ on $\partial\Omega$.
  By Theorem~\ref{T:persistence}(1), there is $m_1>0$ such that $u^\infty(t,x)\ge m_1$. Moreover, by the monotonicity of $\bar u$ and
  uniform convergence on $\bar\Omega$,
  $$\bar u^\infty(t)\coloneqq\max_{x\in\bar\Omega}u^\infty(t,x)
  =\lim_{n\to\infty}\bar u(t+t_n)=u^*+\delta^*$$
  for every
  $t\in\mathbb{R}$, and by mass conservation and the same convergence,
  \[
    \int_\Omega u^\infty(t,x)dx=|\Omega| u^*
    \quad \forall\, t\in\mathbb{R}.
  \]
  We claim that no such profile can exist. Put $M\coloneqq u^*+\delta^*$
  and $z\coloneqq M-u^\infty$. Then $z\ge 0$ on $\mathbb{R}\times\bar\Omega$
  and, for every $t\in\mathbb{R}$, $z(t,\cdot)$ vanishes at some point of
  $\bar\Omega$ but is not identically zero, because the spatial mean of
  $u^\infty(t,\cdot)$ equals $u^*<M$. By the comparison principle for the
  elliptic equation $-\Delta v^\infty+v^\infty=(u^\infty)^\gamma$, we have
  $v^\infty\le \bar u^\infty(t)^\gamma=M^\gamma$, hence
  \begin{equation}\label{E:minimal-repulsion-Delta-v}
    \Delta v^\infty=v^\infty-(u^\infty)^\gamma\le M^\gamma-(u^\infty)^\gamma
    = g\, z,
    \qquad
    g\coloneqq\frac{M^\gamma-(u^\infty)^\gamma}{M-u^\infty}\in[0,C_g],
  \end{equation}
  where $g$ is understood as $\gamma M^{\gamma-1}$ where $u^\infty=M$, and
  $C_g<\infty$ because $m_1\le u^\infty\le M$ (the lower bound $m_1>0$ is needed when $0<\gamma<1$). Since $\chi_0\le0$, expanding the $u$-equation and using \eqref{E:minimal-repulsion-Delta-v}, we obtain
  \begin{align*}
    & u^\infty_t-\Delta u^\infty
    +\chi_0 m (u^\infty)^{m-1}(1+v^\infty)^{-\beta}\nabla v^\infty\cdot\nabla u^\infty \\
    & \qquad = -\chi_0 (u^\infty)^{m}(1+v^\infty)^{-\beta}\Delta v^\infty
    +\chi_0\beta (u^\infty)^{m}(1+v^\infty)^{-\beta-1}|\nabla v^\infty|^2 \\
    & \qquad \le -\chi_0 (u^\infty)^{m}(1+v^\infty)^{-\beta} g\, z
    \eqqcolon h\, z,
  \end{align*}
  where $0\le h\le |\chi_0| M^m C_g<\infty$ by
  \eqref{E:minimal-repulsion-Delta-v}, and $b\coloneqq -\chi_0 m
  (u^\infty)^{m-1}(1+v^\infty)^{-\beta}\nabla v^\infty$ is bounded on
  $\mathbb{R}\times\bar\Omega$ by the regularity of $(u^\infty,v^\infty)$
  from Lemma~\ref{L:Holder}. Hence $z$ satisfies
  \begin{equation}\label{E:minimal-repulsion-z-eq}
    z_t-\Delta z-b\cdot\nabla z+h\,z\ \ge\ 0
    \quad\text{on }\mathbb{R}\times\Omega,
    \qquad \partial_n z=0 \ \text{on }\mathbb{R}\times\partial\Omega.
  \end{equation}
  By the strong comparison principle for parabolic equations with the Neumann boundary condition, either $z>0$ on $\mathbb R\times\overline\Omega$ or $z\equiv0$; either case contradicts the fact that, for every $t\in\mathbb R$, $z(t,\cdot)$ vanishes at some point of $\overline\Omega$ but is not identically zero. Therefore \eqref{E:new-claim-eq1} holds.

  We claim that
  \begin{equation}\label{E:new-claim-eq2}
    \lim_{t\to\infty}\|u(t,\cdot)-u^*\|_\infty=0.
  \end{equation}
  Assume by contradiction that \eqref{E:new-claim-eq2} does not hold. Then there
  are $\epsilon_0>0$, $t_n\to \infty$ and $\{x_n\}\subset\bar\Omega$ such that
  $|u(t_n,x_n)-u^*|\ge \epsilon_0$. Applying Lemma~\ref{L:Holder}(4)
  again, we may assume that there is
  $(u^\infty(t,x),v^\infty(t,x))$ such that $(u(t+t_n,x),v(t+t_n,x))\to
  (u^\infty(t,x),v^\infty(t,x))$ as $n\to\infty$ locally uniformly in
  $(t,x)\in\mathbb{R}\times\bar\Omega$ and $(u^\infty(t,x),v^\infty(t,x))$ is a
  bounded nonnegative classical solution defined for all $t\in\mathbb{R}$.
  By~\eqref{E:new-claim-eq1}, $\bar u^\infty(t)=u^*$ for all
  $t\in\mathbb{R}$, so $u^\infty(t,x)\le u^*$ everywhere while, by mass
  conservation and uniform convergence on $\bar\Omega$, the spatial mean of
  $u^\infty(t,\cdot)$ equals $u^*$. This implies that
  $u^\infty(t,x)=u^*$ for all $t\in\mathbb{R}$ and $x\in\bar\Omega$, and then
  \[
    \lim_{n\to\infty} |u(t_n,x_n)-u^*|=0,
  \]
  which is a contradiction. Hence, \eqref{E:new-claim-eq2} holds. \smallskip
  By~\eqref{E:new-claim-eq2}, $\|u(t,\cdot)-u^*\|_\infty\to 0$ as
  $t\to\infty$. Since $m\ge 1$, $\int_\Omega u(t,x)dx=|\Omega|u^*$ for all
  $t\ge 0$, and $\chi_0\le 0<\chi_{\mathrm M,\beta}^*(u^*)$,
  Lemma~\ref{L:nonlinear-stability-lm}(2) applies and yields
  \eqref{E:exponential-decay-eq1}.  This proves
  Theorem~\ref{T:stability-minimal-model}(i).
\end{proof}

\subsection{Proof of Theorem~\ref{T:stability-minimal-model} under condition (ii)}

\begin{proof}[Proof of Theorem~\ref{T:stability-minimal-model} under condition (ii)]
  Assume that $m = 1$, $\beta \ge 1$, and $0 < \chi_0 <
  \chi_{\mathrm M,\beta}^{**}(u^*)$. Let $(u(t,x),v(t,x))$ be a globally
  defined bounded positive classical solution of the canonical minimal
  model~\eqref{E:main-PE-M} with $\int_\Omega u(0,x)dx=|\Omega| u^*$.

  Since $0<\chi_0<\chi_{\mathrm M,\beta}^{**}(u^*)$, we can first choose
  $0<\epsilon<\underline v_0(u^*)$ sufficiently small that
  \begin{equation}
    \label{E:new-minimal-eq0}
    \frac{\lambda_1}{2\big(\overline{u}_0(u^*)+\epsilon\big)^2}
    -\frac{\chi_0^2 C_{\gamma}(u^*,\epsilon)}
    {8\left(1+\underline{v}_0(u^*)-\epsilon\right)^{2\beta}}>0,
  \end{equation}
  where
  \begin{equation*}
    C_{\gamma}(u^*,\epsilon) \coloneqq
    \begin{cases}
      (u^*)^{2\gamma-2},                                              & 0<\gamma\le 1,  \\[6pt]
      \gamma^2  \big(\overline{u}_0(u^*)+\epsilon\big)^{2(\gamma-1)}, & \gamma> 1.
    \end{cases}
  \end{equation*}
  Indeed, $\overline{u}_0(u^*)\sqrt{C_\gamma(u^*,0)}=\Gamma_\gamma(u^*)$
  by~\eqref{E:Gamma-gamma-minimal}, so \eqref{E:new-minimal-eq0} holds at
  $\epsilon=0$ by the third entry of~\eqref{E:chi0-star-minimal}, and its
  left-hand side is continuous in $\epsilon$.
  We restart the solution at $t=1$.
  By positivity and continuity, its datum $u(1,\cdot)$
  satisfies~\eqref{E:initial-cond-PE}.
  Applying Theorem~\ref{T:persistence}(3) and
  Lemma~\ref{L:minimal-eventual-bounds}, followed by a further
  time shift, we obtain
  \begin{equation}\label{E:new-minimal}
    v(t,x)\ge \underline{v}_0(u^*)-\epsilon
    \quad \text{and} \quad
    u(t,x)\le \overline{u}_0(u^*)+\epsilon
    \qquad \forall\, t\ge 0,\,\, x\in\Omega.
  \end{equation}

  Next, as in the proof of Theorem~\ref{T:global-stable}, let
  \begin{equation}\label{E:F-def-minimal}
    F(t) \coloneqq \int_\Omega
    \Big(u(t,x)-u^*-u^*\ln\frac{u(t,x)}{u^*}\Big)\,dx.
  \end{equation}
  Differentiating $F$ along the solution (here $m=1$) and integrating by parts, we obtain
  \begin{align}\label{E:global-stab-minimal-1}
    F'(t)
    = \int_\Omega \frac{u(t,x)-u^*}{u(t,x)} u_t(t,x)\,dx
    = -u^*\int_\Omega \frac{|\nabla u|^2}{u^2}\,dx+\chi_0 u^*\int_\Omega \frac{\nabla u\cdot\nabla v}{u(1+v)^\beta}\,dx.
  \end{align}
  Applying Young's inequality and the lower bound for $v$ in~\eqref{E:new-minimal}, we obtain
  \begin{align}\label{E:global-stab-minimal-2}
    F'(t)
    & \le - \frac{u^*}{2} \int_\Omega \frac{|\nabla u|^2}{u^2}
    + \frac{u^*\chi_0^2}{2(1+\underline{v}_0(u^*)-\epsilon)^{2\beta}}\int_\Omega |\nabla v|^2.
  \end{align}
  Subtracting the equilibrium identity $-\Delta v^*+v^*=(u^*)^\gamma$ from the elliptic equation gives $-\Delta(v-v^*)+(v-v^*)=u^\gamma-(u^*)^\gamma$; testing with $v-v^*$ and applying Young's inequality yields
  \begin{align}\label{E:global-stab-minimal-3}
    \int_\Omega  |\nabla v|^2
    \le \frac14\int_\Omega \left(u^\gamma-(u^*)^\gamma\right)^2.
  \end{align}
  Since $\int_\Omega (u(t,x)-u^*)\,dx=0$ for all $t\ge 0$, the Neumann
  Poincar\'e inequality and \eqref{E:new-minimal} imply that
  \begin{align}\label{E:global-stab-minimal-4}
    \int_\Omega \frac{|\nabla u|^2}{u^2}
    \ge \frac{\lambda_1}{\big(\overline{u}_0(u^*)+\epsilon\big)^2}\int_\Omega (u-u^*)^2.
  \end{align}
  Note that for $0 < \gamma \le 1$, we have $0 \le
  \frac{u^\gamma-(u^*)^\gamma}{u-u^*} \le (u^*)^{\gamma-1}$ for $u \neq u^*$,
  and hence
  \begin{align*}
    \left(u^\gamma-(u^*)^\gamma\right)^2\le (u^*)^{2\gamma-2}(u-u^*)^2,\qquad\text{for all }u\ge 0.
  \end{align*}
  On the other hand, the mass identity and \eqref{E:new-minimal} give
  $u^*\le\overline{u}_0(u^*)+\epsilon$.
  Thus, for $\gamma>1$, the mean value theorem gives
  \[
    \left|u^\gamma-(u^*)^\gamma\right|
    \le \gamma \big(\overline{u}_0(u^*)+\epsilon\big)^{\gamma-1}  \left|u-u^*\right|
    \qquad \text{for $0\le u\le\overline{u}_0(u^*)+\epsilon$}.
  \]
  Therefore, for all $t\ge 0$, we have
  \begin{equation}\label{E:Cgammau*}
    \int_\Omega \left(u^\gamma-(u^*)^\gamma\right)^2\le C_{\gamma}(u^*,\epsilon)\int_\Omega (u-u^*)^2.
  \end{equation}
  Combining~\eqref{E:global-stab-minimal-2}, \eqref{E:global-stab-minimal-3},
  \eqref{E:global-stab-minimal-4}, and \eqref{E:Cgammau*}, we obtain, for
  all $t\ge 0$,
  \begin{align*}
    F'(t)
    & \le -\Big[\frac{\lambda_1 u^*}{2\big(\overline{u}_0(u^*)+\epsilon\big)^2}
    -\frac{C_{\gamma}(u^*,\epsilon)u^*\chi_0^2}
    {8\left(1+\underline{v}_0(u^*)-\epsilon\right)^{2\beta}}\Big]
    \int_\Omega (u-u^*)^2.
  \end{align*}
  Defining $K_0 \coloneqq \frac{\lambda_1}{2\big(\overline{u}_0(u^*)+\epsilon\big)^2}$
  and $K_1 \coloneqq
  \frac{\chi_0^2 C_{\gamma}(u^*,\epsilon)}
  {8\left(1+\underline{v}_0(u^*)-\epsilon\right)^{2\beta}}$, we
  obtain
  \begin{equation}\label{E:global-stab-minimal-6}
    F'(t)\le -(K_0 - K_1)u^*\int_\Omega (u-u^*)^2,\qquad \forall\, t>0.
  \end{equation}
  By \eqref{E:new-minimal-eq0}, $C\coloneqq (K_0 - K_1)u^*> 0$. Since $F(t)\ge 0$
  for all $t\ge 0$, integrating~\eqref{E:global-stab-minimal-6} over $[0,t]$
  yields
  \[
    C \int_{0}^t\int_\Omega (u(s,x)-u^*)^2\,dx\,ds
    \le F(0)-F(t)\le F(0)\quad\forall\, t\ge 0.
  \]
  Hence,
  \begin{equation}\label{E:global-stab-minimal-7}
    \int_{0}^\infty\int_\Omega (u(s,x)-u^*)^2\,dx\,ds < \infty.
  \end{equation}
  Since $(u,v)$ is bounded, Lemma~\ref{L:Holder}(1) makes $t\mapsto\int_\Omega(u(t,x)-u^*)^2\,dx$ uniformly continuous on $[1,\infty)$. This, together with~\eqref{E:global-stab-minimal-7}, implies
  \[
    \int_\Omega(u(t,x)-u^*)^2\,dx\longrightarrow0.
  \]
  Since $\int_\Omega(u-u^*)^2\,dx=\int_\Omega(u-u^*)\big(u-u^*\big)\,dx\to0$,
  Lemma~\ref{L:nonlinear-stability-lm}(1) with $\theta=1$ gives
  $\Norm{u(t,\cdot)-u^*}_\infty\to0$, which is the global asymptotic
  stability asserted in Definition~\ref{D:stability}. Recall that the
  solution conserves mass, $\int_\Omega u(t,x)\,dx=|\Omega|u^*$, that $m=1$,
  and that $\chi_0<\chi_{\mathrm M,\beta}^{**}(u^*)\le\chi_{\mathrm
  M,\beta}^*(u^*)$ by Remark~\ref{R:sensitivity-thresholds}(2); hence
  Lemma~\ref{L:nonlinear-stability-lm}(2) applies and
  yields~\eqref{E:exponential-decay-eq1} with a rate $\lambda$ independent of
  the solution. Finally, since $\chi_{\mathrm M,\beta}^{**}(u^*)\le
  \chi_\beta/2<\chi_\beta$, $m=1$, and $\beta\ge 1$,
  Part~I~\cite[Theorem~1.2(2)]{chen.ruau.ea:26:chemotaxis*1} gives global
  existence and boundedness of the positive solution for every $u_0$
  satisfying~\eqref{E:initial-cond-PE} with $\int_\Omega u_0(x)\,dx=|\Omega|
  u^*$. This proves Theorem~\ref{T:stability-minimal-model}(ii).
\end{proof}

\appendix

\section{Basic properties of classical solutions}\label{S:Prelim}

%
{For the two canonical models, we establish regularity, compactness,
and eventual bounds not proved in
Part~I~\cite{chen.ruau.ea:26:chemotaxis*1}.
We retain the solution notation and initial-data assumptions introduced
above and set}
\[
  \bar u(t) \coloneqq \sup_{x\in\Omega}u(t,x),\,\,
  \bar v(t)\coloneqq \sup_{x\in\Omega} v(t,x),\,\,\, {\rm and}\,\,
  \underline{u}(t)\coloneqq\inf_{x\in\Omega}u(t,x),\,\,
  \underline{v}(t)\coloneqq\inf_{x\in\Omega} v(t,x).
\]
Note that we always have that $\underline{u}^\gamma(t)\le \underline{v}(t)\le\bar v(t)\le \bar u^\gamma(t)$.

{In the following, $w\in C^{\theta_1,\theta_2}([1,\infty)\times\overline\Omega)$
means that its $C^{\theta_1,\theta_2}$ norms on
$[t,t+1]\times\overline\Omega$ are uniformly bounded for $t\ge1$.}
The following lemma gives uniform regularity and compactness of time translates for
classical solutions of~\eqref{E:main-PE} and will be used in the proofs of
Theorems~\ref{T:persistence}, \ref{T:global-stable},
and~\ref{T:stability-minimal-model}.

\begin{lemma}[Uniform regularity and compactness of time translates]\label{L:Holder}
  Suppose that $u_0$ satisfies~\eqref{E:initial-cond-PE} and $(u(t,x),v(t,x))$
  is a globally defined (i.e., $T_{\max} = \infty$), bounded, positive classical
  solution of~\eqref{E:main-PE}. Then we have the following regularity
  properties:
  \begin{enumerate}

    \item For all $\theta \in (0,1)$, $u(\cdot,\cdot) \in
    C^{\theta/2,\:\theta}\left([1,\infty)\times\overline{\Omega}\right)$.

    \item Put $q\coloneqq 1\wedge\gamma$. For every $\theta\in(0,1)$,
    $  \sup_{t\ge 1}\Norm{v(t,\cdot)}_{C^{2+q\theta}(\overline\Omega)}<\infty$.
    Moreover, for every $\delta\in(0,q\theta)$ there exists $C_\delta>0$
    such that
    \[
      \Norm{v(t,\cdot)-v(s,\cdot)}_{C^{2+\delta}(\overline\Omega)}
      \le C_\delta\abs{t-s}^{(q\theta-\delta)/2}
      \quad {\rm
      whenever}\,\,\, s,t\ge 1\,\,\, {\rm  and}\,\,\,  \abs{t-s}\le 1.
    \]

    \item[(3)] Assume $m\ge 1$.
    There is $\eta\in(0,1)$ such that
    $  u\in C^{1+\eta/2,2+\eta}
    \left([1,\infty)\times\overline\Omega\right).$

    \item[(4)] Assume $m\ge 1$. For every sequence $t_n\to\infty$, define
    \[
      u_n(t,x)\coloneqq u(t_n+t,x),\qquad
      v_n(t,x)\coloneqq v(t_n+t,x),\qquad t>-t_n.
    \]
    There are a subsequence $t_{n_k}$ and a bounded nonnegative pair
    $(u^\infty,v^\infty):\mathbb R\times\overline\Omega\to[0,\infty)^2$ such
    that
    \[
      u_{n_k}\to u^\infty
      \quad\text{in }C_{\mathrm{loc}}^{1,2}(\mathbb R\times\overline\Omega),
      \qquad
      v_{n_k}\to v^\infty
      \quad\text{in }C_{\mathrm{loc}}^{0,2}(\mathbb R\times\overline\Omega).
    \]
    Moreover, $(u^\infty,v^\infty)$ is an entire classical solution of the
    same canonical model as $(u,v)$, and it satisfies $\partial_nu^\infty=\partial_nv^\infty=0$ on $\mathbb R\times\partial\Omega$.

  \end{enumerate}
\end{lemma}

\begin{proof}
  Part~(1) follows from the variation-of-constants formula and the analytic
  Neumann heat-semigroup estimates, using the uniform boundedness of $u$, $v$,
  and $\nabla v$ on positive-time strips; see~\cite{henry:81:geometric,
  pazy:83:semigroups}.

  For part~(2), the map $r\mapsto r^\gamma$ is Lipschitz on bounded intervals
  when $\gamma\ge 1$ and is $\gamma$-H\"older when $0<\gamma<1$.
  Hence, $u^\gamma\in C^{q\theta/2,q\theta}$.
  Elliptic Schauder estimates for $-\Delta v+v=u^\gamma$ give the
  uniform spatial estimate; see~\cite{gilbarg.trudinger:01:elliptic}.
  H\"older interpolation also yields
  \[
    \Norm{u^\gamma(t)-u^\gamma(s)}_{C^\delta}
    \le C\Norm{u^\gamma(t)-u^\gamma(s)}_\infty^{1-\delta/(q\theta)}
    \le C\abs{t-s}^{(q\theta-\delta)/2}.
  \]
  Applying the elliptic Schauder estimate to $v(t)-v(s)$ proves the asserted
  time estimate.

  To prove part~(3), expand the parabolic equation as
  \[
    u_t-\Delta u-B(t,x)\cdot\nabla u=f(t,x),
  \]
  where
  $   B=-m\chi_0\frac{u^{m-1}}{(1+v)^\beta}\nabla v$
  and
  \[
    f=\chi_0\beta\frac{u^m}{(1+v)^{\beta+1}}\abs{\nabla v}^2
    -\chi_0\frac{u^m}{(1+v)^\beta}\Delta v
    +\begin{cases}
      u(1-u^\alpha), & \text{in the canonical logistic model}, \\
      0,              & \text{in the minimal model}.
    \end{cases}
  \]
  Note that the function $u^{m-1}$ is H\"older continuous, while
  $\Delta v=v-u^\gamma$.
  Thus, $B$ and $f$  are uniformly H\"older continuous.
  Part (3) then follows from
  boundary parabolic Schauder estimates; see,
  e.g.,~\cite[Ch.~IV]{ladyzenskaja.solonnikov.ea:68:linear}
  or~\cite{lieberman:96:second}.

  To prove part~(4), fix a sequence $t_n\to\infty$. Given a compact interval
  $  I=[-T,T]\subset\mathbb R$, for all large $n$ we have $t+t_n\ge1$ for every
  $t\in I$. Choose $\theta\in(0,1)$ and
  $\delta\in(0,(1\wedge\gamma)\theta)$, and let $\eta\in(0,1)$ be the
  exponent supplied by part~(3), decreased if necessary so that
  $\eta<(1\wedge\gamma)\theta$ for the signal estimates.
  Parts~(2)--(3) imply that
  \[
    \{u_n\}_n \text{ is bounded in }
    C^{1+\eta/2,2+\eta}(I\times\overline\Omega)\quad {\rm for\,\,\, some}\,\,\, \eta\in (0,1),
  \]
  and the families $\{D_x^\zeta v_n:\abs{\zeta}\le 2\}$ are uniformly bounded
  and equicontinuous on $I\times\overline\Omega$.
  The compact embeddings, the Arzel\`a--Ascoli theorem, and a standard diagonal
  extraction over $I=[-T,T]$, $T=1,2,\dots$, give a subsequence (not relabeled)
  and $(u^\infty,v^\infty)$ such that
  \[
    u_n\to u^\infty \text{ in } C^{\,1,\;2}_{\mathrm{loc}}(\mathbb{R}\times\overline\Omega) \quad \text{and} \quad
    v_n\to v^\infty \text{ in } C^{\,0,\;2}_{\mathrm{loc}}(\mathbb{R}\times\overline\Omega).
  \]
  In particular, the following convergences hold locally uniformly on
  $\mathbb R\times\overline\Omega$:
  \[
    u_n, \partial_tu_n, \nabla u_n, D_x^2u_n
    \longrightarrow
    u^\infty, \partial_tu^\infty, \nabla u^\infty, D_x^2u^\infty,
    \quad\text{and}\quad
    v_n, \nabla v_n, D_x^2v_n
    \longrightarrow
    v^\infty, \nabla v^\infty, D_x^2v^\infty.
  \]
  For $m\ge1$, the coefficients in the expansion
  \[
    \nabla\cdot\Big(\frac{u^m}{(1+v)^\beta}\nabla v\Big)
    =m\frac{u^{m-1}}{(1+v)^\beta}\nabla u\cdot\nabla v
    -\beta\frac{u^m}{(1+v)^{\beta+1}}\abs{\nabla v}^2
    +\frac{u^m}{(1+v)^\beta}\Delta v
  \]
  are continuous even where the nonnegative limit $u^\infty$ vanishes.
  We can therefore pass to the limit in the drift and the reaction
  $u-u^{1+\alpha}$, or zero, in the parabolic equation.
  We also pass to the limit in the elliptic equation $-\Delta v_n+v_n=u_n^\gamma$.
  Thus, $(u^\infty,v^\infty)$ solves the same canonical model as $(u,v)$.
  Finally, convergence of both gradients on $\overline\Omega$ preserves
  the two homogeneous Neumann boundary conditions in the limit.
  Boundedness
  and nonnegativity follow from the corresponding properties of the
  translates. This proves part~(4).
\end{proof}


{We will use the following lemma in the proofs of Theorems~\ref{T:persistence}
and~\ref{T:stability-minimal-model}.}

\begin{lemma}[Eventual upper bound in the minimal
  model]\label{L:minimal-eventual-bounds}
  For the canonical minimal model, assume that $m=1$, $\beta\ge 1$, and $0<\chi_0
  <\min\left\{\frac{\chi_\beta}{2},\chi_\beta^{1/2}\right\}$, where $\chi_\beta$ is
  defined in~\eqref{E:Theta-chi-beta}. Then there exist constants $C_N>0$ and
  $0<q'_N\le 1\le q''_N$, depending only on $(\Omega,N,\gamma)$, such that for every
  $u^*>0$ and any globally
  defined classical solution $(u(t,x),v(t,x))$ with $u(0,\cdot)$
  satisfying~\eqref{E:initial-cond-PE} and $\int_\Omega u(0,x)\,dx=u^*|\Omega|$,
  there exists $T_*>0$ such that
  \begin{equation}\label{E:minimal-eventual-Linfty}
    u(t,x)\le C_N\left(\left(u^*\right)^{q'_N}+\left(u^*\right)^{q''_N}\right)
    \qquad \text{for all $t\ge T_*$ and all $x\in\Omega$.}
  \end{equation}
\end{lemma}

\begin{proof}
  {Set $p_0\coloneqq\max\left\{\frac32,\frac34\gamma N\right\}$
  and $p_N\coloneqq2+\max\left\{N,\gamma N,p_0/\gamma\right\}$.}
  \smallskip

  \noindent\textbf{Step 1 (Eventual $L^{p_0}$ bound).} In this step, we show that
  $\int_\Omega u^{p_0}(t,x)dx$ stays bounded.

  \smallskip

  Since
  $\chi_0<\frac{\chi_\beta}{2}$, we have
  $\frac{2\beta-1}{\chi_0}>\max\{2,\gamma N\}$. Then $p_0<\max\{2,\gamma
  N\}<\frac{2\beta-1}{\chi_0}$. Note that
  \begin{align}\label{E:new-lp-eq1}
    \frac{1}{p_0}\frac{d}{dt}\int_\Omega u^{p_0}
    & \le - \frac{p_0-1}{2} \int_\Omega u^{p_0-2}|\nabla u|^2
    + \frac{(p_0-1)\chi_0^2}{2} \int_\Omega \frac{u^{p_0}|\nabla v|^2}{(1+v)^{2\beta}}.
  \end{align}
  By the derivation preceding \cite[inequality~(4.11)]{chen.ruau.ea:26:chemotaxis*1}, we have
  \begin{align}\label{E:new-lp-eq2}
    & \Big(\frac{(p_0-1)\chi_0{ (2\beta-1)}}{p_0}-\frac{(p_0-1)\chi_0^2}{2}\Big) \int_\Omega \frac{u^{p_0}}{(1+v)^{2\beta}}|\nabla v|^2\nonumber \\
    & \le \frac{p_0-1}{2} \int_\Omega u^{p_0-2}|\nabla u|^2  +\frac{(p_0-1)\Theta_{2(\beta-1)}\chi_0}{p_0} \int_\Omega u^{p_0}.
  \end{align}
  This implies that
  \begin{align*}
    & \frac{(p_0-1)\chi_0^2}{2}\int_\Omega \frac{u^{p_0}|\nabla v|^2}{(1+v)^{2\beta}}                                                                                  \\
    & = \Big(\frac{(p_0-1)\chi_0{ (2\beta-1)}}{p_0}-\frac{(p_0-1)\chi_0^2}{2} + \frac{p_0-1}{p_0}\chi_0^2 \big(p_0-\frac{2\beta-1}{\chi_0}\big)\Big) \int_\Omega
    \frac{u^{p_0}}{(1+v)^{2\beta}}|\nabla v|^2                                                                                                                          \\
    & \le \Big(1-\frac{2(2\beta -1)-2p_0\chi_0}{2(2\beta-1)-p_0\chi_0}\Big)
    \Big(\frac{p_0-1}{2} \int_\Omega u^{p_0-2}|\nabla u|^2  +\frac{(p_0-1)\Theta_{2(\beta-1)}\chi_0}{p_0} \int_\Omega u^{p_0}\Big).
  \end{align*}
  It then follows from~\eqref{E:new-lp-eq1} that
  \begin{align}\label{E:new-lp-eq3}
    \frac{1}{p_0}\frac{d}{dt}\int_\Omega u^{p_0}
    & \le  -\frac{\left((2\beta -1)-p_0\chi_0\right)(p_0-1)}{2(2\beta-1)-p_0\chi_0}
    \int_\Omega u^{p_0-2}|\nabla u|^2  +   \frac{(p_0-1)\Theta_{2(\beta-1)}\chi_0}{p_0} \int_\Omega u^{p_0}.
  \end{align}
  Note that $\frac{x-p_0}{2x-p_0}$ is an increasing function for $x \in
  [p_0,\infty)$, and $(2\beta-1)/\chi_0>\max\{2,\gamma N\}>p_0$. Hence,
  \[
    \frac{(2\beta-1)/\chi_0-p_0}{2(2\beta-1)/\chi_0-p_0}\ge \frac{\max\{2,\gamma N\}-p_0}{2\max\{2,\gamma N\}-p_0}=\frac{1}{5}.
  \]
  Moreover, since $\chi_0<\chi_\beta^{1/2}$ and
  \[
    \Theta_{\kappa}
    =\frac{\kappa^\kappa}{(1+\kappa)^{1+\kappa}}
    =\Big(\frac{\kappa}{1+\kappa}\Big)^{\kappa}\frac{1}{1+\kappa}
    \le \frac{1}{1+\kappa}
    \qquad \text{for all $\kappa\ge 0$,}
  \]
  the quantity $\chi_0\Theta_{2(\beta-1)}$ is bounded by $1$ for all $\beta \ge
  1$. Since $p_0\ge \frac{3}{2}$, we have $p_0(p_0-1)\ge \frac{3}{4}$.
  Multiplying \eqref{E:new-lp-eq3} by $p_0$ and using the above estimate, we
  infer that
  \begin{equation}\label{E:new-lp-eq4}
    \frac{d}{dt}\int_\Omega u^{p_0}
    \le -\frac{3}{20} \int_\Omega u^{p_0-2}|\nabla u|^2  + \left(\max\left\{2,\gamma N\right\} -1\right)\int_\Omega u^{p_0}.
  \end{equation}
  Applying Lemma 4.1 in Part~I~\cite{chen.ruau.ea:26:chemotaxis*1} with $\varepsilon
  = \frac{3}{20}\left(\max \{2, \gamma N\} \right)^{-1}$, we obtain a constant
  $C_{\gamma,N}$, depending only on $(\Omega,N,\gamma)$ and hence independent
  of $\chi_0$ and $\beta$, such that
  \begin{equation}\label{E:new-lp-eq5}
    \frac{d}{dt}\int_\Omega u^{p_0}
    \le -\int_\Omega u^{p_0} +C_{\gamma,N} \left(\int_\Omega u(t,x)dx\right)^{p_0}.
  \end{equation}
  Therefore, there exist $C'_{\gamma,N}>0$ and $T_0>0$ such that for all $t\ge T_0$,
  \begin{equation}\label{E:new-lp-eq6}
    \int_\Omega u^{p_0}(t,x)dx \le C'_{\gamma,N}\left(|\Omega| u^*\right)^{p_0}.
  \end{equation}

  \noindent\textbf{Step 2 (Eventual $L^{2 p_N}$ and $L^{\gamma p_N}$ bounds).}
  In this step, we show
  that $\int_\Omega u^{2p_N}(t,x)dx$ and $\int_\Omega u^{\gamma p_N}(t,x)dx$ stay bounded.

  \smallskip First, note that for any $p>1$, we have
  \begin{equation}\label{E:new-cross-eq2}
    \chi_0^2\int_\Omega \frac{u^p|\nabla v|^2}{(1+v)^{2\beta}}
    \le \frac{1}{2}\int_\Omega u^{p+\gamma}
    + \chi_0^{\frac{2(p+\gamma)}{\gamma}}\frac{\gamma}{p+\gamma}\left(\frac{2p}{p+\gamma}\right)^{p/\gamma}
    \int_\Omega \frac{|\nabla v|^{\frac{2(p+\gamma)}{\gamma}}}{(1+v)^{\frac{2\beta(p+\gamma)}{\gamma}}}.
  \end{equation}
  By Proposition~2.2 in Part~I \cite{chen.ruau.ea:26:chemotaxis*1}, there is
  $C_N''>0$, independent of $\chi_0$ and $\beta$, such that
  \[
    \int_\Omega \frac{|\nabla v|^{2(p+\gamma)/\gamma }}{(1+v)^{2\beta (p+\gamma)/\gamma }}
    = \int_\Omega \frac{|\nabla v|^{2(p+\gamma)/\gamma}}{(1+v)^{(1+(2\beta-1))(p+\gamma)/\gamma }}
    \le \Theta_{2\beta-1}^{(p+\gamma)/\gamma}\, C_N'' \int_\Omega u^{p+\gamma}.
  \]
  Since $\chi_0<\chi_\beta^{1/2}$ and $\Theta_{2\beta-1}\le \frac{1}{2\beta}\le
  \frac{1}{2}$ for $\beta\ge 1$, we have $\chi_0^2 \Theta_{2\beta-1} \le 1$.
  Hence,
  \[
    \chi_0^{\frac{2(p+\gamma)}{\gamma}}\Theta_{2\beta-1}^{\frac{p+\gamma}{\gamma}}
    =\left(\chi_0^2\Theta_{2\beta-1}\right)^{\frac{p+\gamma}{\gamma}}
    \le 1.
  \]
  Combining these estimates with \eqref{E:new-lp-eq1} with $p_0$ replaced by
  $p$, we infer that there is $C_{N,p}'''>0$, independent of $\chi_0$ and
  $\beta$, such that
  \begin{align}\label{E:new-lp-eq7}
    \frac{1}{p}\frac{d}{dt}\int_\Omega u^{p}
    & \le -\frac{1}{2}(p-1)\int_\Omega u^{p-2}|\nabla u|^2
    + C_{N,p}'''\int_\Omega u^{p+\gamma}.
  \end{align}

  Next, put $R_N\coloneqq \max\{2p_N,\gamma p_N\}$.
  We keep track of the constants in the iteration.
  Define a finite sequence by
  $r_0\coloneqq p_0$ and $r_{j+1}\coloneqq
  \min\left\{\frac{3}{2}r_j,R_N\right\}$, and choose $J\in\mathbb{N}$ so that
  $r_J=R_N$. We claim that for every $j\in\{0,\dots,J\}$ there exist constants
  $A_j>0$ and $Q_j\ge r_j$, depending only on $(\Omega,N,\gamma)$, and a
  time $\tau_j\ge T_0$ such that
  \begin{equation}\label{E:new-lp-eq8-aux}
    \sup_{t\ge \tau_j}\int_\Omega u^{r_j}(t,x)\,dx
    \le A_j\big( (u^*)^{p_0}+ (u^*)^{Q_j}\big).
  \end{equation}
  We prove \eqref{E:new-lp-eq8-aux} by induction. First, for $j=0$,
  \eqref{E:new-lp-eq8-aux} follows from \eqref{E:new-lp-eq6} with
  $\tau_0\coloneqq T_0$ and $Q_0=p_0$. Next, assume that
  \eqref{E:new-lp-eq8-aux} holds for some $j<J$, and set $r\coloneqq r_j$ and
  $p\coloneqq r_{j+1}\in(r,2r)$. Writing $w=u^{p/2}$, \eqref{E:new-lp-eq7}
  implies
  \begin{equation}
    \label{E:new-eq1}
    \frac{1}{p}\frac{d}{dt}\int_\Omega w^2
    + \frac {p-1}{p^2}\int_\Omega |\nabla w|^2 + \int_\Omega w^2
    \le \int_\Omega w^2
    + C_{N,p}'''\int_\Omega w^{2+\frac{2\gamma}{p}}.
  \end{equation}
  Set $q\coloneqq \frac{2r}{p}$ and $\ell\coloneqq 2+\frac{2\gamma}{p}$. Note
  that $\ell>2>q$.
  By H\"older's interpolation inequality, we have
  \[
    \Norm{w}_2^2
    \le \Norm{w}_q^{2\theta_0}\Norm{w}_\ell^{2(1-\theta_0)}
    = \left(\int_\Omega w^q\right)^{\frac{2\theta_0}{q}}
    \left(\int_\Omega w^\ell\right)^{\frac{2(1-\theta_0)}{\ell}},
    \quad \text{where}\quad
    \theta_0=\frac{q(\ell-2)}{2(\ell-q)}.
  \]
  Then by Young's inequality, for every $\varepsilon>0$ there exists
  $C_\varepsilon>0$ such that
  \[
    \int_\Omega w^2
    \le \varepsilon\int_\Omega w^\ell
    + C_\varepsilon \int_\Omega w^q.
  \]
  Taking $\varepsilon=1$ and applying the Gagliardo--Nirenberg inequality,
  we obtain $C_{N,p,\gamma}>0$ such that
  \[
    \left(C'''_{N,p}+1\right)\Norm{w}_\ell^\ell
    \le C_{N,p,\gamma}\Norm{\nabla w}_2^{\theta \ell}
    \Norm{w}_q^{(1-\theta)\ell}
    + C_{N,p,\gamma}\Norm{w}_q^\ell, \quad \text{where}\quad
    \theta=\frac{2N(\ell-q)}{\ell(2N+2q-Nq)}.
  \]
  Indeed, $\frac{\theta \ell}{2}<1$ is equivalent to $\gamma N<2r$, which
  follows from $r\ge p_0\ge\frac{3}{4}\gamma N$. Hence, Young's inequality
  yields
  \[
    C_{N,p,\gamma}\Norm{\nabla w}_2^{\theta \ell}\Norm{w}_q^{(1-\theta)\ell}
    \le \frac{p-1}{2p^2}\int_\Omega |\nabla w|^2
    + \widetilde C_{N,p,\gamma}\left(\int_\Omega w^q\right)^{\sigma_j},
  \]
  where $\widetilde C_{N,p,\gamma}>0$ depends only on
  $(\Omega,N,\gamma)$, and
  \[
    \sigma_j
    \coloneqq \frac{2(1-\theta)\ell}{q(2-\theta \ell)}
    =\frac{2N+2\ell-\ell N}{2N+2q-\ell N}.
  \]
  Hence, by the interpolation estimate and after enlarging
  $\widetilde C_{N,p,\gamma}$ if necessary,
  \[
    \int_\Omega w^2 + C'''_{N,p}\int_\Omega w^\ell
    \le \frac{p-1}{2p^2}\int_\Omega |\nabla w|^2
    + \widetilde C_{N,p,\gamma}\left[
    \int_\Omega w^q
    + \left(\int_\Omega w^q\right)^{\sigma_j}
    + \left(\int_\Omega w^q\right)^{\ell/q}
    \right].
  \]
  We then infer that there exists $\widetilde C_j>0$, depending only on
  $(\Omega,N,\gamma)$, such that
  \begin{equation*}
    \frac{1}{p}\frac{d}{dt}\int_\Omega w^2
    + \frac{p-1}{2p^2}\int_\Omega |\nabla w|^2
    + \int_\Omega w^2
    \le \widetilde C_j\left[
    \int_\Omega w^q
    + \left(\int_\Omega w^q\right)^{\sigma_j}
    + \left(\int_\Omega w^q\right)^{\ell/q}
    \right]
  \end{equation*}
  for all $t\ge \tau_j$. Since $w^q=u^r$, the induction hypothesis implies that
  there exist exponents $Q_j\ge r_j$, depending only on $(N,\gamma)$,
  such that
  \[
    \int_\Omega w^q
    + \left(\int_\Omega w^q\right)^{\sigma_j}
    + \left(\int_\Omega w^q\right)^{\ell/q}
    \le \widetilde C_j\left((u^*)^{p_0}
    +(u^*)^{Q_j \max\{\sigma_j,\ell/q\}}\right)
    \qquad \text{for all $t\ge \tau_j$.}
  \]
  Then
  \[
    \frac{d}{dt}\int_\Omega u^p + p\int_\Omega u^p
    \le p\widetilde C_j\left((u^*)^{p_0}+(u^*)^{Q_{j+1}}\right)
    \qquad \text{for all $t\ge \tau_j$,}
  \]
  where $Q_{j+1}=Q_j\max\{\sigma_j,\ell/q\}$.
  Notice that $Q_{j+1}\ge p(=r_{j+1})$.
  Then by Gr\"onwall's inequality, there are
  $A_{j+1}>0$, depending only on $(\Omega,N,\gamma)$, and a time
  $\tau_{j+1}\ge \tau_j$ such that \eqref{E:new-lp-eq8-aux} holds with $j+1$
  in place of $j$. This proves the claim.

  Now $p_0<\min\{2p_N,\gamma p_N\}\le R_N$. For
  $s\in\{2p_N,\gamma p_N\}$, the pointwise inequality
  $u^s\le u^{p_0}+u^{R_N}$ allows us to use the bounds at the first and last
  steps of the induction.
  Hence, there exist $\widetilde C_N>0$, depending only on
  $(\Omega,N,\gamma)$, $\tilde q_N\ge R_N$, and some
  $\widetilde T_0\ge T_0$ such that, for each
  $s\in\{2p_N,\gamma p_N\}$,
  \begin{equation}\label{E:new-lp-eq8}
    \sup_{t\ge \widetilde T_0}\int_\Omega u^s(t,x)\,dx
    \le \widetilde C_N\left(\left(u^*\right)^{p_0}
    +\left(u^*\right)^{\widetilde q_N}\right).
  \end{equation}
  Replacing $T_0$ by $\widetilde T_0$, we may assume in the following
  that~\eqref{E:new-lp-eq8} holds for all $t\ge T_0$. \smallskip

  \noindent\textbf{Step 3 (Eventual $L^\infty$ bound).} First, set $p\coloneqq
  p_N$ and let $A_p=-\Delta+I$ on $L^p(\Omega)$ with Neumann boundary
  conditions; write $X_p^\sigma\coloneqq\mathcal D(A_p^\sigma)$.
  Fix $t\ge T_0+1$ and write the variation-of-constants formula on $[t-1,t]$:
  \begin{align}\label{E:new-lp-eq9}
    u(t,\cdot)
    = e^{-A_p}u(t-1,\cdot)
    - \chi_0 \int_{t-1}^t e^{-A_p(t-s)}
    \nabla\cdot\!\left(\frac{u(s,\cdot)}{(1+v(s,\cdot))^\beta}\nabla v(s,\cdot)\right)\,ds
    + \int_{t-1}^t e^{-A_p(t-s)}u(s,\cdot)\,ds.
  \end{align}
  Since $p=p_N>N$, choose $N/(2p)<\sigma<1/2$.
  The fractional-domain embedding
  \cite[Theorem~1.6.1]{henry:81:geometric} gives
  $X_p^\sigma\hookrightarrow C(\overline\Omega)$.
  For $0<\tau\le1$, the analytic-semigroup estimate
  \cite[Theorem~1.4.3]{henry:81:geometric} and the Neumann divergence estimate
  \cite[Lemma~1.3(iv)]{winkler:10:aggregation} give
  \[
    \Norm{A_p^\sigma e^{-A_p\tau}f}_p
    \le C\tau^{-\sigma}\Norm{f}_p
  \]
  and, by splitting the semigroup at $\tau/2$,
  \[
    \Norm{A_p^\sigma e^{-A_p\tau}\nabla\cdot\phi}_p
    \le C\tau^{-\sigma}\left(1+\tau^{-1/2}\right)
    \Norm{\phi}_{L^p(\Omega;\mathbb R^N)}.
  \]
  Both kernels are integrable on $(0,1)$ because $\sigma<1/2$.
  Applying these estimates and the embedding to~\eqref{E:new-lp-eq9} gives a
  constant $C_\infty>0$, depending only on $(\Omega,N,\gamma)$, such that for
  all $t\ge T_0+1$,
  \begin{align*}
    \Norm{u(t,\cdot)}_\infty
    \le C_\infty\Big(
    \sup_{s\ge T_0}\Norm{u(s,\cdot)}_p
    + \chi_0\sup_{s\ge T_0}\Big\|{\frac{u(s,\cdot)\abs{\nabla v(s,\cdot)}}{(1+v(s,\cdot))^\beta}}\Big\|_p
    \Big).
  \end{align*}
  By H\"older's inequality,
  \begin{align*}
    \Big \|{\frac{u(s,\cdot)\abs{\nabla v(s,\cdot)}}{(1+v(s,\cdot))^\beta}}\Big\|_p
    \le \|{u(s,\cdot)}\|_{2p}\Big\|{\frac{\nabla v(s,\cdot)}{(1+v(s,\cdot))^\beta}}\Big\|_{2p}
    \qquad \text{for all $s\ge 0$.}
  \end{align*}

  Next, apply Part~I~\cite[Proposition~2.2]{chen.ruau.ea:26:chemotaxis*1}
  with $p=p_N$ and the canonical
  elliptic coefficients. There exists a fixed $\widehat C_N>0$, depending only
  on $(\Omega,N,\gamma)$, such that
  \[
    \int_\Omega \frac{\abs{\nabla v(s,\cdot)}^{2p}}{(1+v(s,\cdot))^{(1+\beta)p}}
    \le \widehat C_N\,\Theta_\beta^{p}\int_\Omega u^{\gamma p}(s,x)\,dx
    \qquad \text{for all $s\ge T_0$.}
  \]
  Since $\beta\ge 1$, we have $2\beta p\ge (1+\beta)p$ and therefore
  \[
    \Big\|{\frac{\nabla v(s,\cdot)}{(1+v(s,\cdot))^\beta}}\Big\|_{2p}
    \le \widehat  C_N^{\frac{1}{2p}}\Theta_\beta^{\frac{1}{2}}
    \Big(\int_\Omega u^{\gamma p}(s,x)\,dx\Big)^{\frac{1}{2p}}
    \qquad \text{for all $s\ge T_0$.}
  \]
  Combining this estimate with \eqref{E:new-lp-eq8} and H\"older's inequality,
  we obtain a constant $\widehat C_N>0$, depending only on
  $(\Omega,N,\gamma)$, such that
  \[
    \sup_{s\ge T_0}\Norm{\frac{u(s,\cdot)\abs{\nabla v(s,\cdot)}}{(1+v(s,\cdot))^\beta}}_p
    \le \widehat C_N\,\Theta_\beta^{\frac{1}{2}}
    \left(\left(u^*\right)^{p_0/p_N}+\left(u^*\right)^{\widetilde q_N/p_N}\right).
  \]
  Similarly, by \eqref{E:new-lp-eq8} and H\"older's inequality,
  \[
    \sup_{s\ge T_0}\Norm{u(s,\cdot)}_p
    \le \abs{\Omega}^{\frac{1}{2p_N}}
    \Big(\sup_{s\ge T_0}\int_\Omega u^{2p_N}\Big)^{\!\frac{1}{2p_N}}
    \le \widehat C_N
    \Big(\big(u^*\big)^{p_0/(2p_N)}
    +\big(u^*\big)^{\widetilde q_N/(2p_N)}\Big),
  \]
  where the factor $\abs{\Omega}^{1/(2p_N)}$ is absorbed into the next
  constant $\widehat C_N$.

  Finally, since $\Theta_\beta
  =\frac{1}{1+\beta}\Big(\frac{\beta}{1+\beta}\Big)^{\beta} $ for all $\beta>0$,
  the assumption $\chi_0<\chi_\beta^{1/2}$ yields
  \[
    \chi_0\Theta_\beta^{1/2}
    \le \Big(\frac{2}{\max\{2,\gamma N\}}\frac{2\beta-1}{1+\beta}\big(\frac{\beta}{1+\beta}\big)^\beta\Big)^{1/2}
    \le \Big(\frac{4}{\max\{2,\gamma N\}}\Big)^{1/2}
  \]
  which is independent of $\chi_0$ and $\beta$. Combining the two estimates
  above and noting that $p_0/(2p_N)<p_0/p_N<\widetilde q_N/p_N$, we conclude
  that there exists $C_N>0$, depending only on $(\Omega,N,\gamma)$, such
  that
  \[
    \Norm{u(t,\cdot)}_\infty
    \le C_N\left(\left(u^*\right)^{q'_N}+\left(u^*\right)^{q''_N}\right)
    \qquad \text{for all $t\ge T_0+1$,}
  \]
  where $q'_N\coloneqq \frac{p_0}{2p_N}<1$ and $q''_N\coloneqq \frac{\widetilde
  q_N}{p_N}\ge 1$. This proves \eqref{E:minimal-eventual-Linfty} with
  $T_*\coloneqq T_0+1$.
\end{proof}

\section*{Acknowledgments}
\addcontentsline{toc}{section}{Acknowledgments} Both L.~C. and I.~R. were
partially supported by NSF grants DMS-2246850 and DMS-2443823. L.~C. was also
partially supported by a Collaboration Grant for Mathematicians (\#959981) from
the Simons Foundation.
L.~C. and I.~R. also acknowledge support from OpenAI through its ChatGPT for Academic Researchers program.

\subsection*{AI declaration}

During the preparation of this work, the authors used OpenAI's GPT models to assist with proofreading, formatting, verifying mathematical cross-references, and reviewing mathematical reasoning.
After using these tools, the authors reviewed, verified, and edited the content as needed and take full responsibility for the content of the article.

\end{document}